\documentclass{article}
\usepackage[utf8]{inputenc}
\usepackage{amsmath, amsthm, amsfonts, amssymb}
\usepackage{mathtools}
\usepackage{braket}
\usepackage{hyperref}
\usepackage{todonotes}
\usepackage[shortlabels]{enumitem}
\usepackage{tikz-cd}
\usetikzlibrary{decorations.markings}
\usepackage{tikz}
\usepackage{lmodern,bm,bbm,mathrsfs}
\usepackage{float}
\usepackage[
backend=biber,
url=false,
style=alphabetic,
maxnames=4,
minnames=3,
maxbibnames=99,
giveninits,
uniquename=init
]{biblatex}
\bibliography{bibliography}
\DeclareFieldFormat{postnote}{#1}
\DeclareFieldFormat{multipostnote}{#1}
\DeclareFieldFormat{pages}{#1}

\newtheorem{nbc}{PLACEHOLDER}[section]
\newtheorem{theorem}[nbc]{Theorem}
\newtheorem{proposition}[nbc]{Proposition}
\newtheorem{lemma}[nbc]{Lemma}
\newtheorem{corollary}[nbc]{Corollary}

\newtheorem*{conjecture*}{Conjecture}
\newtheorem*{theorem*}{Theorem}

\theoremstyle{definition}
\newtheorem{definition}[nbc]{Definition}
\newtheorem{remark}[nbc]{Remark}

\newtheorem{example}[nbc]{Example}

\newcommand{\C}{\mathbb{C}}

\newcommand{\Z}{\mathbb{Z}}
\newcommand{\N}{\mathbb{N}}

\newcommand{\Eb}{\mathbb{E}}
\newcommand{\Lb}{\mathbb{L}}
\newcommand{\Tb}{\mathbb{T}}

\newcommand{\Cr}{\mathrm{C}}
\newcommand{\Nr}{\mathrm{N}}

\newcommand{\Ac}{\mathcal{A}}

\newcommand{\Ec}{\mathcal{E}}
\newcommand{\Fc}{\mathcal{F}}
\newcommand{\Gc}{\mathcal{G}}
\newcommand{\Hc}{\mathcal{H}}
\newcommand{\Ic}{\mathcal{I}}
\newcommand{\Kc}{\mathcal{K}}

\newcommand{\Nc}{\mathcal{N}}
\newcommand{\Oc}{\mathcal{O}}

\newcommand{\Tc}{\mathcal{T}}

\newcommand{\Xc}{\mathcal{X}}
\newcommand{\Yc}{\mathcal{Y}}
\newcommand{\Zc}{\mathcal{Z}}

\newcommand{\Cf}{\mathfrak{C}}
\newcommand{\Mf}{\mathfrak{M}}

\newcommand{\es}{\mathsf{e}}
\newcommand{\As}{\mathsf{A}}
\newcommand{\Fs}{\mathsf{F}}
\newcommand{\Ss}{\mathsf{S}}
\newcommand{\Ts}{\mathsf{T}}
\newcommand{\Zs}{\mathsf{Z}}

\newcommand{\indgffct}{\mathbf{e}}
\newcommand{\indgfset}{E}
\newcommand{\indgfel}{e}

\newcommand{\vir}{\mathrm{vir}}
\newcommand{\num}{\mathrm{num}}
\newcommand{\cok}{\mathrm{cok}}
\newcommand{\col}{\mathrm{col}}
\newcommand{\homc}{\Hc om}

\newcommand{\CoCh}{\mathrm{CoCh}}
\newcommand{\Coh}{\mathrm{Coh}}
\newcommand{\Db}{\mathrm{D}^b}
\newcommand{\Dm}{\mathrm{D}^-}
\newcommand{\DT}{\mathrm{DT}}
\newcommand{\exc}{\mathrm{exc}}
\newcommand{\hilb}{\mathrm{Hilb}}
\newcommand{\HC}{\mathrm{HC}}
\newcommand{\Hom}{\mathrm{Hom}}
\newcommand{\im}{\mathrm{im}}

\newcommand{\ind}{\mathsf{index}}
\newcommand{\loc}{\mathrm{loc}}
\newcommand{\PT}{\mathrm{PT}}
\newcommand{\pr}{\mathrm{pr}}
\newcommand{\pt}{\mathrm{pt}}
\newcommand{\qcoh}{\mathrm{QCoh}}
\newcommand{\slg}{\mathrm{SL}}

\newcommand{\sym}{\mathrm{Sym}}

\newcommand{\id}{\mathrm{id}}
\newcommand{\syst}{\mathbf{Syst}}
\newcommand{\cTor}{\Tc \mathit{or}}

\newcommand{\BG}[1]{\mathrm{B}{#1}}
\newcommand{\dv}[1]{\mathbf{#1}}
\newcommand{\restr}[2]{\left.{#1}\right\vert_{#2}}
\renewcommand{\vec}[1]{\bm{#1}}
\newcommand{\pexp}[1]{\mathrm{PExp}\left(#1\right)}
\newcommand{\pexpI}[1]{\mathrm{PExp}^I\left(#1\right)}
\newcommand{\abs}[1]{\left| #1 \right|}

\DeclareMathOperator{\rk}{rk}

\makeatletter
\newcommand{\extp}{\@ifnextchar^\@extp{\@extp^{\,}}}
\def\@extp^#1{\mathop{\bigwedge\nolimits^{\!#1}}}

\tikzcdset{
  open/.code     = {\tikzcdset{hook, circled};},
  closed/.code   = {\tikzcdset{hook, slashed};},
  open'/.code    = {\tikzcdset{hook', circled};},
  closed'/.code  = {\tikzcdset{hook', slashed};},
  circled/.code  = {\tikzcdset{markwith = {\draw (0,0) circle (.375ex);}};},
  slashed/.code  = {\tikzcdset{markwith = {\draw[-] (-.4ex,-.4ex) -- (.4ex,.4ex);}};},
  markwith/.code ={
    \pgfutil@ifundefined%
    {tikz@library@decorations.markings@loaded}%
    {\pgfutil@packageerror{tikz-cd}{You need to say %
      \string\usetikzlibrary{decorations.markings} to use arrows with markings}{}}{}%
    \pgfkeysalso{/tikz/postaction = {
      /tikz/decorate,
      /tikz/decoration={markings, mark = at position 0.5 with {#1}}}
    }
  },
}

\makeatother

\title{K-Theoretic Donaldson--Thomas Theory of $[\mathbb{C}^2/\mu_r]\times\mathbb{C}$ and Factorization}
\author{Felix Thimm}

\begin{document}

\maketitle

\begin{abstract}
We compute the equivariant K-theoretic Donaldson--Thomas invariants of $[\C^2/\mu_r]\times \C$ using factorization and rigidity techniques. For this, we develop a generalization of Okounkov's factorization technique that applies to Hilbert schemes of points on orbifolds. We show that the (twisted) virtual structure sheaves of Hilbert schemes of points on orbifolds satisfy the desired factorization property. We prove that the generating series of Euler characteristics of such factorizable systems are the plethystic exponential of a simpler generating series. For $[\C^2/\mu_r]\times \C$, the computation is then completed by a rigidity argument, involving an equivariant modification of Young's combinatorial computation of the corresponding numerical Donaldson--Thomas invariants.
\end{abstract}

\tableofcontents

\section{Introduction}

\subsection{Orbifold Donaldson--Thomas Theory}

Donaldson--Thomas (DT) theory, originally formulated in \cite{thomas2001holomorphic}, studies generating series of the degrees of virtual classes of Hilbert schemes of curves and points on a given Calabi--Yau 3-fold. Similarly, one can study the same generating series for a Calabi--Yau orbifold of dimension three.

We focus here on the invariants for Hilbert schemes of points. For a given class $\alpha$ in the numerical Grothendieck-group $\Nr_0(\Xc)$ of $0$-dimensional sheaves on $\Xc$, \cite{os08_quot_dm} define a Hilbert scheme
\begin{equation*}
    \hilb^\alpha(\Xc)
\end{equation*}
of substacks $\Zc\subset \Xc$ of class $[\Oc_\Zc]=\alpha$. This Hilbert scheme can be equipped with a perfect obstruction theory by describing it as a moduli space of ideal sheaves. Using \cite{bf97_intrinsic_normal_cone}, this defines a virtual cycle and a virtual structure sheaf $\Oc^\vir_{\hilb^\alpha(\Xc)}\in K\left(\hilb^\alpha(\Xc)\right)$. After twisting by a square root $\Kc_\vir^{\frac{1}{2}}$ of the virtual canonical bundle, we obtain the twisted virtual structure sheaf
\begin{equation*}
    \widehat{\Oc}^\vir_{\hilb^\alpha(\Xc)}\coloneqq \Oc^\vir_{\hilb^\alpha(\Xc)}\otimes \Kc_\vir^{\frac{1}{2}}.
\end{equation*}
This twist was introduced and motivated in \cite{nekrasov2014membranes}. We study DT generating series
\begin{equation}
    \Zs(\Xc) := \sum_{\alpha \in \Nr_0(\Xc)} q^{\alpha} \chi\left(\hilb^\alpha(\Xc), \widehat{\Oc}^\vir\right).
\end{equation}
Here we take $\Xc$ to be either projective or toric. In the latter case, the Euler characteristic is defined via localization.

\subsubsection{Main Result}
Our main theorem is a computation of the equivariant K-theoretic DT generating series $\Zs\left([\C^3/\mu_r],q_0,\dots,q_r\right)$.
\begin{theorem*}[Theorem \ref{rig:thm:main_thm}]
    Let $\Ts=(\C^*)^3$, acting on $\C^3$ with weights $(t_1,t_2,t_3)$. The $\Ts$-equivariant K-theoretic degree $0$ DT generating series for $[\C^3/\mu_r]$, with $\mu_r$ acting on $\C^3$ with weight $(1,r-1,0)$, is
    \begin{equation*}
        \Zs\left([\C^3/\mu_r],q_0,\dots,q_{r-1}\right) = \pexp{\Fs_r(q)+\Fs_r^\col(q_0,\dots,q_{r-1})},
    \end{equation*}
    where
    \begin{align*}
        \Fs(q) &\coloneqq \frac{[t_2t_3][t_1t_3][t_1t_2]}{[t_1][t_2][t_3]}\frac{1}{[\kappa^{1/2}q][\kappa^{1/2}q^{-1}]},\\
        \Fs_r(q) &\coloneqq \sum_{k=0}^{r-1} \Fs\left(t_1^{r-k}t_2^{-k},t_1^{-r+k+1}t_2^{k+1},t_3,q\right),\\
        \Fs_r^\col(q) &\coloneqq \frac{[t_1t_2]}{[t_3]}\frac{1}{[\kappa^{1/2}q][\kappa^{1/2}q^{-1}]}\left(\sum_{0< i \leq j < r} \left(q_{[i,j]}+q_{[i,j]}^{-1}\right)\right),
    \end{align*}
    where $\kappa=t_1t_2t_3$, $q=q_0\cdots q_{r-1}$, $q_{[i,j]}=q_i\cdots q_j$, $[w]=w^{1/2}-w^{-1/2}$, and $\pexp{-}$ denotes the plethystic exponential, whose precise definition we give in Section \ref{fact:sec:pexp}.
\end{theorem*}
This formula was conjectured in \cite[Conj. 2]{Cirafici_2022}. In \cite[Cor. 6.2]{Cao_2023} it was shown to follow from corresponding conjectural formula for 4-folds. We follow here the notation in \cite[Cor. 6.2]{Cao_2023}. Note that we don't require a sign in front of our variable $q_0$ as our chosen twist for the twisted virtual structure sheaf includes the sign $(-1)^{n_0}$.

\subsubsection{Colored Plane Partitions}
The generating series we study are (equivariant) K-theoretic refinements of more classical numerical Donaldson--Thomas invariants
\begin{equation}
    \Zs^\num(\Xc) := \sum_{\alpha \in \Nr_0(\Xc)} q^{\alpha} \chi\left(\hilb^\alpha(\Xc), \nu\right),
\end{equation}
defined using Euler characteristics weighted by the Behrend function $\nu$. If $\Xc=[\C^3/G]$ for some finite abelian subgroup $G\subset \slg(3)$, these generating series turn out to be signed counts of colored plane partitions, as seen in \cite[Appendix A]{young10}.

Plane partitions are finite subsets $\pi$ of $\Z_{\geq 0}^3$, such that if $(i+1,j,k)$, $(i,j+1,k)$ or $(i,j,k+1)$ are contained in $\pi$, then so is $(i,j,k)$. By labeling each box $(i,j,k)$ by the monomial $x^iy^jz^k$, these correspond to $\Ts$-fixed $0$-dimensional subschemes of $\C^3$, which must be cut out by monomial ideals in $\C[x,y,z]$. The action of $G$ on $\C^3$ corresponds to an action on $\C[x,y,z]$ whose weight spaces are spanned by $\C x^iy^jz^k$. The boxes $(i,j,k)$ of the plane partition are then colored by the characters of $G$ of the corresponding $\C x^iy^jz^k$.

For $G=\mu_r$ acting with weights $(1,-1,0)$ \cite{young10} computes the generating series to be
\begin{equation*}
    \Zs^\num\left([\C^3/\mu_r],-q_0,q_1,\dots,q_{r-1}\right) = \pexp{\frac{-1}{[q][q^{-1}]}\left(r+\sum_{0<i\leq j<r}\left(q_{[i,j]}+q_{[i,j]}^{-1}\right)\right)},
\end{equation*}
where $q=q_0\cdots q_{r-1}$, $q_{[i,j]}=q_i\cdots q_j$, $[w]=w^{1/2}-w^{-1/2}$, and $\pexp{-}$ denotes the plethystic exponential, whose precise definition we give in Section \ref{fact:sec:pexp}. As explained in \cite[Appendix A]{young10}, the sign in front of the formal variable $q_0$ come from relating the combinatorial count of colored plane partitions to the Behrend-function weighted count. Our main computation, Theorem \ref{rig:thm:main_thm}, refines that result to equivariant K-theoretic DT invariants.

\subsubsection{Orbifold Crepant Resolution Conjecture}

An important conjecture in orbifold DT theory is the crepant resolution conjecture \cite{bryan2010orbifold}. Given a $3$-dimensional Calabi--Yau orbifold $\Xc$, take a crepant resolution
\begin{equation*}
    \begin{tikzcd}[column sep=small]
        \Xc\ar[dr,"\pi"] & & Y\ar[dl,"\nu"]\\
        & X. &
    \end{tikzcd}
\end{equation*}
This exists using \cite{bridgeland2000mukai}, who also give an equivalence between the derived categories of coherent sheaves on $\Xc$ and $Y$. This lets us identify the numerical Grothendieck groups $\Nr_c(\Xc)$ and $\Nr_c(Y)$. However, this identification does not respect the filtration by dimension of supports. For $\Xc$ satisfying the Hard Lefschetz condition, the crepant resolution conjecture identifies a version of the DT generating series of $\Xc$ with the one of $Y$. We state here the conjecture for the DT invariants of Hilbert schemes of points on $\Xc$
\begin{conjecture*}[{{\cite[Conj. 2]{bryan2010orbifold}}}]
    \begin{equation*}
        \Zs^\DT_0(\Xc) = \Zs^\DT_0(Y)\Zs^\PT_\exc(Y)\widetilde{\Zs}^\PT_\exc(Y).
    \end{equation*}
\end{conjecture*}
Here $\Zs^\PT_\exc(Y)$ is the PT generating series of curves in $Y$ which are contracted to points in the singular locus of $X$, and $\widetilde{\Zs}^\PT_\exc(Y)$ is related to it by a change of variables.

For numerical DT-invariants the above crepant resolution conjecture has been shown using wall-crossing in \cite{calabrese2014crepant}. The curve-counting version was shown in the case of toric orbifolds with transverse $A$-singularities in \cite{ross_toric_ocr}, and for general orbifolds using wall-crossing in \cite{beentjes2018proof}. It is natural to expect refined versions of this conjecture to hold as well.

Our main computation gives the equivariant K-theoretic count for the left-hand side in the case $\Xc=[\C^3/\mu_r]$. Moreover, we introduce a general setup of factorization techniques to study DT invariants of Hilbert schemes of points on orbifolds, which allows us to show in Proposition \ref{pp:prop:general_open_locus_form} that $\Zs^\DT_0(Y)$ divides $\Zs^\DT_0(\Xc)$ in a controlled way for any $\Xc$.

\subsection{Strategy}

We use two main tools in computing the generating series of our main result: A version for orbifolds of the factorization techniques used in \cite{ok15} to prove Nekrasov's formula, and a semi-equivariant extension of the argument used in \cite{young10}, which allows us to compute a limit of the generating series in the equivariant parameters. By the rigidity principle, introduced in \cite{ok15}, this determines our generating series.

\subsubsection{Factorization for Orbifolds}
In \cite{ok15} and \cite{kr}, factorization for schemes is used in the following way. Take (twisted) virtual structure sheaves, defined using quiver-theoretic descriptions of Hilbert schemes of points on $X=\C^3$. It is shown that their pushforwards to $\sym^n(X)$ along the Hilbert-Chow morphism satisfy a factorization property. That roughly means the following. Take the open locus 
\begin{equation*}
    \begin{tikzcd}
        U\ar[r,open] & \sym^{n_1}(X)\times \sym^{n_2}(X)\ar[r] & \sym^{n_1+n_2}(X)
    \end{tikzcd}
\end{equation*}
where the two collections of points are disjoint from each other. Then on $U$ we have 
\begin{equation*}
    \HC_*\widehat{\Oc}_{n_1}^\vir\boxtimes\HC_*\widehat{\Oc}_{n_2}^\vir\cong \HC_*\widehat{\Oc}_{n_1+n_2}^\vir 
\end{equation*}
with certain compatibilities under consecutive splittings. Under this condition, it can be shown that the generating series of Euler characteristics of the virtual structure sheaves is the plethystic exponential of a generating series of Euler characteristics of certain classes $\Gc_n$ on $X$
\begin{equation*}
    \Zs(X) = 1+\sum_{n>0} q^{n} \chi\left(\hilb^{n}(X),\widehat{\Oc}^{\vir}_{\hilb^{n}(X)}\right) = \pexp{\sum_{n>0}q^{n} \chi(X,\Gc_{n})}.
\end{equation*}

In Section \ref{fact:sec:fact} we develop a generalization of the notion of factorizable systems, which applies in the orbifold setting. One of the difficulties is to figure out a suitable replacement for the number of points and the symmetric product. For orbifolds, we use $0$-dimensional effective numerical K-theory classes $\alpha\in \Nr_0(\Xc)$ instead of $n\in \N$. Fixing a coarse moduli space of the orbifold $\pi:\Xc\to X$, we use the symmetric product $\sym^{\pi_*(\alpha)}(X)$, where $\pi_*(\alpha)$ can be identified with an integer. We find a factorization property for virtual structure sheaves pushed forward along the morphisms
\begin{equation*}
    \hilb^{\alpha}(\Xc) \to \sym^{\pi_*(\alpha)}(X).
\end{equation*}
Importantly, the factorization property is only satisfied with respect to splittings of the $0$-dimensional K-theory class $\alpha$ and not its image $\pi_*(\alpha)$ in $X$.

A K-theory class of a substack of an orbifold can potentially be written as a sum of effective K-theory classes, which are not K-theory classes of substacks. This is also apparent when thinking of colored (plane) partitions, where we may have the following situation. The set of boxes of a colored (plane) partition may potentially be partitioned into two subsets in such a way that the count of colored boxes in (at least) one of the subsets can not itself be obtained from a colored (plane) partition, see for example Figure \ref{fig:factorization_index_set}. So we have to additionally restrict to classes in a factorization index set
\begin{equation*}
    I := \left\{\alpha\ |\ \hilb^{\alpha}(\Xc)\neq \emptyset\right\},
\end{equation*}
so that we only allow splittings of K-theory classes into K-theory classes which also come from the Hilbert scheme of points. This is not necessary in the scheme case.

A version of our main theorem about such factorizable system is the following, which allows us to compute generating series of their Euler characteristics in a simple way.

\begin{theorem*}[Corollary \ref{fact:cor:comp_pexp_fact}]
    Let $I$ be a factorization index semigroup. For a factorizable system $\Fc_{\alpha}$ on $\sym^{b(\alpha)}(X)$, there exist classes $[\Gc_{\alpha}]$ on $X$ such that
    \begin{equation*}
        1 + \sum_{\alpha\in I} q^{\alpha} \chi\left(\sym^{b(\alpha)}(X),\Fc_{\alpha}\right) = \pexp{\sum_{\alpha\in I}q^{\alpha} \chi(X,\Gc_{\alpha})}.
    \end{equation*}
\end{theorem*}

The construction of the classes $\Gc_\alpha$ involves tracking possible sequences of splittings of $\alpha$ into various parts. The resulting classes $\Gc_\alpha$ in the plethystic exponential are classes on $X$, the coarse space of the orbifold $\Xc$. One additional feature in the orbifold case is that the construction of these classes allows us to restrict the support of many of these classes to the complement of the non-stacky locus in $X$. As this is often much smaller, the possible poles of these functions computed by localization can be restricted.

In Section \ref{sect:factorization_application} we prove that (twisted) virtual structure sheaves are factorizable and obtain a strong result about the generating series of Euler characteristics 
\begin{equation*}
    \Zs(\Xc) = 1+\sum_{\alpha\in I} q^{\alpha} \chi\left(\hilb^{\alpha}(\Xc),\widehat{\Oc}^{\vir}_{\hilb^{\alpha}(\Xc)}\right) = \pexp{\sum_{\alpha\in I}q^{\alpha} \chi(X,\Gc_{\alpha})}
\end{equation*}
in Corollary \ref{fact:cor:comp_pexp_fact}. Note here a technical requirement that $I$ is closed under addition, which is satisified in the examples we consider.

Finally, we develop a technique of compatible factorizations, which allows us to compare the classes $\Gc_\alpha$ constructed using different factorizable sequences, which are compatible along an embedding in a suitable sense. This yields relations between the generating series for $\Xc$ and its crepant resolution in Proposition \ref{pp:prop:general_open_locus_form}.

\subsubsection{Limit-Equivariant Slicing}
As in \cite{ok15}, we use the rigidity principle to show that, except for certain fixed factors, which we compute, the functions in the plethystic exponential depend only on $\kappa=t_1t_2t_3$ and not on the individual $t_i$. This means it suffices to compute any limit
\begin{equation*}
    \overrightarrow{\Zs}\left([\C^3/\mu_r],\dv{q}\right) \coloneqq \lim_{t_i^{\pm 1}\to \infty}\Zs\left([\C^3/\mu_r],\dv{q}\right)
\end{equation*}
in the parameters $t_i$ that keeps $\kappa$ fixed. \cite{ok15} suggests a particular limit of this kind
\begin{equation}\label{intro:eq:limit}
    t_1,t_3\to 0,\ \lvert t_1 \rvert\ll\lvert t_3\rvert,\ \kappa\ \mathrm{fixed},
\end{equation}
which admits a very simple formula in Proposition \ref{ppc:partition_limit_contribution} for the limit contributions of each plane partition, computed using \cite[Appendix A]{nekrasov2014membranes}.

In \cite{young10}, Young uses a slicing argument, where he decomposes a colored plane partition into monochrome slices, which are partitions. Then, working in a vector space $\big(\extp^{\infty/2}\big)_0V$, which has an orthonormal basis given by all partitions $\ket{\lambda}$, he uses operators
\begin{equation*}
    \Gamma_\pm(x) = \mathrm{exp}\left(\sum_l \frac{x^l}{l}\alpha_{\pm l}\right)
\end{equation*}
which sum over all possible next or previous slices within a plane partition. To compute the desired count of colored plane partitions he introduces weight operators
\begin{equation*}
    Q_i\ket{\lambda} = q_i^{\lvert\lambda\rvert}\ket{\lambda},
\end{equation*}
which multiply by the formal variables $q_i$. Computing commutators of these operators allows the full computation of the generating series
\begin{equation*}
    \Zs^{\mathrm{num}}\left([\C^3/\mu_r],\dv{q}\right) = \bra{\phi}\cdots \bar{A}_+(1)\bar{A}_+(1)\bar{A}_+(1)\bar{A}_-(1)\bar{A}_-(1)\bar{A}_-(1)\cdots\ket{\phi},
\end{equation*} 
where $\ket{\phi}$ is the empty partition and $\bar{A}_\pm(x)\coloneqq \Gamma_\pm(x) Q_{r-1} \cdots \Gamma_\pm(x) Q_1 \Gamma_\pm(x) Q_0$.

A fully equivariant version of this slicing argument is not known, because it is not clear that the equivariant weight of each colored plane partition in the generating series can be computed from certain weights of each slice. However, using the particular limit \eqref{intro:eq:limit}, we can compute the limit contributions of each colored plane partition from simple weights on each slice in Proposition \ref{ppc:partition_limit_contribution}. So, we introduce a limit weight operator
\begin{equation*}
    K_\pm\ket{\lambda} = \left(\kappa^{1/2}\right)^{\pm\lvert\lambda\rvert}\ket{\lambda}.
\end{equation*}
into Young's slicing argument, which allow us to use the argument to compute the desired limit $\overrightarrow{\Zs}\left([\C^3/\mu_r],\dv{q}\right)$ of the generating series in Section \ref{rig:sec:eq_lim_comp} to complete the proof of our main theorem.

\subsection{Future Directions}

\subsubsection{Extension to $\mu_2\times \mu_2$}

A similar formula was proven in the unrefined case for $[\C^3/\mu_2\times \mu_2]$ in \cite{young10}, where $G=\mu_2\times \mu_2$ acts as the group of diagonal matrices with determinant $1$ that square to the identity. Refined formulas were computed in \cite[Cor. 6.2]{Cao_2023}, again assuming a corresponding conjectural formula for 4-folds. It would be interesting to adapt the techniques of this paper to also compute that case. The main additional challenge is to adapt the limit-slicing argument to this case, which seems to be non-trivial.

\subsubsection{Calabi--Yau 4-folds}
Donaldson--Thomas invariants for Calabi--Yau 4-folds have recently been introduced by \cite{Borisov_2017} and \cite{oh2022counting}. Factorization techniques have been used by \cite{kr} to compute the equivariant K-theoretic DT invariants of Hilbert schemes of points on $\C^4$ confirming a conjecture by Nekrasov and Piazzalunga. They use a combination of factorization techniques and localization, and finally relate a specialization of their generating series to the DT generating series of Hilbert schemes of points on $\C^3$, which is given by Nekrasov's formula, proven in \cite{ok15}.

We are working on extending the computation here to the case of $[\C^4/\mu_r]$, with $\mu_r$ acting on the first two coordinates. In fact, following the argument in \cite{kr}, our main theorem must be used in place of Nekrasov's formula for an orbifold computation.

\subsection{Acknowledgements}
I would like to thank my supervisor Jørgen Rennemo for many helpful discussions about this project, and for sharing and explaining an early draft of \cite{kr}, which inspired parts of this work. I would also like to thank Martijn Kool, Nick Kuhn, and Reinier Schmiermann for helpful discussions related to this project.
\section{Setup}

\subsection{Orbifolds}

\begin{definition}
    We fix some conventions. Throughout, we write \textbf{orbifold} for a smooth separated finite type DM-stack with generically trivial stabilizers over $\C$. We work with various types of orbifolds. We consider an orbifold $\Xc$ to be
    \begin{itemize}
        \item \textbf{(quasi-)projective} if it is (quasi-)projective in the sense of \cite[Def. 5.5]{kresch_geom_dm}. In particular, it has a (quasi-)projective coarse moduli space $\pi:\Xc\to X$,
        \item equipped with a $\Ts$\textbf{-action} if it comes equipped with an action by a connected reductive algebraic group $\Ts$, which will often be a torus,
        \item \textbf{Calabi--Yau} if $\Kc_\Xc\cong \Oc_\Xc$. In the $\Ts$-equivariant case, we write $\kappa$ for the $\Ts$-weight of $\Kc_\Xc$, such that $\Kc_\Xc\cong \kappa \Oc_\Xc$ as $\Ts$-equivariant sheaves.
        \item \textbf{toric} if $\Xc$ is a smooth separated toric DM-stack in the sense of \cite{Borisov_2004, fantechi2009smooth}. In this case, there is a $\Ts=(\C^*)^3$-action on $\Xc$.
    \end{itemize}

    In the $\Ts$-equivariant and the toric cases, we will always assume that the stack of fixed points $\Xc^\Ts$ is non-empty. If $\Xc$ is Calabi--Yau, that makes $\kappa$ uniquely defined.
\end{definition}

In the toric case, we have the following well-known local description, which allows us to understand the $\Ts$-fixed points in the Hilbert scheme of points and simplify computations.

\begin{lemma}\label{setup:lemma:toric_loc_str}
    A $d$-dimensional toric orbifold $\Xc$ with $\Xc^\Ts\neq \emptyset$ is locally isomorphic to $[\C^d/G]$ for some finite abelian diagonally embedded subgroup $G\subset \slg(d)$.
\end{lemma}
\begin{proof}
    Let $\Xc$ be a $d$-dimensional toric orbifold with associated stacky fan $\Sigma$. Then
    \begin{equation*}
        \Xc^\Ts\neq \emptyset\ \Leftrightarrow\ \Sigma\text{ has a }d\text{-dimensional cone}.
    \end{equation*}
    The $\Leftarrow$ direction follows immediately from \cite[Prop. 4.3]{Borisov_2004}. For the $\Rightarrow$ direction, if $\Sigma$ has no $d$-dimensional cone, then $\Xc\cong \Xc'\times(\C^*)^k$ for some $k>0$ with $\Ts$ acting by $(t_{d-k+1},\dots,t_{d})$ on $(\C^*)^k$, so $\Xc^\Ts=\emptyset$.

    Now the lemma follows by combining \cite[Prop. 4.3]{Borisov_2004} and \cite[Eq. (3)]{borisov2005ktheory}.
\end{proof}

\begin{definition}
    Because we require an orbifold $\Xc$ to have generically trivial stabilizers, there exists an open dense subscheme $U$ in $\Xc$. We call this the \textbf{non-stacky locus}. The coarse moduli space $\pi:\Xc\to X$ restricts to the identity on $U$, making $U$ an open subscheme of $X$. We denote by $S$ its complement with the reduced subscheme structure.
\end{definition}

\begin{remark}
    Throughout, we consider the choice of $U$ and $S$ as given. In examples, the choice of $U$ will be clear. For example, for any global quotient orbifold $[V/G]$, we take $U$ to be $[\tilde{U}/G]$, where $\tilde{U}$ is the open subscheme where $G$ acts freely.
\end{remark}

\subsection{Moduli Spaces on Orbifolds}\label{sect:orbifold_moduli}

Studying generating series of orbifold DT invariants, we encounter various types of moduli spaces on orbifolds. Let $\Xc$ be an orbifold, possibly equipped with a $\Ts$-action.

\subsubsection{Hilbert Schemes of Points on Orbifolds}
Similar to \cite[Section 5.1]{bcz_nef_divisors_for_moduli_of_complexes_with_cpt_supp}, we consider the Grothendieck group $K_c(\Xc)$ of $D^b_c(\Xc)$, the compactly supported objects in the bounded derived category $D^b(\Xc)$ of coherent sheaves on $\Xc$. The numerical Grothendieck group $\Nr_c(\Xc)$ is $K_c(\Xc)$ modulo the kernel of the Euler pairing 
\begin{equation*}
    \chi(-,-) : K(\Db_{\mathrm{perf}}(\Xc))\times K_c(\Xc) \to \Z
\end{equation*}
on $K_c(\Xc)$. We will focus on $\Nr_0(\Xc)$, the subgroup generated by 0-dimensional sheaves. For a given class $\alpha\in \Nr_0(\Xc)$, \cite{os08_quot_dm} define a Hilbert scheme
\begin{equation*}
    \hilb^\alpha(\Xc)
\end{equation*}
of substacks $\Zc\subset \Xc$ of class $[\Oc_\Zc]=\alpha$. Even though $\Xc$ is an orbifold, this turns out to be a scheme, as the substacks $\Zc$, which it parametrizes, do not have automorphisms. Although such substacks are more than just collections of points, we will refer to these $\hilb^\alpha(\Xc)$ as Hilbert schemes of points on $\Xc$, writing $\hilb(\Xc)\coloneqq \bigsqcup_{\alpha\in\Nr_0(\Xc)}\hilb^\alpha(\Xc)$.

Consider the subset of effective classes
\begin{equation*}
    \Cr_0(\Xc) \coloneqq \left\{\alpha=[E]\in \Nr_0(\Xc)\ |\ 0\neq E \text{ zero-dimensional}\right\}
\end{equation*}
Note that by definition, $\hilb^0(\Xc)$ is a point, and for $\alpha\neq 0$ the Hilbert scheme $\hilb^\alpha(\Xc)$ is only non-empty if $\alpha$ is effective.

\begin{example}\label{ex:global_quotient_1}
    Given a finite abelian group $G$ of order $r$ acting on a variety $\C^d$ acting as a diagonally embedded subgroup $G\subset \slg(d)$, we consider the global quotient stack $\Xc=[\C^d/G]$. Its numerical Grothendieck group of 0-dimensional sheaves is $\Nr_0(\Xc)=\Z^{\oplus r}$. We see this as follows.

    Take the $G$-equivariant embedding of the origin into $\C^d$ and view it as a closed embedding
    \begin{equation*}
        p:\BG{G}\hookrightarrow \Xc.
    \end{equation*}
    As $K(\BG{G})\cong \Z^{\oplus r}$ spanned by the irreducible representations $\rho_0,\dots,\rho_{r-1}$ of $G$, it suffices to show that $p_*$ is an isomorphism.

    To show that $p_*$ is surjective, take the class $[\Fc]$ of a $0$-dimensional coherent sheaf. Any sheaf can be deformed algebraically to be supported at the fixed point, so we may assume $\Fc$ is supported at the origin. We show it is in the image of $p_*$ by induction on the length of $\Fc$. This is trivial for length $0$. Consider the short exact sequence
    \begin{equation*}
        0\to \Fc' \to \Fc \to p_*p^*\Fc \to 0,
    \end{equation*}
    where $\Fc'$ is the kernel of the adjunction. By induction $[\Fc']$ is in the image of $p_*$ as a sheaf of lower length than $\Fc$. Hence, $[\Fc]=[\Fc']+[p_*p^*\Fc]$ is also in the image of $p_*$.

    To show that $p_*$ is injective, it suffices to show that $p_*(\rho_i)\neq 0$ for all $i$. We take $q:\Xc\to\BG{G}$ given by the $G$-equivariant morphism to the point. This yields the non-vanishing classes $q^*\rho_j$. Using $p\circ q=\id$ and the projection formula, we get
    \begin{equation*}
        \chi\left(q^*\rho_j,p_*\rho_i\right)=\delta_{ij},
    \end{equation*}
    which shows that the class $p_*(\rho_i)$ doesn't vanish for any $i$.
\end{example}

\subsubsection{Moduli of 0-dimensional Sheaves on Orbifolds}\label{setup:sec:moduli_0_dim_sheaves}

We additionally consider moduli of 0-dimensional sheaves on orbifolds. For schemes, \cite[Ex. 4.3.6]{hl10} shows that the moduli spaces of 0-dimensional sheaves are exactly the symmetric products of the base scheme. Following this, we use moduli of 0-dimensional sheaves on orbifolds as our analogues of symmetric products. The existence of such a moduli space can be collected as follows from the literature.

\begin{proposition}
    Let $\Xc$ be an orbifold, possibly equipped with a $\Ts$-action. Then there exists an algebraic stack $\Mf_0$ of zero-dimensional sheaves on $\Xc$, which is locally of finite presentation and has affine diagonal. If $\Xc$ is equipped with a $\Ts$-action, so is $\Mf_0$. Moreover, this stack has a good moduli space
    \begin{equation*}
        \Mf_0 \to M_0,
    \end{equation*}
    which also comes equipped with a $\Ts$-action, compatible with the good moduli space maps.
\end{proposition}
\begin{proof}
    Note that the $\Ts$-action on good moduli space is induced by universality of a good moduli space if one exists.

    Note first that the moduli stack of coherent sheaves on $\Xc$ in \cite[Section 8]{hall2013openness} and the moduli stack $\Mf_{QCoh(\Xc)}$ of \cite[Def. 7.8]{alper2022existence} only differ by the proper support assumption of the former. So, for the substack $\Mf_0$ of sheaves with zero-dimensional support these stacks agree. Hence, $\Mf_0$ is an algebraic stack locally of finite presentation and with affine diagonal by \cite[Theorem 8.1]{hall2013openness}.

    To show, that $\Mf_0$ has a good moduli space, we follow the proof of \cite[Theorem 7.23]{alper2022existence}. As $\Mf_0$ is only locally of finite presentation, the proof needs to be adapted, so that instead of \cite[Theorem A]{alper2022existence} we use \cite[Theorem 4.1]{alper2022existence}, which, over $\C$ gives the following conditions for the existence of a good moduli space:
    \begin{itemize}
        \item $\Mf_0$ has affine diagonal,
        \item closed points have linearly reductive stabilizers,
        \item $\Mf_0$ is $\Theta$-reductive with respect to DVRs essentially of finite type over $\C$, and
        \item $\Mf_0$ has unpunctured inertia with respect to DVRs essentially of finite type over $\C$.
    \end{itemize}
    We have already seen that $\Mf_0$ has affine diagonal. That closed points have linearly reductive stabilizers is implied by \cite[Prop. 3.47]{alper2022existence} as in the proof of \cite[Theorem 7.23]{alper2022existence}. By \cite[Lemma 7.16, 7.17]{alper2022existence} $\Mf_0$ is with $\Theta$-reductive and $\Ss$-complete with respect to essentially finite type DVRs. The proof of \cite[Theorem 5.3(2)]{alper2022existence} then shows that $\Mf_0$ has unpunctured inertia with respect to DVRs essentially of finite type.
\end{proof}

Given a class $\alpha\in \Nr_0(\Xc)$, we consider the subspace
\begin{equation*}
    M_\alpha(\Xc)\subset M_0(\Xc)
\end{equation*}
of sheaves on $\Xc$ of class $\alpha$. This comes with a Hilbert-Chow morphism
\begin{equation*}
    \HC_\alpha :\hilb^\alpha(\Xc) \to M_\alpha(\Xc),\ [\Zc]\mapsto [\Oc_\Zc].
\end{equation*}

\subsubsection{Coarse Spaces}\label{setup:sec:coarse}

Let $\pi:\Xc\to X$ be a coarse moduli space of $\Xc$. We assume $X$ is connected. Pushforward of sheaves maps $\Nr_0(\Xc)$ into $\Nr_0(X)$. For a connected $X$, we have $\Nr_0(X)=\Z$. We write
\begin{equation*}
    b(\alpha) := \pi_*(\alpha)\in\Z
\end{equation*} 
for any $\alpha\in\Nr_0(\Xc)$. In fact, $\pi_*$ admits a section $\Nr_0(X)=\Z\to \Nr_0(\Xc)$ by taking a point $p$ in the non-stacky locus of $\Xc$ and mapping $1\in\Z$ to $[p]$. Hence, we get a splitting
\begin{equation*}
    \Nr_0(\Xc)\cong \Z\oplus\widetilde{\Nr}_0(\Xc),
\end{equation*}
where the projection to the first component is just $b=\pi_*$. Pushforward of sheaves induces a morphism
\begin{equation*}
    \xi_\alpha:M_\alpha(\Xc) \to \sym^{b(\alpha)}(X),\ [\Fc]\mapsto [\pi_*\Fc],
\end{equation*}
where we use the identification of moduli of 0-dimensional sheaves with symmetric products from \cite[Ex. 4.3.6]{hl10}.

\begin{example}\label{ex:global_quotient_2}
    In the global quotient setup of Example \ref{ex:global_quotient_1}, $b(\alpha)$ of an integer vector $\alpha=\mathbf{n}=(n_0,\dots,n_{r-1})$ is exactly $n_0$.
\end{example}

\subsection{Virtual Structure}

We are interested in computing virtual invariants of the above Hilbert schemes of points on orbifolds. These are equipped with an obstruction theory, which yields a (twisted) virtual structure sheaf. Let $\Xc$ be an orbifold, and $\alpha\in \Nr_0(\Xc)$ a given class. We assume $\Xc$ is Calabi--Yau of dimension 3.

\subsubsection{Obstruction Theory}\label{setup:sec:obstruction_theory}

\begin{definition}
    Let $M$ be a Deligne-Mumford stack. An \textbf{obstruction theory} is an object $\Eb\in\Dm_\qcoh(M)$ together with a morphisms
    \begin{equation*}
        \phi: \Eb \to \tau_{\geq -1}\Lb_{M},
    \end{equation*}
    such that $h^0(\phi)$ is an isomorphism and $h^{-1}(\phi)$ is surjective. An obstruction theory is
    \begin{itemize}
        \item \textbf{perfect} if it is a perfect complex of amplitude $[-1,0]$,
        \item \textbf{symmetric} if $\Eb$ is a perfect complex and there is an isomorphism $\Theta: \Eb \xrightarrow{\sim} \Eb^\vee[1]$ (or $\Theta: \Eb \xrightarrow{\sim} \kappa\otimes\Eb^\vee[1]$ in the equivariant case) satisfying $\Theta^\vee[1]=\Theta$.
    \end{itemize}
    The dual of a perfect obstruction theory is referred to as the \textbf{virtual tangent sheaf} $\Tb^\vir$.
\end{definition}

Identifying the Hilbert schemes of points $\hilb^\alpha(\Xc)$ with a moduli space of ideal sheaves of class $[\Oc_\Xc]-\alpha$ in $\Xc$, \cite{ht13_atiyah_obstruction_theory} gives us a perfect obstruction theory
\begin{equation*}
    \Eb = \pi_{\hilb,*}\left(\homc(\Ic,\Ic)_0 \right)[2]\to \tau_{\geq -1}\Lb_{\hilb^\alpha(\Xc)},
\end{equation*} 
where $\Ic$ is the universal ideal sheaf and $\homc(\Ic,\Ic)_0$ denotes the traceless part. For a CY3 orbifold $\Xc$, possibly equipped with a $\Ts$-action, this perfect obstruction theory is symmetric by Grothendieck-Verdier duality.

In the case of a global quotient stack $\Xc=[W/G]$, where $G$ is a finite group, we can describe the obstruction theory for $\Xc$ using the one on $W$ as follows. First note that ideal sheaves on $\Xc$ are exactly ideal sheaves on $W$ equipped with a $G$-equivariant structure. These correspond exactly to $G$-invariant subschemes of $W$. Hence, we find that $\hilb(\Xc)$ is the $G$-fixed part of $\hilb(W)$
\begin{equation*}
    \hilb(\Xc)=\hilb(W)^G.
\end{equation*}
The $G$-action on $W$ equips $\hilb(W)$ with a $G$-action and the obstruction theory $\Eb_W$ above naturally comes with a $G$-equivariant structure\cite{ricolfi2020equivariant}. We can naturally identify the obstruction theories
\begin{equation}\label{setup:eq:obstruction_theory_fixed_part}
    \Eb_\Xc\cong \restr{\Eb_W}{\hilb(\Xc)}^f,
\end{equation}
where $\restr{\Eb_W}{\hilb(\Xc)}^f$ is the $G$-fixed part of the restriction of $\Eb_W$ to $\hilb(\Xc)$.

\subsubsection{Virtual Structure Sheaf}\label{sec:vir_str_sh}
Given a perfect obstruction theory $\Eb$ on a moduli space as above, together with a presentation as a 2-term complex of vector bundles $\Eb \cong [E^{-1}\to E^0]$, \cite{bf97_intrinsic_normal_cone} construct a virtual fundamental class and a virtual structure sheaf. This turns out to be independent of the specific chosen presentation $E^\bullet$. We are interested in the virtual structure sheaf. Consider the embedding of the intrinsic normal cone $\Cf\hookrightarrow h^1/h^0(E^{\bullet\vee})$ given by the perfect obstruction theory. Then by existence of the resolution $E^\bullet$, we get an induced cone $C\subset E_1$. The virtual structure sheaf is
\begin{equation*}
    \Oc^\vir = \bigoplus_i \cTor^i_{E_1}\left(\Oc_C,\Oc_X\right)[i]\in \Db(X),
\end{equation*}
where $\Oc_X$ becomes a sheaf on $E_1$ via the $0$-section. Note that \cite[Remark 5.4]{bf97_intrinsic_normal_cone} define this as a graded commutative sheaf of algebras, whereas we define it as an element in the derived category, which is reflected in our notation as shifts, which are implicit in their notation.

\subsubsection{Twisted Virtual Structure Sheaf}
For computations, it is useful to consider a twisted virtual structure sheaf, that is the virtual structure sheaf tensored by a specific choice of line bundle. A specific choice of such a twist
\begin{equation*}
    \widehat{\Oc}^\vir = \Oc^\vir \otimes \det(\Eb)^{1/2}\in K(\hilb^\alpha(\Xc))
\end{equation*}
allows the use of the so-called rigidity principle, which we will recall and use in Section \ref{sect:rigidity}. This choice and technique were introduced originally by Okounkov in \cite{ok15}. There, a quiver-theoretic description of the moduli space is used to introduce the K-theory class of the desired twist. In our case, we use a sheaf theoretic description, because this makes it easier to prove a factorization property of the resulting (twisted) virtual structure sheaves on Hilbert schemes of points on orbifolds. For this to work, we use a description of this twist as a line bundle given in \cite{levine23_twist}. Note that any shift of the chosen line bundle $\det(\Eb)^{1/2}$ is still a root of $\det(\Eb)$ in $\Db_{\Z/2}(-)$ and in K-theory, but the shift makes the resulting sheaves factorizable. Otherwise, the diagrams which are required to commute in the definition of a factorizable system might only commute up to a sign. First, we need the following Lemma.

\begin{lemma}\label{setup:lemma:twist_loc_free_rank}
    Let $\Xc$ be an CY3 orbifold, and let $\hilb(\Xc)$ be its Hilbert scheme of points with universal family $\Zc$. Take $p_\Zc$ to be the morphism $\Zc\hookrightarrow\hilb(\Xc)\times\Xc\to\hilb(\Xc)$. Then
    \begin{equation*}
        p_{\Zc *}(\Oc_\Zc)
    \end{equation*}
    is a locally free sheaf of rank $b(\alpha)$ on the component $\hilb^\alpha(\Xc)$.
\end{lemma}
\begin{proof}
    To prove this, we consider the following cartesian diagram for a point $p$ in $\hilb^\alpha(\Xc)$.
    \begin{equation*}
        \begin{tikzcd}
            \Zc|_p \ar[r]\ar[d,"\restr{p_\Zc}{p}"] & \Zc \ar[d, "p_\Zc"]\\
            p \ar[r] & \hilb^\alpha(\Xc).
        \end{tikzcd}
    \end{equation*}
    By definition, the family of substacks of $\Xc$ over $\kappa(p)$ is $0$-dimensional and $\Oc_\Zc$ is properly supported and flat over $\hilb^\alpha(\Xc)$. We use \cite[Theorem A]{hall_cohom_base_change_stacks} to prove our Lemma. In fact, because $p_\Zc$ has $0$-dimensional fibers and hence no higher pushforwards, note that for $q=0$, conditions (1)(c) and (2)(a) of \cite[Theorem A]{hall_cohom_base_change_stacks} are satisfied, so that
    \begin{equation*}
        b^0(p): \restr{p_{\Zc *}(\Oc_\Zc)}{p} \xrightarrow{\sim} \left(\restr{p_{\Zc}}{p}\right)_*\left(\Oc_{\Zc|_p}\right) 
    \end{equation*}
    is an isomorphism and $p_{\Zc *}(\Oc_\Zc)$ is locally free in an open neighborhood of $p$. Since $p$ was chosen arbitrarily, this makes $p_{\Zc *}(\Oc_\Zc)$ locally free. We can use the isomorphism $b^0(p)$ to compute its rank
    \begin{align*}
        \rk\left(p_{\Zc *}(\Oc_\Zc)\right) &= \dim_{\kappa(p)}\left(\restr{p_{\Zc *}(\Oc_\Zc)}{p}\right)&\\
        &= \chi\left(p, \restr{p_{\Zc *}(\Oc_\Zc)}{p}\right)&\\
        &= \chi\left(p, \left(\restr{p_{\Zc}}{p}\right)_*\left(\Oc_{\Zc|_p}\right)\right)& (\text{using }b^0(p))\\
        &= \chi\left(\Zc|_p, \Oc_{\Zc|_p}\right)& (\text{pushforward-invariance of }\chi).
    \end{align*}
    Since we have already shown $p_{\Zc *}(\Oc_\Zc)$ is locally free, we may assume $p$ is a closed point of $\hilb^\alpha(\Xc)$, so that $\Zc|_p$ is a closed substack of $\Xc$ with $\left[\Oc_{\Zc|_p}\right]=\alpha$. Then we may again use invariance of $\chi$ under pushforward along the closed embedding and the coarse moduli space $\pi:\Xc\to X$ to get
    \begin{equation*}
        \rk\left(p_{\Zc *}(\Oc_\Zc)\right) = \chi\left(X, \pi_*\Oc_{\Zc|_p}\right),
    \end{equation*}
    which is just $b(\alpha)$ by definition of $b(-)$.
\end{proof}

Using this Lemma, we can explicitly define the twisted virtual structure sheaf.

\begin{definition}\label{setup:def:twisted_vir_str_sh}
    Let $\Xc$ be an CY3 orbifold, and let $\hilb(\Xc)$ be its Hilbert scheme of points with universal family $\Zc$. We define the twisted virtual structure sheaf as
    \begin{equation*}
        \widehat{\Oc}^\vir = \Oc^\vir \otimes \det\left(p_{\Zc *}(\Oc_\Zc)\right)^{-1}[b(\alpha)]
    \end{equation*}
    in the projective case, and
    \begin{equation*}
        \widehat{\Oc}^\vir = \Oc^\vir \otimes \det\left(\kappa^{\frac{1}{2}}p_{\Zc *}(\Oc_\Zc)\right)^{-1}[b(\alpha)]
    \end{equation*}
    in the toric case, where $\kappa$ is the $\Ts$-weight of the canonical bundle of $\Xc$.
\end{definition}

\begin{remark}
    To work with expressions like $\kappa^{\frac{1}{2}}$, we implicitly work in a cover of the torus $\Ts$, so that all square roots of $\Ts$-characters exist. 
\end{remark}

In the (equivariant) projective case \cite{levine23_twist} shows that this is a root of the determinant of the obstruction theory. Note that the computations of determinants (but not necessarily the rank computations) in their proof still work for projective DM stacks using \cite[Cor. 2.10]{nironi09_duality} instead of Grothendieck-Serre duality for proper morphisms of schemes. In the case of toric Calabi--Yau orbifolds, we show that the twist defined above is a root of $\det(\Eb)$ at $\Ts$-fixed points of the Hilbert scheme of points. As $\Ts$-equivariant Euler characteristics are defined via localization, this suffices for computational purposes.

\begin{lemma}
    Let $\Xc$ be a toric CY3 orbifold. Consider the Hilbert scheme of points $\hilb (\Xc)$ with its universal family $\Zc$. Let $\Eb$ be the $\Ts$-equivariant obstruction theory on $\hilb(\Xc)$. For any closed point $p$ in the $\Ts$-fixed locus $\hilb (\Xc)$, there is a $\Ts$-equivariant identification of K-theory classes
    \begin{equation*}
        \det([\Eb|_p]) = \left[\restr{\kappa^{-\rk(p_{\Zc *}(\Oc_\Zc))}\det\left(p_{\Zc *}(\Oc_\Zc)\right)^{-2}}{p}\right] = \left[\restr{\det\left(\kappa^{\frac{1}{2}}p_{\Zc *}(\Oc_\Zc)\right)^{-2}}{p}\right].
    \end{equation*}
\end{lemma}
\begin{proof}
    The proof is similar to the one of \cite[Theorem 5.2]{levine23_twist}, except that we work at a fixed point $p=[Z]$ in $\hilb (\Xc)$ directly making some computations easier in the toric case. From now on we consider only K-theory classes, omitting $[-]$ from the notation. Using the push-pull formula, we get
    \begin{align*}
        \Eb|_p &= R\Hom (I_Z,I_Z)_0[2],\\
        \restr{p_{\Zc *}(\Oc_\Zc)}{p} &= R\Gamma (\Oc_Z).
    \end{align*}
    We write $\chi(-,-)$ for $R\Hom (-,-)$ and $\chi(-)$ for $R\Gamma (-)$. Then
    \begin{equation*}
        R\Hom (I_Z,I_Z)_0=\chi(I_Z,I_Z)-\chi(\Oc,\Oc) = \chi(\Oc_Z,\Oc_Z)-\chi(\Oc,\Oc_Z)-\chi(\Oc_Z,\Oc).
    \end{equation*}
    Using that $Z$ is of codimension $2$, we get $\det(\chi(\Oc_Z,\Oc_Z))=\Oc$. It remains to check the equivariant weight of this trivial line bundle. This is a local computation at the fixed point, so we may assume $\Xc=[\C^3/G]$ by Lemma \ref{setup:lemma:toric_loc_str}, and $Z$ corresponds to some colored partition $\pi$. We can write the structure sheaf of the fixed point $Z$ as $\Oc_Z=\sum_{\dv{n}\in\pi}\dv{t}^\dv{n} \Oc_0(\rho_{C(\dv{n})})$, where $\rho_{C(\dv{n})}$ is the character determined by the coloring of $\pi$. Tensoring the standard $\Ts$-equivariant resolution of $\Oc_0$ with an irreducible representation $\rho_{j}$, we compute the character
    \begin{equation*}
        \chi_{\Ts\times G}(\Oc_0(\rho_i),\Oc_0(\rho_j)) = \rho_i\rho_j\chi_\Ts(\Oc_0,\Oc_0) = \rho_i\rho_j\prod_{k=1}^3 (1-t_k).
    \end{equation*}
    Together with the presentation of $\Oc_Z$, this gives us
    \begin{equation*}
        \chi_{\Ts\times G}(\Oc_Z,\Oc_Z) = \chi_{\Ts\times G}(\Oc_Z)\chi_{\Ts\times G}(\Oc_Z)^\vee\prod_{k=1}^3 (1-t_k),
    \end{equation*}
    which becomes $1$ after taking the determinant.
    
    We have $\chi(\Oc,\Oc_Z)=\chi(\Oc_Z)$ and
    \begin{equation*}
        \chi(\Oc_Z,\Oc)=\chi(\Oc_Z^\vee) = -\kappa^{-1}\chi(\Oc_Z)^\vee
    \end{equation*}
    by equivariant Serre duality for compactly supported sheaves. Putting together the above identifications to compute determinants, we get
    \begin{align*}
        \det(\Eb|_p) &= \det(\kappa^{-1}\chi(\Oc_Z)^\vee-\chi(\Oc_Z))\\
        &= \kappa^{-\rk(\chi(\Oc_Z))}\det(\chi(\Oc_Z))^{-2}\\
        &=\restr{\kappa^{-\rk(p_{\Zc *}(\Oc_\Zc))}\det\left(p_{\Zc *}(\Oc_\Zc)\right)^{-2}}{p}\\
        &=\restr{\det\left(\kappa^{\frac{1}{2}}p_{\Zc *}(\Oc_\Zc)\right)^{-2}}{p}.
    \end{align*}
\end{proof}

\subsection{Invariants}

\subsubsection{Generating Series}
With the obstruction theory and twisted virtual structure sheaves above, we are interested in computing the generating series of degree $0$ DT invariants
\begin{equation}
    \Zs(\Xc) := 1+ \sum_{\alpha \in \Cr_0(\Xc)} q^{\alpha} \chi\left(\hilb^\alpha(\Xc), \widehat{\Oc}^\vir\right)\in K(\pt)[\Cr_0(\Xc)],
\end{equation}
where $q^{\alpha}$ is the standard notation for the semigroup ring element associated to $\alpha\in\Cr_0(\Xc)$. This definition works as stated only for proper $\Xc$. For a quasi-compact $\Xc$ with an action of a group $\Ts$, so that the fixed loci of $\hilb^\alpha(\Xc)$ are proper, we take the equivariant Euler characteristic $\chi_\Ts$, which is defined via localization.

\begin{example}
    Continuing Example \ref{ex:global_quotient_1} of global quotient orbifolds, the sum above simplifies as $\Nr_0([\C^d/G])=\Z^r$ and $\Cr_0([\C^d/G])=\N^r\setminus \{0\}$.
    \begin{equation*}
        \Zs(\Xc) := 1+ \sum_{\dv{n}\in\N^r\setminus \{0\}} q_0^{n_0}\cdots q_{r-1}^{n_{r-1}} \chi\left(\hilb^\dv{n}([\C^d/G]), \widehat{\Oc}^\vir\right).
    \end{equation*}
\end{example}

\subsubsection{Localization}\label{setup:sec:loc}
In the case of a projective orbifold $\Xc$, the moduli spaces under consideration will also be projective, making $\chi$ well-defined. If $\Xc$ is not necessarily projective, but is equipped with a $\Ts$-action, the moduli spaces have an induced $\Ts$-action, and we define $\chi$ via localization as follows. For a noetherian separated scheme $M$ with a $\Ts$-action, the localization theorem \cite[Theorem 2.2]{thomason92loc} tells us that
\begin{equation*}
    i_*:K_\Ts\left(M^\Ts\right)_\loc \to K_\Ts\left(M\right)_\loc
\end{equation*}
is an isomorphism, where $i:M^\Ts\hookrightarrow M$ denotes the embedding of the fixed locus, and $K_\Ts(-)_\loc$ denotes equivariant K-theory with the equivariant parameters inverted. If the fixed locus $M^\Ts$ is proper, we define for a K-theory class $F$ on $M$
\begin{equation*}
    \chi\left(M,F\right) \coloneqq \chi\left(M^\Ts,i_*^{-1}(F)\right) \in K_\Ts(\pt)_\loc.
\end{equation*}
For a $\Ts$-equivariant morphism $\pi:M\to N$, consider the commutative diagram
\begin{equation*}
    \begin{tikzcd}
        M^\Ts \ar[r, hookrightarrow, "i^M"]\ar[d, "\pi^\Ts"] & M \ar[d,"\pi"]\\
        N^\Ts \ar[r, hookrightarrow, "i^N"] & N.
    \end{tikzcd}
\end{equation*}
By commutativity, we can check $i^N_* \pi^\Ts_* (i^M_*)^{-1}=\pi_*$, which implies that $\chi$ defined by localization is still invariant under pushforward if $M^\Ts$ and $N^\Ts$ are proper.

If $M$ is smooth, then we can find an explicit inverse for $i_*$. This is often also simply referred to as localization. To define this inverse, consider the normal bundle $\Nc_{M^\Ts/M}$ of $M^\Ts$ in $M$. Then
\begin{equation*}
    i_*^{-1} = \frac{i^*(-)}{\es(\Nc_{M^\Ts/M})},
\end{equation*}
where $\es(-)$ is the K-theoretic Euler class defined by setting $\es(\Ec)\coloneqq \left[\bigwedge^\bullet \left(\Ec^\vee\right)\right]\in K_\Ts(M^\Ts)$ for a locally free sheaf $\Ec$ and extending by $\es(\Ec_1-\Ec_2)\coloneqq \es(\Ec_1)/\es(\Ec_2)$.

In the course of this paper, we will evaluate Euler characteristics by localization on coarse spaces $X$ of toric orbifolds $\Xc$, symmetric products thereof, components of moduli spaces of zero-dimensional sheaves on $\Xc$, and Hilbert schemes of points on $\Xc$. For a toric orbifold $\Xc$, its coarse space $X$ is a toric variety and hence has isolated fixed points. Consequently, any symmetric product of $X$ also has isolated fixed points. The Hilbert scheme of points on $\Xc$ also has isolated fixed points\cite[Lemma 13]{bryan2010orbifold}. We describe these in the local case more detail in Section \ref{sec:colored_vertex} below. Any Euler characteristic on $M_0(\Xc)$, which we evaluate will be of a class pushed forward via the Hilbert-Chow morphism making the Euler characteristic well-defined by further pushforward to $\sym^{b(\alpha)}(X)$.

\subsubsection{Virtual Localization}
Another case, in which an explicit inverse to $i_*$ is known is virtual localization\cite{graber1997localization, ciocanfontanine2007virtual}. Given a moduli space $M$ with a $\Ts$-action as above, which is additionally equipped with a perfect obstruction theory $\Eb$, we can restrict the perfect obstruction theory to the fixed locus. It splits into a $\Ts$-fixed and a $\Ts$-moving part
\begin{equation*}
    \Eb|_{M^\Ts} = \Eb|_{M^\Ts}^f \oplus \Eb|_{M^\Ts}^m.
\end{equation*}
Then $\Eb|_{M^\Ts}^f$ turns out to be an obstruction theory on $M^\Ts$ and $\Nc^\vir\coloneqq \Eb^\vee|_{M^\Ts}^m$ is the virtual normal bundle. The virtual localization formula is then
\begin{equation*}
    \widehat{\Oc}^\vir_M = i_*\left(\frac{\widehat{\Oc}^\vir_{M^\Ts}}{\hat{\es}(\Nc^\vir)}\right),
\end{equation*}
where $\hat{\es}(-)\coloneqq \es(-)\otimes \det(-)^{1/2}$ is the symmetrized K-theoretic Euler class.

\subsubsection{Colored Vertex}\label{sec:colored_vertex}
We consider specifically the case $\Xc=[\C^3/G]$, where $G$ is a finite abelian diagonally embedded subgroup of $\slg(3)$. We consider this equivariantly with the natural action by $\Ts=(\C^*)^3$. In this case, the $\Ts$-fixed locus of $\hilb(\Xc)$ consists of isolated fixed points, corresponding to plane partitions colored by the irreducible representations of $G$. This is explained for example in \cite[Appendix A.2]{young10}. As in Example \ref{ex:global_quotient_1}, the color vector corresponds to the K-theory class of the structure sheaf of the substack. For later use, if $\pi$ is a colored plane partition corresponding to a fixed point, we denote by $\pi^G$ the collection of its $0$-colored boxes.

We study the contribution to the generating series $\Zs(\Xc)$ at each such fixed point. 
\begin{proposition}
    Let $\Xc=[\C^3/G]$ with $G$ is a finite abelian diagonally embedded subgroup of $\slg(3)$. For each colored plane partition $\pi$ corresponding to a fixed point of the Hilbert scheme of points $\hilb(\Xc)$, there is a subset $W(\pi)$ of the $\Ts$-weights of $\restr{\Eb_\Xc^\vee}{[I_\pi]}$, such that we get the local contributions 
    \begin{align*}
        \left[\widehat{\Oc}^\vir\right] &= \sum_{\pi} (-1)^{\abs{\pi_0}} \left[\Oc_{[I_\pi]}\right]\hat{a}(\pi)\in K_\Ts(\hilb(\Xc))_\loc,\\
        \hat{a}(\pi) &= \prod_{w\in W} \frac{(\kappa/w)^{1/2}-(w/\kappa)^{1/2}}{w^{1/2}-w^{-1/2}} = \prod_{w\in W} \frac{[\kappa/w]}{[w]},
    \end{align*}
    where we write $[w]\coloneqq w^{1/2}-w^{-1/2}$.
\end{proposition}
\begin{proof}
    If $\pi$ is a colored plane partition corresponding to a fixed point, we let $\Oc_{[I_\pi]}$ be the skyscraper sheaf at the corresponding fixed point $[I_\pi]$. Then the virtual localization formula yields
    \begin{equation*}
        \left[\widehat{\Oc}^\vir\right] = \sum_{\pi} \left[\Oc_{[I_\pi]}\right]\frac{1}{\hat{\es}\left(\Eb_\Xc^\vee|_{[I_\pi]}^m\right)}\in K_\Ts(\hilb(\Xc))_\loc.
    \end{equation*}
    The local contribution $\frac{1}{\hat{\es}\left(\restr{\Eb_\Xc^\vee}{[I_\pi]}^m\right)}$ can be understood as follows. First, write $\restr{\Eb_\Xc^\vee}{[I_\pi]}$ as a sum of its $\Ts$-weights as a class in $K_\Ts(\pt)$. Then, as $\Xc$ is Calabi--Yau, the obstruction theory $\Eb_\Xc$ is symmetric, which means there is a subset $W(\pi)$ of the $\Ts$-weights such that
    \begin{equation*}
        \restr{\Eb_\Xc^\vee}{[I_\pi]} = \sum_{w\in W} \left(w-\frac{\kappa}{w}\right).
    \end{equation*}
    We have seen in Section \ref{setup:sec:obstruction_theory} that the obstruction theory on $\hilb(\Xc)$ is the $G$-fixed part of the obstruction theory on $\hilb(\C^3)$. Then \cite[Section 4.5]{maulik2004gromovwitten} shows that $\Eb_\Xc^\vee$ has no trivial $\Ts$-weights at any fixed point, so $\restr{\Eb_\Xc^\vee}{[I_\pi]}$ does not have any trivial $\Ts$-weights either. Inserting the above weight decomposition into the localization formula gives us
    \begin{equation*}
        \left[\widehat{\Oc}^\vir\right] = \sum_{\pi} \left[\Oc_{[I_\pi]}\right]\det\left(\restr{\Eb^\vee_\Xc}{[I_\pi]}\right)^{-1/2}\prod_{w\in W(\pi)}\left(\frac{1-w/\kappa}{1-w^{-1}}\right)\in K_\Ts(\hilb(\Xc))_\loc.
    \end{equation*}
    The twist at $[I_\pi]$ can be written in terms of weights as
    \begin{equation*}
        \det\left(\restr{\Eb^\vee_\Xc}{[I_\pi]}\right)^{-1/2} = (-1)^{\abs{\pi^G}} \left(\prod_{w\in W}\frac{w^2}{\kappa}\right)^{-1/2} = (-1)^{n_0} \prod_{w\in W}\frac{(\kappa/w)^{1/2}}{w^{1/2}},
    \end{equation*}
    where the sign comes from our choice of shift in Definition \ref{setup:def:twisted_vir_str_sh} of the twist of the virtual structure sheaf. Multiplying by this twist yields exactly the desired formula.
\end{proof}
\section{Factorization}\label{fact:sec:fact}

The constructions and proofs in this section follow closely the treatment of factorization for schemes in \cite{kr}, but the theory for orbifolds is more complicated in several ways. We set up a more general framework using $0$-dimensional K-theory classes of an orbifold instead of numbers of points in a scheme. This applies in the orbifold case, but the combinatorics, which are involved in the proof of the main Theorem \ref{fact:thm:fact} of this section, become more complicated.

\subsection{Factorizable Systems}
\subsubsection{Setup}
We develop a theory of factorization for orbifolds. This works equivariantly with respect to a group action on the underlying orbifold, for example the action of a torus. For clarity, we suppress this equivariance from the notation. The proofs in the equivariant setting are the same. Following \cite[5.3]{ok15}, factorization for schemes $X$ is a property of systems of sheaves $\Fc_n$ on $\sym^n(X)$, which allows us to compute $1+\sum_{n\geq 1}\chi(\sym^n(X),\Fc_n)q^n$ as the plethystic exponential of a series $\sum_{n\geq 1}\chi(X,\Gc_n)q^n$ where $\Gc_i$ is a sequence of sheaves on $X$. This simplifies the problem as $X$ is easier to work with than $\sym^n(X)$. Moreover, in the equivariant case, $\chi(X,\Gc_n)$ can be computed as the product of a fixed localization contribution with a Laurent polynomial in the equivariant parameters. As described in Section \ref{setup:sec:loc}, the definition of $\chi$ in the equivariant case is slightly more involved. In this section, we set up a theory of factorization for orbifolds. In the $\Ts$-equivariant case, the $\Ts$-equivariant structure of the sheaves is preserved by all constructions of this section. In Corollary \ref{fact:cor:comp_pexp_fact}, $\chi$ is defined via localization in the equivariant case.

We will define a notion of a factorizable system in a more general setup. To motivate this, we consider the following morphisms
\begin{equation}\label{fact:cd:setups}
    \begin{tikzcd}
        & & \left[X^{b(\alpha)}/S_{b(\alpha)}\right] \arrow[d, "p"]\\
        \hilb^{\alpha}(\Xc) \arrow[r, "HC_{\alpha}"] & M_{\alpha}(\Xc) \arrow[r, "\xi_\alpha"] & \sym^{b(\alpha)}(X).
    \end{tikzcd}
\end{equation}
Recall from Section \ref{setup:sec:coarse} that $b(\alpha)$ is the pushforward of $\alpha$ to $X$. The vertical map is the quotient morphism. In the cases we are interested in, we will show that the twisted virtual structure sheaves yield a factorizable system of sheaves $\hat{\Oc}^{vir}$ on $\hilb^{\alpha}(\Xc)$. By then pushing forward, we obtain a factorizable system on $\sym^{b(\alpha)}(X)$. Finally pulling back to the stack $\left[X^{b(\alpha)}/S_{b(\alpha)}\right]$, we obtain a factorizable system of $S_{b(\alpha)}$-equivariant sheaves on $X^{b(\alpha)}$. This third factorizable system will be used to prove the main result of this section, Theorem \ref{fact:thm:fact} about generating series of $K$-theory classes of factorizable systems of sheaves.

Before we formulate our notion of factorization, we define the notion of a factorization index set.
\begin{definition}\label{fact:def:fact_ind_set}
    Given a subset $I$ of $\Cr_0(\Xc)$, an element $\alpha$ of $I$ is called $I$\textbf{-indecomposable} if there are no elements $\alpha_1,\alpha_2$ in $I$ such that $\alpha=\alpha_1+\alpha_2$.

    A subset $I$ of $\Cr_0(\Xc)$ is called a \textbf{factorization index set} if
    \begin{itemize}
        \item $b(\alpha)>0$ for every $\alpha\in I$, and
        \item every $I$-indecomposable element $\alpha$ has $b(\alpha)=1$
    \end{itemize}
    Note that we don't require $I$ to be closed under addition. If $I$ is closed under addition, we call $I$ a \textbf{factorization index semigroup}.

    Take a point $p$ in the non-stacky locus of $\Xc$ and write $\Delta_I$ for the subset $\Z_{>0}[\Oc_p]$ of $\Cr_0(\Xc)$. 
\end{definition}

\begin{remark}
    Let $I\subseteq \Cr_0(\Xc)$ be a factorization index set. By the first condition, any decomposition of some $\alpha\in I$ into $\alpha_1+\alpha_2$ with $\alpha_i\in I$ must have $b(\alpha_i)<b(\alpha)$. This in turn implies that for every factorization index set $I\subseteq\Cr_0(\Xc)$:
    \begin{itemize}
        \item Any $\alpha\in I$ with $b(\alpha)=1$ must be $I$-indecomposable, since any decomposition would require a summand with $b(-)=0$.
        \item Any $\alpha\in I$ can be written as a finite sum of $I$-indecomposable elements, because $b(-)>0$ on $I$.
    \end{itemize}
\end{remark}

The condition for a factorization index set roughly says that an $I$-indecomposable class can only have one underlying point in the coarse space, possibly together with some stacky contribution. We will use this in the proof of Theorem \ref{fact:thm:fact} below, to guarantee the existence of splittings of classes $\alpha$ in $I$ whenever $b(\alpha)>1$. The reader should have the following example in mind, which is the example we use in all applications.

\begin{lemma}\label{fact:lemma:hilb_fact_index_set}
    For an orbifold $\Xc$, set
    \begin{equation*}
        I := \{\alpha\ |\ \hilb^{\alpha}(\Xc)\neq \emptyset\}.
    \end{equation*}
    This is a factorization index set. If $\Xc=\Xc'\times \C$, then $I$ is closed under addition.
\end{lemma}
\begin{proof}
    The first condition of a factorization index set holds immediately, because, by definition, all classes $\alpha\in I$ can be written as $\alpha=[\Oc_\Zc]$, where $\Zc$ is some substack. Take a class $\alpha=[\Oc_\Zc]$ in $I$, which is $I$-indecomposable. Any morphism $f:\Oc_\Xc\to\Oc_\Zc$ gives us a splitting
    \begin{equation*}
        \alpha= [\Oc_\Zc] = [\im(f)] + [\cok(f)].
    \end{equation*}
    Both $[\im(f)]$ and $[\cok(f)]$ are in $I$ as $\im(f)$ and $\cok(f)$ are sheaves with surjective morphisms from $\Oc_\Xc$. By $I$-indecomposability of $\alpha$, $\im(f)$ or $\cok(f)$ must vanish. Hence, every non-zero morphism $f\in\Hom_\Xc\left(\Oc_\Xc,\Oc_\Zc\right)$ must be surjective. By the adjunction of $i:\Zc\hookrightarrow\Xc$, which preserves surjections, this means that every non-zero morphism $f\in\Hom_\Zc\left(\Oc_\Zc,\Oc_\Zc\right)$ must be surjective. Using finite length of $\Zc$ and additivity of the length function on the short exact sequence
    \begin{equation*}
        0\to \ker(f) \to \Oc_\Zc \xrightarrow{f} \Oc_\Zc \to \cok(f) \to 0
    \end{equation*}
    associated to $f$, we see that $f$ is surjective if and only if it is injective. So, every non-zero morphism $f\in\Hom_\Zc\left(\Oc_\Zc,\Oc_\Zc\right)$ must be an isomorphism, which gives us $\Hom_\Zc\left(\Oc_\Zc,\Oc_\Zc\right)=\C$ as in \cite[Cor. 1.2.8]{hl10}. Therefore, 
    \begin{equation*}
        b(\alpha)=h^0(\Oc_\Zc)=\dim \Hom_\Zc\left(\Oc_\Zc,\Oc_\Zc\right) =1.
    \end{equation*}

    To prove the second part, take $\alpha_1=[\Oc_{\Zc_1}]$ and $\alpha_2=[\Oc_{\Zc_2}]$ in $I$. As $\Xc$ is $\Xc'\times \C$, we may move $\Zc_1$ and $\Zc_2$ apart along the trivial direction without changing the K-theory classes $\alpha_i$, so that we may assume the $\Zc_i$ are disjoint. Then $\left[\Oc_{\Zc_1\sqcup\Zc_2}\right]$ has class $\alpha_1+\alpha_2$ and is contained in $I$.
\end{proof}
\begin{figure}[!ht]
    \centering
    \includegraphics[width=0.7\linewidth]{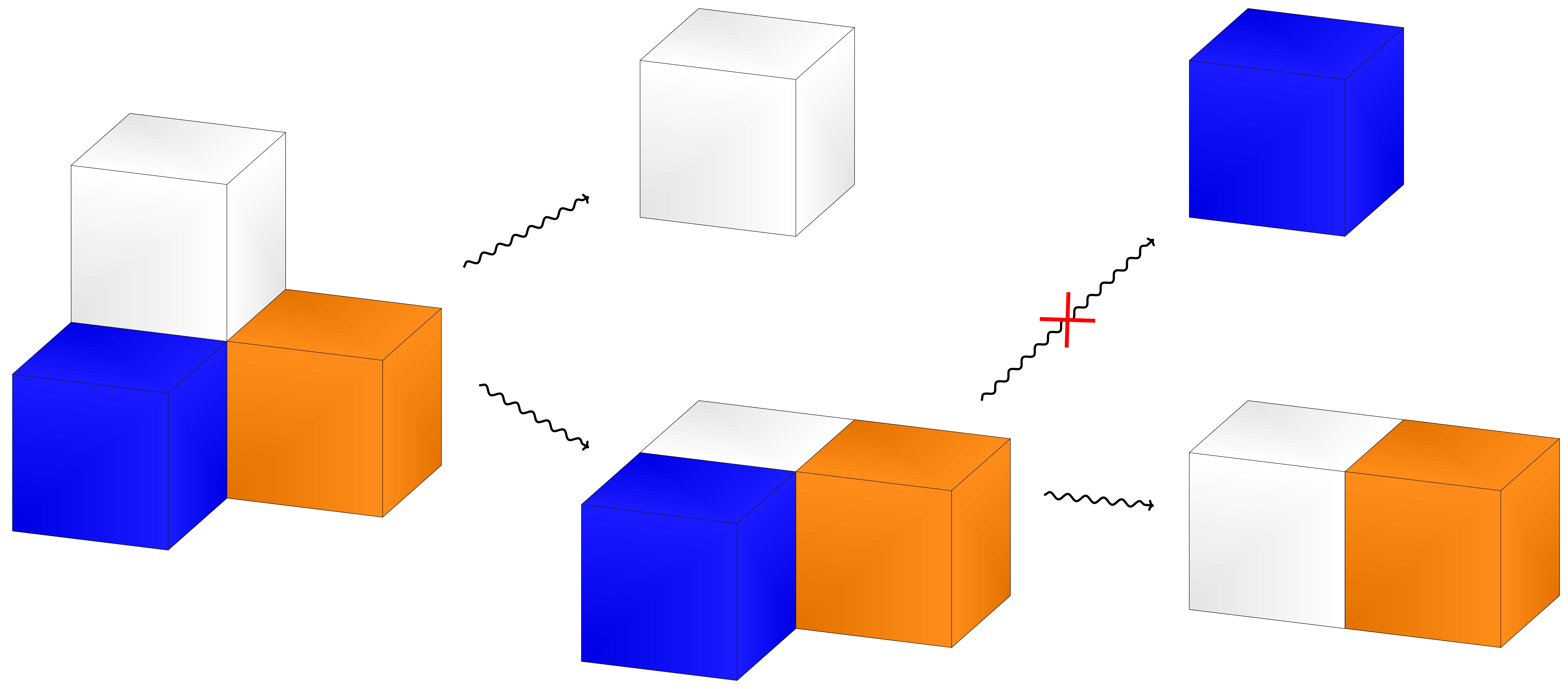}
    \caption{Possible splittings of the class $(2,1,1)$ in terms of decompositions of colored plane partitions. A splitting into invalid colored plane partitions is not a splitting in $I$.}
    \label{fig:factorization_index_set}
\end{figure}

\begin{example}
    If we consider colored plane partitions by studying an orbifold $[\C^3/G]$, then color vectors of plane partitions correspond to K-theory classes, and $I$ will contain all color vectors which occur in this setup. For example the box $(0,0,0)$ is always colored with the color $0$, so the vector $(1,0,0,\dots)$ is in $I$, but the vector $(0,1,0,0,\dots)$ is not. The restrictions on splittings of K-theory classes imposed by the factorization index set from Lemma \ref{fact:lemma:hilb_fact_index_set} are visualized in Figure \ref{fig:factorization_index_set} above. The class $(1,1,1)$ can not split off a class $(0,1,0)$ in $I$, because, while this is the class of a $0$-dimensional sheaf on $\Xc=[\C^2/\mu_3]\times\C$, it is not the structure sheaf of a $0$-dimensional substack. This is reflected by the associated box not being a valid colored plane partition when split off.
\end{example}

\subsubsection{General Factorizable Systems}\label{fact:sec:gen_fact_sys}
Now we give a general definition of a factorizable system. In practice, we only use this definition in the various ways explained in Section \ref{fact:sec:fact_examples}. Factorizable systems in the below definitions consist of (possibly 2-periodic) complexes of coherent sheaves, i.e. elements of $\CoCh(\Coh(-))$ or $\CoCh_{\Z/2}(\Coh(-))$.

Given complexes of sheaves $\Fc_1,\dots,\Fc_k$ on any scheme, an operation $\odot\in\left\{\otimes,\oplus\right\}$, and a permutation $\sigma\in S_k$, we denote by
\begin{equation*}
    S: \Fc_1\odot\cdots\odot\Fc_k \xrightarrow{\sim}\Fc_{\sigma(1)}\odot\cdots\odot\Fc_{\sigma(k)} 
\end{equation*}
the standard isomorphism.

\begin{definition}\label{def:factorizable_system_general}
    Consider the following data: 
    \begin{itemize}
        \item Let $I$ be a factorization index set.
        \item Let $\indgffct:I\to \indgfset$ be an additive morphism to a subset $\indgfset$ of a monoid, in the sense that, for any $\alpha_1+\alpha_2=\alpha$ in $I$, $\indgffct(\alpha_1)+\indgffct(\alpha_2)$ equals $\indgffct(\alpha)$ and is in particular in $\indgfset$.
        \item For any $\indgfel\in\indgfset$ we have a scheme $M_\indgfel(\Xc)$ with an action of a finite group $G_\indgfel$.
        \item For any $\indgfel_1+\indgfel_2=\indgfel$ in $\indgfset$, we have an embedding $G_{\indgfel_1}\times G_{\indgfel_2}\subseteq G_{\indgfel}$.
        \item For any $\indgfel_1+\indgfel_2=\indgfel$ in $\indgfset$, we have $G_{\indgfel_1}\times G_{\indgfel_2}$-equivariant open subschemes $U_{\indgfel_1\indgfel_2}\subseteq M_{\indgfel_1}(\Xc)\times M_{\indgfel_2}(\Xc)$. Taking $\sum_{i=1}^r \indgfel_i=\indgfel$ in $\indgfset$, set $U_{\indgfel_1\dots\indgfel_r}\subseteq \prod_i M_{\indgfel_i}(\Xc)$ to be the intersection of the pullbacks of all the $U_{\indgfel_i\indgfel_j}$. Then we assume that there are $G_{\indgfel_1}\times \cdots \times G_{\indgfel_r}$-equivariant flat morphisms $U_{\indgfel_1\dots\indgfel_r}\to M_{\indgfel}(\Xc)$ and for any permutation $\tau$ natural isomorphisms $\tau:U_{\indgfel_1\dots\indgfel_r}\xrightarrow{\sim}U_{\indgfel_{\tau_1}\dots\indgfel_{\tau_r}}$ making the diagram
        \begin{equation}\label{fact:eq:U_switching_comm_diag}
            \begin{tikzcd}
                U_{\indgfel_1\dots\indgfel_r} \arrow[dr] \arrow[r, "\tau"] & U_{\indgfel_{\tau_1}\dots\indgfel_{\tau_r}} \arrow[d]\\
                & M_{\indgfel}(\Xc)
            \end{tikzcd}
        \end{equation}
        commute.
        \item Fix $\boxdot$ to be either $\boxtimes$ or $\boxplus$.
    \end{itemize} 
    
    With this data, a collection $\{\Fc_{\alpha}\}_{\alpha\in I}$ of (possibly 2-periodic) complexes of sheaves on $M_{\indgffct(\alpha)}(\Xc)$ is called \textbf{factorizable} with respect to this data if there exists a collection of morphisms of $G_{\indgffct(\alpha_1)}\times G_{\indgffct(\alpha_2)}$-equivariant complexes of sheaves for $\alpha_1+\alpha_2=\alpha$ in $I$
    \begin{equation*}
        \phi_{\alpha_1\alpha_2}:\restr{\left(\Fc_{\alpha_1}\boxdot \Fc_{\alpha_2}\right)}{U_{\indgffct(\alpha_1)\indgffct(\alpha_2)}}\rightarrow \restr{\Fc_{\alpha}}{U_{\indgffct(\alpha_1)\indgffct(\alpha_2)}},
    \end{equation*}
    such that for any $\alpha\in I$ and $\indgfel_1,\indgfel_2\in \indgfset$ with $\indgfel_1+\indgfel_2=\indgffct(\alpha)$, the morphism
    \begin{equation}\label{fact:eq:sum_isom_fact}
        \bigoplus_{\substack{\alpha_1+\alpha_2=\alpha\\\indgffct(\alpha_i)=\indgfel_i}} \phi_{\alpha_1\alpha_2}: \bigoplus_{\substack{\alpha_1+\alpha_2=\alpha\\\indgffct(\alpha_i)=\indgfel_i}} \restr{\left(\Fc_{\alpha_1}\boxdot \Fc_{\alpha_2}\right)}{U_{\indgfel_1\indgfel_2}}\rightarrow \restr{\Fc_{\alpha}}{U_{\indgfel_1\indgfel_2}}
    \end{equation}
    is an isomorphism of $G_{\indgfel_1}\times G_{\indgfel_2}$-equivariant complexes of sheaves. Moreover, we require the $\phi_{\alpha_1\alpha_2}$ to satisfy the following requirements.

    \begin{itemize}
        \item Associativity: Given $\alpha_1,\alpha_2,\alpha_3\in I$, such that all sums of them are in $I$, we have
        \begin{equation*}
            \phi_{\alpha_1+\alpha_2,\alpha_3} \circ \left(\phi_{\alpha_1\alpha_2}\boxdot \id_{\Fc_{\alpha_3}}\right) = \phi_{\alpha_1,\alpha_2+\alpha_3} \circ \left(\id_{\Fc_{\alpha_1}}\boxdot \phi_{\alpha_2\alpha_3}\right)
        \end{equation*}
        as morphisms
        \begin{equation*}
            \restr{\left(\Fc_{\alpha_1}\boxdot \Fc_{\alpha_2}\boxdot \Fc_{\alpha_3}\right)}{U_{\indgffct(\alpha_1)\indgffct(\alpha_2)\indgffct(\alpha_3)}} \rightarrow \restr{\Fc_{\alpha_1+\alpha_2+\alpha_3}}{U_{\indgffct(\alpha_1)\indgffct(\alpha_2)\indgffct(\alpha_3)}} .
        \end{equation*}
        \item Commutativity: Because the diagram \eqref{fact:eq:U_switching_comm_diag} commutes (here we use $r=2$), we get an isomorphism $v:\tau^*\left(\restr{(-)}{U_{\indgfel_2\indgfel_1}}\right)\cong \restr{(-)}{U_{\indgfel_1\indgfel_2}}$. We require that the diagram
        \begin{equation}\label{fact:eq:gen_comm_diagram}
            \begin{tikzcd}
                \restr{\left(\Fc_{\alpha_1}\boxdot \Fc_{\alpha_2}\right)}{U_{\indgffct(\alpha_1)\indgffct(\alpha_2)}} \arrow[d] \arrow[r, "\phi_{\alpha_1\alpha_2}"] & \restr{\Fc_{\alpha_1+\alpha_2}}{U_{\indgffct(\alpha_1)\indgffct(\alpha_2)}} \arrow[d, "v"]\\
                \tau^*\left(\restr{\left(\Fc_{\alpha_2}\boxdot \Fc_{\alpha_1}\right)}{U_{\indgffct(\alpha_2)\indgffct(\alpha_1)}}\right) \arrow[r, "\tau^*\phi_{\alpha_2\alpha_1}"] & \tau^*\left(\restr{\Fc_{\alpha_1+\alpha_2}}{U_{\indgffct(\alpha_2)\indgffct(\alpha_1)}}\right)
            \end{tikzcd}
        \end{equation} 
        commutes, where the first vertical arrow arises by composition of the natural isomorphisms $S$ and $v$.
    \end{itemize}
\end{definition}

\subsubsection{Examples}\label{fact:sec:fact_examples}

We explain the various examples of Definition \ref{def:factorizable_system_general}, which we use in this paper. For applications, we consider only the factorization index set given by Lemma \ref{fact:lemma:hilb_fact_index_set}, but for the following examples, the specific choice of factorization index set is unimportant.

Note that virtual structure sheaves, as defined in Section \ref{sec:vir_str_sh} are elements of $\Db_{\Z/2}(\Coh(-))$ of the form $\bigoplus_i \Hc^i(\Ec)[i]$ for some $\Ec\in \Db_{\Z/2}(\Coh(-))$. For the purpose of factorization, we can view them as genuine complexes $\bigoplus_i \Hc^i(\Ec)[i]$ with trivial differential and all factorization morphisms from Section \ref{sec:fact_vir_str_sheaf} are well-defined morphisms of complexes.

\begin{remark}
    More generally, we have a functor
    \begin{equation*}
        \Db_{\Z/2}(\Coh(-)) \ni \Ec \mapsto \bigoplus_i \Hc^i(\Ec)[i] \in \CoCh_{\Z/2}(\Coh(-)),
    \end{equation*}
    which preserves the class of $\Ec$ in $K(\Coh(-))$. We can then define factorizable systems of objects in $\Db_{\Z/2}(\Coh(-))$ in an analogous way and use the above functor to prove the main result for such systems.
\end{remark}

\begin{example}\label{fact:ex:0_dim_sheaves}
    We consider factorization for moduli $M_\alpha(\Xc)$ of $0$-dimensional sheaves from Section \ref{setup:sec:moduli_0_dim_sheaves}. Here, we take $\indgffct=\id$, $G_\alpha$ the trivial group, and the operation $\boxdot=\boxtimes$. The open subsets $U_{\alpha_1\dots\alpha_r}$ in $M_{\alpha_1}(\Xc)\times\cdot\times M_{\alpha_r}(\Xc)$ contain the points $\left(\Fc_1,\dots,\Fc_r\right)$, where the $\Fc_i$ have pairwise disjoint support. The flat morphisms are induced by the natural direct sum morphism
    \begin{equation*}
        \prod_{i=1}^{k}M_{\alpha_i}(\Xc) \to M_{\sum_i\alpha_i}(\Xc).
    \end{equation*}
\end{example}

\begin{example}\label{fact:ex:hilb}
    To examine factorization properties of obstruction theories and virtual structure sheaves, we consider factorization for Hilbert schemes of points $\hilb^\alpha(\Xc)$. We take $\indgffct=\id$, $G_\alpha$ the trivial group, and either operation $\boxdot\in\{\boxplus,\boxtimes\}$. The open subsets and flat morphisms are defined by fiber product, as shown in the commutative diagram
    \begin{equation*}
        \begin{tikzcd}
            \overline{U}_{\alpha_1 \alpha_2} \ar[r, open, "j"]\ar[rr, bend left, open, "j'"]\ar[d] & \hilb^{\alpha_1}(\Xc)\times \hilb^{\alpha_1}(\Xc) \ar[d] & \hilb^{\alpha_1+\alpha_2}(\Xc) \ar[d] \\
            U_{\alpha_1 \alpha_2} \ar[r, open] & M_{\alpha_1}(\Xc)\times M_{\alpha_2}(\Xc) \ar[r] & M_{\alpha_1+\alpha_2}(\Xc).
        \end{tikzcd}
    \end{equation*}
    with cartesian squares.
\end{example}

Passing to the coarse space along $\pi:\Xc\to X$, we need to take $\indgffct=b$ from Section \ref{setup:sec:coarse}.

\begin{example}\label{fact:ex:sym_coarse}
    We also consider factorization on $\sym^{b}(X)$, where we take $\indgffct=b$, $G_b$ the trivial group, and the operation $\boxdot=\boxtimes$. Viewing $\sym^{b}(X)$ as a moduli space of $0$-dimensional sheaves on $X$, following \cite[Ex. 4.3.6]{hl10}, the open subsets $U_{b_1\dots b_r}$ in $\sym^{b_1}(X)\times\cdot\times \sym^{b_r}(X)$ contain the points $\left(\Fc_1,\dots,\Fc_r\right)$, where the $\Fc_i$ have pairwise disjoint support. As in Example \ref{fact:ex:0_dim_sheaves}, the flat morphisms to $\sym^{b_1+\cdots+b_r}$ come from the direct sum morphism.
\end{example}

\begin{example}\label{fact:ex:product_coarse}
    For the proof of the main factorization theorem below, $S_b$-equivariant factorization on $X^b$ plays an important role. Here, we take $\indgffct=b$, $G_b=S_b$ the symmetric group. The open subsets $U_{b_1b_2}$ of $X^{b_1}\times X^{b_2}$ exactly contain the points $\left(y_1,\dots,y_{b_1},z_1,\dots,z_{b_2}\right)$, where the two subsets of points $(y_1,\dots,y_{b_1})$ and $(z_1,\dots,z_{b_2})$ are disjoint. The flat morphisms are induced by the identity $X^{b_1}\times X^{b_2}=X^{b_1+b_2}$.
\end{example}

\subsubsection{Pushforward \& Pullback Compatibility}\label{fact:sec:fact_functoriality}

We establish that, under some hypotheses, factorizable systems are preserved under pushforward and pullback.

\begin{lemma}\label{fact:lemma:comp_setup_cd}
    Let $I$ be a factorization index set and $\boxdot\in\{\boxplus,\boxtimes\}$. Let $\indgffct:I\to\indgfset$, $M_\indgfel(\Xc)$, $G_\indgfel$, $U_{\indgfel_1\indgfel_2}$ and $\indgffct':I\to\indgfset'$, $N_{\indgfel'}(\Xc)$, $G'_{\indgfel'}$, $U_{\indgfel'_1\indgfel'_2}$ be two collections of factorization data. Let $f:\indgfset\to \indgfset'$ be a surjective additive morphism such that $f\circ \indgffct=\indgffct'$. Let $f^G_\indgfel:G_\indgfel\to G'_{f(\indgfel)}$ be a group morphism and $f^M_\indgfel:M_\indgfel(\Xc)\to N_{f(\indgfel)}(\Xc)$ be a morphism that is equivariant with respect to $f^G_\indgfel$. Assume for every $\indgfel_1+\indgfel_2=\indgfel$ in $\indgfset$ mapping to $\indgfel'_1+\indgfel'_2=\indgfel'$ in $\indgfset'$ under $f$, the open subschemes and flat morphisms are given by fiber product as follows
    \begin{equation}\label{fact:eq:functoriality_f_diagram}
        \begin{tikzcd}[row sep=large]
            U_{\indgfel_1 \indgfel_2} \ar[rr, hookrightarrow, "i_{\indgfel_1 \indgfel_2}"]\ar[drr]\ar[d, open']\ar[dr, swap, "f^U_{\indgfel_1 \indgfel_2}"] & & \overline{U}_{\indgfel'_1\indgfel'_2} \ar[dl, swap, pos=0.35, "f^{\overline{U}}_{\indgfel'_1\indgfel'_2}"]\ar[d] \\
            M_{\indgfel_1}(\Xc)\times M_{\indgfel_2}(\Xc) \ar[d,swap,"f^M_{\indgfel_1}\times f^M_{\indgfel_2}"] & U_{\indgfel'_1\indgfel'_2} \ar[dr]\ar[dl, open'] & M_{\indgfel}(\Xc) \ar[d, "f^M_{\indgfel}"]\\
            N_{\indgfel'_1}(\Xc)\times N_{\indgfel'_2}(\Xc) & & N_{\indgfel'}(\Xc),
        \end{tikzcd}
    \end{equation}
    where the left- and right-most squares are cartesian, $i_{\indgfel_1 \indgfel_2}$ is an embedding of connected components, and all $i_{\indgfel_1 \indgfel_2}$ jointly cover $\overline{U}_{\indgfel'_1\indgfel'_2}$. Then
    \begin{enumerate}[(a)]
        \item Take $\boxdot=\boxtimes$. If $\left\{\Fc_\alpha\right\}_{\alpha\in I}$ is a $G_{\indgffct(\alpha)}$-equivariant factorizable system of sheaves on $M_{\indgffct(\alpha)}(\Xc)$, and all $\left(f^M_{\indgffct(\alpha)}\right)_*\Fc_\alpha$ are complexes of coherent sheaves, then $\left\{\left(f^M_{\indgffct(\alpha)}\right)_*\Fc_\alpha\right\}_{\alpha\in I}$ is a $G'_{f(\indgffct(\alpha))}$-equivariant factorizable system of sheaves on $N_{f(\indgffct(\alpha))}(\Xc)$.
        \item Take $f=\id$, so that $\indgffct=\indgffct'$ and $i_{\indgfel_1\indgfel_2}=\id$ by the condition on $i_{\indgfel_1\indgfel_2}$ above. If $\left\{\Fc_\alpha\right\}_{\alpha\in I}$ is a $G'_{\indgffct(\alpha)}$-equivariant factorizable system of sheaves on $N_{\indgffct(\alpha)}(X)$, then $\left\{\left(f^M_{\indgffct(\alpha)}\right)^*\Fc_\alpha\right\}_{\alpha\in I}$ is a $G_{\indgffct(\alpha)}$-equivariant factorizable system of sheaves on $M_{\indgffct(\alpha)}(\Xc)$.
    \end{enumerate}
\end{lemma}
\begin{proof}
    The proofs of both statements are analogous, so here, we only show the more complicated case of pushing forward with $\boxdot=\boxtimes$. For the operation $\boxtimes$, we have $\left(f^M_{\indgfel_1}\times f^M_{\indgfel_2}\right)_*$ commutes with $\boxtimes$ as follows
    \begin{equation}\label{fact:eq:boxdot_push_commute}
        \left(f^M_{\indgfel_1}\times f^M_{\indgfel_2}\right)_*\left(-\boxtimes -\right) \cong \left(f^M_{\indgfel_1}\right)_*(-)\boxtimes\left(f^M_{\indgfel_2}\right)_*(-),
    \end{equation}
    by using the pullbacks in the definition of $\boxtimes$.

    As pushforward preserves isomorphisms and commutative diagrams, the only statement to check is how pushing forward along $f^M_{\indgffct(\alpha)}$ interacts with the direct sum splitting in \eqref{fact:eq:sum_isom_fact}. Consider the base change $f^U_{\indgfel_1\indgfel_2}$ of the morphism $f^M_{\indgfel_1}\times f^M_{\indgfel_2}$. We push the factorization morphisms $\phi_{\alpha_1\alpha_2}$ forward along $f^U_{\indgffct(\alpha_1)\indgffct(\alpha_2)}$. Take a given $\alpha$ and $\indgfel'_1+\indgfel'_2=\indgfel'$ in $\indgfset'$ with $\indgfel'=\indgffct'(\alpha)$. Take $\indgfel_i$ such that $f(\indgfel_i)=\indgfel'_i$. By assumption, on the components $U_{\indgfel_1\indgfel_2}$, we have the isomorphisms  
    \begin{equation*}
        \bigoplus_{\substack{\alpha_1+\alpha_2=\alpha\\\indgffct(\alpha_i)=\indgfel_i}} \phi_{\alpha_1\alpha_2}: \bigoplus_{\substack{\alpha_1+\alpha_2=\alpha\\\indgffct(\alpha_i)=\indgfel_i}} \restr{\left(\Fc_{\alpha_1}\boxtimes \Fc_{\alpha_2}\right)}{U_{\indgfel_1\indgfel_2}} \rightarrow \restr{\Fc_{\alpha}}{U_{\indgfel_1\indgfel_2}}.
    \end{equation*}
    Now sum these isomorphisms over the various components $U_{\indgfel_1 \indgfel_2}$ of $\overline{U}_{\indgfel'_1\indgfel'_2}$, for any splitting $\indgfel_1+\indgfel_2=\indgfel$ with $f(\indgfel_i)=\indgfel'_i$. We get isomorphisms
    \begin{equation*}
        \bigoplus_{\substack{\indgfel_1+\indgfel_2=\indgfel\\ f(\indgfel_i)=\indgfel'_i}}\bigoplus_{\substack{\alpha_1+\alpha_2=\alpha\\\indgffct(\alpha_i)=\indgfel_i}} i_{\indgfel_1 \indgfel_2 *}\phi_{\alpha_1 \alpha_2}: \bigoplus_{\substack{\indgfel_1+\indgfel_2=\indgfel\\ f(\indgfel_i)=\indgfel'_i}}\bigoplus_{\substack{\alpha_1+\alpha_2=\alpha\\\indgffct(\alpha_i)=\indgfel_i}} i_{\indgfel_1 \indgfel_2 *}\left(\restr{\left(\Fc_{\alpha_1}\boxtimes \Fc_{\alpha_2}\right)}{U_{\indgfel_1 \indgfel_2}}\right) \rightarrow \restr{\Fc_{\alpha}}{\overline{U}_{\indgfel'_1\indgfel'_2}}.
    \end{equation*}
    Reorder the double direct sum to a direct sum over $\alpha_i$ with $\indgffct(\alpha_i)=\indgfel_i$, and use the uniquely determined $\indgfel_i=\indgffct(\alpha_i)$. We also push forward along $f^{\overline{U}}_{\indgfel'_1\indgfel'_2}$ to obtain isomorphisms
    \begin{equation*}
        \bigoplus_{\substack{\alpha_1+\alpha_2=\alpha\\\indgffct'(\alpha_i)=\indgfel'_i}} \left(f^{U}_{\indgffct(\alpha_1)\indgffct(\alpha_2)}\right)_*\phi_{\alpha_1 \alpha_2}
    \end{equation*}
    between
    \begin{equation*}
        \bigoplus_{\substack{\alpha_1+\alpha_2=\alpha\\\indgffct'(\alpha_i)=\indgfel'_i}}\restr{\left( \left(f^M_{\indgffct(\alpha_1)}\times f^M_{\indgffct(\alpha_2)}\right)_*\left(\Fc_{\alpha_1}\boxtimes \Fc_{\alpha_2}\right)\right)}{U_{\indgfel'_1 \indgfel'_2}} \rightarrow \restr{\left(\left(f^M_{\indgffct(\alpha)}\right)_*\Fc_{\alpha}\right)}{U_{\indgfel'_1\indgfel'_2}},
    \end{equation*}
    where we use the push-pull formula along the cartesian squares above. Using the canonical isomorphism \eqref{fact:eq:boxdot_push_commute} on the left-hand side yields the desired factorization isomorphisms.
\end{proof}

We use this Lemma for pushing forward and pulling back along the morphisms in the diagram (\ref{fact:cd:setups}) as follows.

\begin{example}\label{fact:ex:hilb_to_sym}
    We first consider the morphism
    \begin{equation*}
        \eta_\alpha\coloneqq \xi_\alpha\circ \HC_\alpha: \hilb^\alpha(\Xc)\to\sym^{b(\alpha)}(X).
    \end{equation*}
    A quasi-projective orbifold (in our sense) $\Xc$ admits an open embedding into a projective finite type DM-stack $\overline{\Xc}$ by \cite[Sec. 5]{kresch_geom_dm}, which in particular has projective coarse moduli space $\overline{X}$. As coarse moduli are Zariski-local, we get a cartesian diagram
    \begin{equation*}
        \begin{tikzcd}
            \Xc \ar[r,open]\ar[d] & \overline{\Xc}\ar[d]\\
            X \ar[r,open] & \overline{X},
        \end{tikzcd}
    \end{equation*}
    which in turn yields the cartesian diagram
    \begin{equation*}
        \begin{tikzcd}
            \hilb(\Xc) \ar[r,open]\ar[d,"\eta"] & \hilb\left(\overline{\Xc}\right)\ar[d]\\
            \sym(X) \ar[r,open] & \sym\left(\overline{X}\right).
        \end{tikzcd}
    \end{equation*}
    By \cite[Theorem 1.5]{os08_quot_dm}, the Hilbert scheme of points $\hilb\left(\overline{\Xc}\right)$ is projective, making $\eta$ proper.

    Hence, $\eta_\alpha$ preserves coherence and $\hilb^\alpha(\Xc)$ and $\sym^{b(\alpha)}(X)$ come with factorization data from Examples \ref{fact:ex:hilb} and \ref{fact:ex:sym_coarse} respectively. 
    
    Taking $f=b$, we argue that the conditions of Lemma \ref{fact:lemma:comp_setup_cd} are satisified, so that we may push forward factorizable systems from $\hilb^\alpha(\Xc)$ to $\sym^{b(\alpha)}(X)$. The only thing to check are the conditions on the open subschemes in diagram \eqref{fact:eq:functoriality_f_diagram}. By Example \ref{fact:ex:hilb} the open subschemes for the Hilbert scheme are defined via pullback from the ones for moduli of $0$-dimensional sheaves in Example \ref{fact:ex:0_dim_sheaves}, so we may consider these.

    Recall the fixed isomorphism $\Nr_0(\Xc)\cong \Z\oplus\widetilde{\Nr}_0(\Xc)$. Write $\alpha=(b(\alpha),\tilde{\alpha})$ for classes on $\Xc$, and consider factorizable systems on the moduli $M_{(b,\tilde{\alpha})}(\Xc)$ as in Examples \ref{fact:ex:0_dim_sheaves}. For given $\alpha\in I$ and $b_1,b_2$, the loci $U_{(b_1,\tilde{\alpha}_1),(b_2,\tilde{\alpha}_2)}$ in $M_{(b_1,\tilde{\alpha}_1)}(\Xc)\times M_{(b_2,\tilde{\alpha}_2)}(\Xc)$ are exactly the subschemes of two $0$-dimensional sheaves with disjoint support of classes $(b_i,\tilde{\alpha}_i)$. The coarse moduli space $\pi:\Xc \to X$ is bijective on geometric points, so the supports of two $0$-dimensional sheaves on $\Xc$ is disjoint if and only if the supports of their pushforwards to $X$ are disjoint. Hence, $U_{(b_1,\tilde{\alpha}_1),(b_2,\tilde{\alpha}_2)}$ are exactly the fiber product in the left cartesian square in \eqref{fact:eq:functoriality_f_diagram}. Analogously, the fiber product in right cartesian square in \eqref{fact:eq:functoriality_f_diagram} is the filled in by the disjoint union of the $U_{(b_1,\tilde{\alpha}_1),(b_2,\tilde{\alpha}_2)}$, varying $\tilde{\alpha}_1$ and $\tilde{\alpha}_2$, as required.
\end{example}

\begin{example}\label{fact:ex:sym_to_Xb}
    With $f=\id$, the quotient morphism $X^{b(\alpha)}\to\sym^{b(\alpha)}$ satisfies the conditions of Lemma \ref{fact:lemma:comp_setup_cd}(b), so that we may pull back factorizable systems from $\sym^{b(\alpha)}$ to $S_{b(\alpha)}$-equivariant factorizable systems on $X^{b(\alpha)}$.
\end{example}

\subsubsection{Extension to More General Morphisms}

We note here, that the above theory of factorization extends to the following setting. Take a morphism
\begin{equation*}
    f:\Yc \to X,
\end{equation*}
from an orbifold $\Yc$ to a connected scheme $X$. Then $\Nr_0(X)\cong \Z$. Take
\begin{equation*}
    \Nr_\exc(\Yc)\coloneqq \left\{\alpha\in\Nr(\Yc)\ |\ f_*(\alpha)\in \Nr_0(X)\right\}.
\end{equation*}
We assume there is a splitting of $f_*:\Nr_\exc(\Yc)\to \Nr_0(X)$, such that
\begin{equation*}
    \Nr_\exc(\Yc) \cong \Z\oplus \widetilde{\Nr}_\exc(\Yc),
\end{equation*}
where the first projection is just $f_*$ and is denoted by $b(-)$. Then we can analogously define a factorization index set to be a subset $I\subseteq \Cr_\exc(\Yc)$ satisfying the same two conditions of Definition \ref{fact:def:fact_ind_set}. Similarly, we can define factorizable systems on $\sym^{b(\alpha)}(X)$ and on $\left[X^{b(\alpha)}/S_{b(\alpha)}\right]$ as in Examples \ref{fact:ex:sym_coarse} and \ref{fact:ex:product_coarse}.

As the proof of the main factorization Theorem \ref{fact:thm:fact} below works entirely on $\sym^{b(\alpha)}(X)$ and on $\left[X^{b(\alpha)}/S_{b(\alpha)}\right]$, we get an analogous result in this more general setting. For example, this applies to the following interesting situations:
\begin{itemize}
    \item For $f$ the coarse moduli space of an orbifold, this recovers the above notion of factorization for orbifolds.
    \item For $f$ the crepant resolution of an orbifold, this allows us to factorize generating series indexed by exceptional classes.
    \item For $f$ a fibration, this allows us to factorize generating series indexed by fiber classes.
\end{itemize}
We plan to explore these applications in future work.

\subsection{Plethystic Exponentials \& Generating Series}

\subsubsection{Plethystic Exponentials}\label{fact:sec:pexp}

First, we define two notions of a partition of $\alpha$, which will come up in the statement and proof of Theorem \ref{fact:thm:fact}. One should think of the difference between partitions of an integer $n$ and the set $[n]$.

\begin{definition}
    Given a factorization index set $I$, an $I$-\textbf{partition} of $\alpha\in I$ is a multiset $\lambda=\{\alpha_1,\dots,\alpha_k\}\subset I$, such that $\sum_i \alpha_i = \alpha$. We denote the set of $I$-partitions of $\alpha$ by $P_I(\alpha)$ or just $P(\alpha)$ if the factorization index set is clear. The length of a partition $\lambda=\{\alpha_1,\dots,\alpha_k\}$ is $k$. We denote the set of $I$-partitions of $\alpha$ of length $k$ by $P_{I,k}(\alpha)$.
\end{definition}

\begin{definition}
    Given $\alpha$, denote by $[\alpha]$ the tuple $([b(\alpha)],\tilde{\alpha})$. Given a factorization index set $I$, an $I$-\textbf{partition} of $[\alpha]$ is a collection
    \begin{equation*}
        \{(B_1,\tilde{\alpha}_1), \dots, (B_k,\tilde{\alpha}_k)\},
    \end{equation*}
    where $B_i\subseteq [b(\alpha)]$, such that $\{B_1,\dots,B_k\}\in P[b(\alpha)]$, and $\tilde{\alpha}_i\in \widetilde{\Nr}_0(\Xc)$ classes such that the collection $\{(\abs{B_1},\tilde{\alpha}_1),\dots,(\abs{B_k},\tilde{\alpha}_k)\}$ is an $I$-partition of $\alpha$ in the above sense.
    
    We denote the set of $I$-partitions of $[\alpha]$ by $P_I[\alpha]$ or just $P[\alpha]$ if the factorization index set is clear. The length of a partition $\Xi=\{(B_1,\tilde{\alpha}_1), \dots, (B_k,\tilde{\alpha}_k)\}$ is $k$.

    Given an element $\Xi=\{(B_1,\tilde{\alpha}_1), \dots, (B_k,\tilde{\alpha}_k)\}$ of $P_I[\alpha]$, we denote by $B(\Xi)$ the underlying partition $\{B_1, \dots, B_k\}$ of $[b(\alpha)]$.

    Note that an element $\Xi=\{(B_1,\tilde{\alpha}_1), \dots, (B_k,\tilde{\alpha}_k)\}$ of $P_I[\alpha]$ yields an element $\abs{\Xi}=\{(\abs{B_1},\tilde{\alpha}_1),\dots,(\abs{B_k},\tilde{\alpha}_k)\}$ of $P_I(\alpha)$ by definition. We write $P_I^\lambda[\alpha]=\{\Xi\in P_I[\alpha]\ |\ \abs{\Xi}=\lambda\}$.
\end{definition}

Now we can define a $K$-theoretic version of the plethystic exponential. Note that throughout this paper, we consider the K-theory class of a (possibly 2-periodic) complex of sheaves to be the class in $K(\Coh(-))$, rather than one in $K(\CoCh^{(\Z/2)}\Coh(-))$.

\begin{definition}\label{fact:def:G_partition_sheaf}
    Let $\Xi=\{(B_1,\tilde{\alpha}_1), \dots, (B_k,\tilde{\alpha}_k)\}$ be a partition of $[\alpha]$. For each $\beta\in I$, let $\Gc_{\beta}$ be an $S_{b(\beta)}$-equivariant complex of sheaves on $X^{b(\beta)}$. Assume $\Xi$ is ordered such that $\min(B_1) \leq \dots \leq \min(B_k)$. We define
    \begin{equation*}
        \Gc_{\Xi} := \Gc_{(\abs{B_1},\tilde{\alpha}_1)} \boxtimes \dots \boxtimes \Gc_{(\abs{B_k},\tilde{\alpha}_k)},
    \end{equation*}
    as a complex of sheaves on $X^{b(\alpha)}$, where each pullback is along the map $\pr_{b_i^1,\dots,b_i^{l_i}}:X^{b(\alpha)} \to X^{\abs{B_i}}$ specified by the subset $B_i=\{b_i^1,\dots,b_i^{l_i}\}$, where we order $B_i$ such that $b_i^1<\dots <b_i^{l_i}$.

    For a partition $\lambda\in P(\alpha)$ of $\alpha$, we define the complex of sheaves on $X^{b(\alpha)}$
    \begin{equation*}
        \Gc_\lambda := \bigoplus_{\Xi\in P^\lambda[\alpha]} \Gc_{\Xi},
    \end{equation*}
    which inherits a natural $S_{b(\alpha)}$-equivariant structure, as $S_{b(\alpha)}$ permutes the various projections $X^{b(\alpha)} \to X^{\abs{B_i}}$ by permuting the partitions $\{B_i\}_i$ of $[b(\alpha)]$ realizing the partition $\{\abs{B_i}\}_i$ of $b(\alpha)$. For compatibility with constructions in Section \ref{fact:sec:combinatorics}, the equivariant structure involves an additional sign. We will explain this sign and details about the equivariant structure in more generality in Definition \ref{fact:def:partition_sheaf_via_tree}.
\end{definition}

Analogously, we make the following definition on the symmetric product.

\begin{definition}\label{fact:def:lambda_eq_sheaf_sym}
    For each $\beta\in I$, let $\Gc_{\beta}$ be a complex of sheaves on $\sym^{b(\alpha)}(X)$. Let $\lambda=\{\alpha_1,\dots,\alpha_k\}$ in $P(\alpha)$ be a partition of $\alpha$. Take a partition $\Xi=\{(B_1,\tilde{\alpha}_1), \dots, (B_k,\tilde{\alpha}_k)\}$ in $P^\lambda[\alpha]$. Assume $\Xi$ is ordered such that $\min(B_1) \leq \dots \leq \min(B_k)$. Set
    \begin{equation*}
        \Gc_{\Xi} := \Gc_{(\abs{B_1},\tilde{\alpha}_1)} \boxtimes \dots \boxtimes \Gc_{(\abs{B_k},\tilde{\alpha}_k)},
    \end{equation*}
    as a complex of sheaves on $\prod_{i=1}^k\sym^{\abs{B_i}}(X)$. Take $S^B_\lambda$ to be the group $\prod_i S_{l_i}$, where $l_i$ is the number of occurrences of $\alpha_i$ in $\lambda$. It acts on $\prod_{i=1}^k\sym^{\abs{B_i}}(X)$ by permuting the blocks $\sym^{\abs{B_i}}(X)$ with the same $\alpha_i$. $\Gc_{\Xi}$ carries a natural $S^B_\lambda$-equivariant structure. We have the composition of morphisms
    \begin{equation*}
        \vartheta: \left[\left(\prod_{i=1}^k\sym^{\abs{B_i}}(X)\right)/S^B_\lambda\right] \to \prod_{\alpha_i\in \lambda} \sym^{l_i}\left(\sym^{b(\alpha_i)}(X)\right) \to \sym^{b(\alpha)}(X),
    \end{equation*}
    where the product in the second term goes over every $\alpha_i$ in $\lambda$ viewed as a set instead of a multiset, the first morphism is the coarse space, and the second morphism adds up the different points in $X$. Set
    \begin{equation*}
        \Gc_\lambda \coloneqq \vartheta_*\left(\Gc_\Xi\right).
    \end{equation*}
    We will see in Lemma \ref{fact:lemma:lambda_pieces_push_pull_identities}(b) below that the construction of $\Gc_\lambda$ is independent of the choice of $\Xi$.
\end{definition}

We now show some important compatibilities and properties of the previous two constructions.

\begin{lemma}\label{fact:lemma:lambda_pieces_push_pull_identities}
    For each $\beta\in I$, let $\Gc_{\beta}$ be a complex of sheaves on $\sym^{b(\alpha)}(X)$, and let $\Ec_{\beta}$ be an $S_{b(\beta)}$-equivariant complex of sheaves on $X^{b(\beta)}$. Let $\lambda=\{\alpha_1,\dots,\alpha_k\}$ in $P(\alpha)$ be a partition of $\alpha$. By abuse of notation, denote the quotient morphism $[X^n/S_n]\to \sym^n(X)$ by $p$ for every $n$. Then
    \begin{enumerate}[(a)]
        \item $p_*\left(\Ec_\lambda\right)\cong \left(p_*\Ec\right)_\lambda$,
        \item $\Gc_\lambda\cong p_*\left((p^*\Gc)_\lambda\right)$,
        \item $p_*\left(\Ec_\lambda\right)\cong p_*\left((p^*p_*\Ec)_\lambda\right)$, and
        \item the operation $\left\{\Gc_\beta\right\}\mapsto \Gc_\lambda$ descends to K-theory.
    \end{enumerate}
\end{lemma}
\begin{proof}
    Note first that the coarse space $p$ satisfies that the adjunction $p_*p^*\cong \id$ is a natural isomorphism. Assuming (a), we show (b), (c), and (d). For (b), we see
    \begin{equation*}
        p_*\left((p^*\Gc)_\lambda\right) \cong \left(p_*p^*\Gc\right)_\lambda \cong \Gc_\lambda,
    \end{equation*}
    where we first use (a) and then the adjunction isomorphism. For (c), we see
    \begin{equation*}
        p_*\left((p^*p_*\Ec)_\lambda\right) \cong \left(p_*p^*p_*\Ec\right)_\lambda \cong \left(p_*\Ec\right)_\lambda \cong p_*\left(\Ec_\lambda\right),
    \end{equation*}
    where we first use (a), then the adjunction isomorphism, and then (a) again. For (d), note that the construction of $\Gc_\lambda$ in Definition \ref{fact:def:lambda_eq_sheaf_sym} involves only pullbacks along flat morphisms and a pushforward along $\vartheta$, which is the composition of a finite morphism and a proper coarse moduli space, making $\vartheta_*$ exact. Hence, the operation $\left\{\Gc_\alpha\right\}\mapsto \Gc_\lambda$ descends to K-theory.

    It remains to prove (a). Take $\Xi=\{(B_1,\tilde{\alpha}_1), \dots, (B_k,\tilde{\alpha}_k)\}$ in $P^\lambda[\alpha]$ and assume $\Xi$ is ordered such that $\min(B_1) \leq \dots \leq \min(B_k)$. Consider $\Ec_\Xi$ on $X^{b(\alpha)}$. It comes with a natural $\prod_i S_{b(\alpha_i)}$-equivariant structure, where $\prod_i S_{b(\alpha_i)}$ acts on $X^{b(\alpha)}$ by permuting the indices within the blocks indexed by the $B_i$. As in Definition \ref{fact:def:lambda_eq_sheaf_sym} $S^B_\lambda$ acts on $X^{b(\alpha)}$ by permuting entire blocks with indices $B_i$ and $B_j$ with $\alpha_i=\alpha_j$. Define the subgroup
    \begin{equation*}
        S_\Xi  
    \end{equation*}
    to be the subgroup of $S_{b(\alpha)}$ fixing the partition $\Xi$. $\Ec_\Xi$ comes with a natural $S_\Xi$-equivariant structure. We can identify this group as
    \begin{equation*}
        S_\Xi = \left(\textstyle\prod_i S_{b(\alpha_i)}\right)\rtimes S^B_\lambda.
    \end{equation*}
    By Proposition \ref{qst:prop:morphism}, the inclusion of $S_\Xi$ in $S_{b(\alpha)}$ gives us a morphism $q_\Xi:\left[X^{b(\alpha)}/S_\Xi\right]\to \left[X^{b(\alpha)}/S_{b(\alpha)}\right]$. We have a bijection $S_{b(\alpha)}/S_\Xi\xrightarrow{\sim}P^\lambda(\alpha)$ given by $[\sigma]\mapsto \sigma(\Xi)$, which together with Proposition \ref{qst:prop:induced_eq_str} gives us an isomorphism
    \begin{equation}\label{fact:eq:quotient_pushforward_equals_G_lambda}
        \left(q_\Xi\right)_* \Ec_\Xi \cong \Ec_\lambda.
    \end{equation}
    Details about the equivariant structures are explained in Remark \ref{fact:rmk:eq_str_id} after the equivariant structure of $\bar{\Gc}_\lambda$ is introduced in detail in Definition \ref{fact:def:partition_sheaf_via_tree}. Consider the commutative diagram
    \begin{equation*}
        \begin{tikzcd}
            & \left[\left(\prod_{i=1}^k\sym^{\abs{B_i}}(X)\right)/S^B_\lambda\right] \ar[d]\ar[dd, bend left=78, "\vartheta"] \\
            \left[X^{b(\alpha)}/S_\Xi\right] \ar[d,"q_\Xi"]\ar[ur, "\textstyle\prod_i p"]\ar[r] & \prod_{\alpha_i\in \lambda} \sym^{l_i}\left(\sym^{b(\alpha_i)}(X)\right) \ar[d] \\
            \left[X^{b(\alpha)}/S_{b(\alpha)}\right] \ar[r, "p"] & \sym^{b(\alpha)}(X).
        \end{tikzcd}
    \end{equation*}
    This gives us
    \begin{equation*}
       p_*\left(\Ec_\lambda\right) \cong p_*\left(q_\Xi\right)_*\left(\Ec_\Xi\right)\cong \vartheta_*\left(\textstyle\prod_i p\right)_*\left(\Ec_\Xi\right).
    \end{equation*}
    But by construction $\left(\textstyle\prod_i p\right)_*\left(\Ec_\Xi\right)$ is the same as $\left(p_*\Ec\right)_\Xi$, finishing the proof of (a).
\end{proof}

Using the pieces in Definitions \ref{fact:def:G_partition_sheaf} and \ref{fact:def:lambda_eq_sheaf_sym}, together with Lemma \ref{fact:lemma:lambda_pieces_push_pull_identities}(d), we can now define plethystic exponentials in K-theory.

\begin{definition}\label{fact:def:pexp}
    For a system $\Gc_{\alpha}$ of $S_{b(\alpha)}$-equivariant complexes of sheaves on $X^{b(\alpha)}$, we define the plethystic exponential to be
    \begin{equation*}
        \pexpI{\sum_{\alpha\in I}q^{\alpha} [\Gc_{\alpha}]} := 1+ \sum_{\alpha\in I} q^{\alpha} \sum_{\lambda\in P(\alpha)} [\Gc_\lambda],
    \end{equation*}
    where $1$ denotes the unit in the K-theory of $X^0=\pt$.

    For a system $\Gc_{\alpha}$ of complexes of sheaves on $\sym^{b(\alpha)}(X)$, we define the plethystic exponential to be
    \begin{equation*}
        \pexpI{\sum_{\alpha\in I}q^{\alpha} [\Gc_{\alpha}]} := p_\ast\pexpI{\sum_{\alpha\in I}q^{\alpha} [p^*\Gc_{\alpha}]},
    \end{equation*}
    where $p:[X^{b(\alpha)}/S_{b(\alpha)}] \to \sym^{b(\alpha)}(X)$ denotes the quotient morphism. Here we use Lemma \ref{fact:lemma:lambda_pieces_push_pull_identities}(d) for well-definedness in K-theory.

    For a system $\Gc_{\alpha}$ of complexes of sheaves on $X$, we define the plethystic exponential to be
    \begin{equation*}
        \pexpI{\sum_{\alpha\in I}q^{\alpha} [\Gc_{\alpha}]} := \pexpI{\sum_{\alpha\in I}q^{\alpha} [i_\ast\Gc_{\alpha}]},
    \end{equation*}
    where $i:X \to \sym^{b(\alpha)}(X)$ denotes the diagonal embedding.
\end{definition}

The plethystic exponential defined above can in general depend on $I$, which is reflected in the notation $\mathrm{PExp}^I$. We will show its relation to the definition in \cite{ok15}. If the factorization index set $I$ is closed under addition, as is the case for schemes, the definitions agree.
\begin{lemma}\label{fact:lemma:comp_formula}
    For a system $\Gc_{\alpha}$ of complexes of sheaves on $\sym^{b(\alpha)}(X)$, we have
    \begin{equation*}
        \chi\left(\pexpI{\sum_{\alpha\in I}q^\alpha [\Gc_\alpha]}\right)= \restr{\Ss^\bullet\left(\sum_{\alpha\in I}q^\alpha \chi\left([\Gc_\alpha]\right)\right)}{I},
    \end{equation*}
    where $\Ss^\bullet$ denotes the morphism in K-theory induced by the symmetric product of vector spaces, and $\restr{-}{I}$ restricts the right-hand side to all terms $(\cdots)q^\alpha$ with $\alpha\in I$. In the $\Ts=(\C^*)^3$-equivariant case, this identification gives us the explicit computational formula
    \begin{equation*}
        \chi_\Ts\left(\pexpI{\sum_{\alpha\in I}q^\alpha [\Gc_\alpha]}\right)= \restr{\exp\left(\sum_{n>0}\frac{1}{n}\sum_{\alpha\in I}q^{n\alpha}\chi_\Ts([\Gc_\alpha],t_1^n,t_2^n,t_3^n)\right)}{I},
    \end{equation*}
    where $t_1,t_2,t_3$ is a basis for the characters of $\Ts$. For $I$ closed under addition, for example in the scheme case of Lemma \ref{fact:lemma:hilb_fact_index_set}, the restriction $\restr{-}{I}$ is not necessary. In particular, in this case $\chi\left(\pexpI{-}\right)$ is multiplicative under addition of the argument. If $I$ is closed under addition, by abuse of notation, we will sometimes write $\mathrm{PExp}$ for the function $\Ss^\bullet$.
\end{lemma}
\begin{proof}
    Take $\alpha\in I$, $\lambda=\{\alpha_1,\dots,\alpha_k\}$ in $P(\alpha)$, and $\Xi\in P^\lambda[\alpha]$. By invariance of Euler characteristics under pushforward and Lemma \ref{fact:lemma:lambda_pieces_push_pull_identities} (b), we obtain the identity
    \begin{equation}\label{fact:eq:sym_mod_SB}
        \chi\left(\pexpI{\sum_{\alpha\in I}q^\alpha [\Gc_\alpha]}\right)= 1+ \sum_{\alpha\in I} q^{\alpha} \sum_{\lambda\in P(\alpha)} \chi\left(\left[\left(\textstyle\prod_i \sym^{b(\alpha_i)}(X)\right)/S^B_\lambda\right],\Gc_\Xi\right),
    \end{equation}
    where we've chosen a representative $\Xi$ for every $\lambda$.

    The coefficient of $q^\alpha$ in $\Ss^\bullet\left(\sum_{\alpha\in I}q^\alpha \chi\left([\Gc_\alpha]\right)\right)$ is
    \begin{equation*}
        \sum_{\lambda\in P(\alpha)}\left(\textstyle\prod_i \chi\left(\Gc_{\alpha_i}\right)\right)^{S_\lambda^B},
    \end{equation*}
    where we have again chosen a representative $\Xi\in P^\lambda(\alpha)$ which determines the order of the product. As above, $S_\lambda^B$ acts by permuting the various multiples of the same factors and we take the $S_\lambda^B$-invariant part. Evaluating the Euler characteristics in the right-hand side of \eqref{fact:eq:sym_mod_SB} is the same as taking the $S_\lambda^B$-invariant part of the pushforward of the class of $\Gc_\Xi$ to $BS_\lambda^B$. Summing over $\lambda\in P(\alpha)$, this yields exactly the terms of $\Ss^\bullet\left(\sum_{\alpha\in I}q^\alpha \chi\left([\Gc_\alpha]\right)\right)$ above, proving the first identification. 

    We can now use that both sides behave the same under addition of arguments. So, we may assume all $\chi\left([\Gc_\alpha]\right)$ are simply characters $\dv{t}^{\vec \mu_\alpha}$. Then we have by definition
    \begin{equation*}
        \chi\left(\pexpI{\sum_{\alpha\in I}q^\alpha [\Gc_\alpha]}\right) = 1+ \sum_{\alpha\in I} q^\alpha \sum_{\lambda\in P(\alpha)} \dv{t}^{\sum_{\alpha_i\in \lambda}\vec \mu_{\alpha_i}}.
    \end{equation*}
    We can rewrite this sum as
    \begin{equation*}
        \restr{\prod_{\alpha\in I} \left(\sum_{k\geq 0} \left(q^\alpha \dv{t}^{\vec \mu_{\alpha}}\right)^k\right)}{I} = \restr{\prod_{\alpha\in I} \left(1-q^\alpha \dv{t}^{\vec \mu_\alpha}\right)^{-1}}{I},
    \end{equation*}
    where we identified the geometric series. The formal power series of the logarithm now allows us to rewrite this in the desired form.
    \begin{equation*}
        \chi\left(\pexpI{\sum_{\alpha\in I}q^\alpha [\Gc_\alpha]}\right) = \restr{\exp\left(\sum_{n>0}\frac{1}{n}\sum_{\alpha\in I}q^{n\alpha} \dv{t}^{n \vec \mu_\alpha}\right)}{I}.
    \end{equation*}
\end{proof}

\subsubsection{Factorization Theorem}

Now we can state the main result of this section.
\begin{theorem}\label{fact:thm:fact}
    Let $I$ be a factorization index semigroup. For a factorizable system $\Fc_{\alpha}$ on $\sym^{b(\alpha)}(X)$, there exist classes $[\Gc_{\alpha}]$ on $X$ such that
    \begin{equation*}
        1 + \sum_{\alpha\in I} q^{\alpha} [\Fc_{\alpha}] = \pexpI{\sum_{\alpha\in I}q^{\alpha} [\Gc_{\alpha}]}
    \end{equation*} 
    in $\bigoplus_{\alpha\in I}q^\alpha K(\sym^{b(\alpha)}(X))$.

    Let $U$ be the non-stacky locus, and let $S$ be its complement in $X$. If the system is given such that for any $\alpha\in I\setminus \Delta_I$, $\restr{\Fc_{\alpha}}{\sym^{b(\alpha)}(U)}$ vanishes, then the classes $[\Gc_{\alpha}]$ can be chosen so that they are supported on $S\subset X$ for $\alpha\in I\setminus \Delta_I$.
\end{theorem}

The second case of the theorem is fulfilled for pushforwards of virtual structure sheaves from the Hilbert scheme, which is our main application.

\begin{remark}
    Note we require $I$ to be closed under addition. The dependence on this condition is subtle, and we plan to study the necessary corrections to the above Theorem to make it work with $I$ that are not closed under addition.
\end{remark}

\begin{remark}
    The last part of the above theorem would work for any subset that is closed under addition instead of $\Delta_I$ and any subscheme of $X$.
\end{remark}

By applying Euler characteristics on both sides, we obtain the following computational formula, which can be combined with Lemma \ref{fact:lemma:comp_formula} to compute the left-hand side in simpler terms.

\begin{corollary}\label{fact:cor:comp_pexp_fact}
    Let $I$ be a factorization index semigroup. For a factorizable system $\Fc_{\alpha}$ on $\sym^{b(\alpha)}(X)$, there exist classes $[\Gc_{\alpha}]$ on $X$ such that
    \begin{equation*}
        1 + \sum_{\alpha\in I} q^{\alpha} \chi\left(\sym^{b(\alpha)}(X),\Fc_{\alpha}\right) = \pexpI{\sum_{\alpha\in I}q^{\alpha} \chi(X,\Gc_{\alpha})},
    \end{equation*}
    where the right-hand side can be computed using Lemma \ref{fact:lemma:comp_formula}.

    As in the theorem above, if the system is given such that for any $\alpha\in I\setminus \Delta_I$, $\restr{\Fc_{\alpha}}{\sym^{b(\alpha)}(U)}$ vanishes, then the classes $[\Gc_{\alpha}]$ can be chosen so that they are supported on $S\subset X$ for $\alpha\in I\setminus \Delta_I$.
\end{corollary}

\subsection{Proof of Theorem \ref{fact:thm:fact}}

The proof will work with an $S_{b(\alpha)}$-equivariant factorizable system on $X^{b(\alpha)}$. By Lemma \ref{fact:lemma:comp_setup_cd} and Example \ref{fact:ex:sym_to_Xb}, a factorizable system on $\sym^{b(\alpha)}(X)$ gives this by pullback. 

\subsubsection{K-Theory Class}

The (2-periodic) complexes of sheaves of the given factorizable system $\Fc_\alpha$ factor into simpler parts according to partitions of $\alpha$. Roughly, the $\Gc_\alpha$ are defined as (double) complexes tracking the possible splittings of $\alpha$ into smaller parts by systems of subordinate partitions. For now, we will only define the complexes $\Gc_\alpha$ with trivial differential and examine their K-theory classes. Later in Section \ref{fact:sec:combinatorics}, we will introduce a non-trivial differential and use that K-theory classes are independent of the differential.

From now on, we will assume given a certain choice of factorization index semigroup $I$ and suppress it from the notation whenever reasonable.

To keep track of the different ways to split $[\alpha]$ into smaller parts by different partitions, we use the following notion of an $\alpha$-index tree.

\begin{definition}
    An $\alpha$\textbf{-index tree} is a rooted tree, with each non-leaf having at least two children, together with a tuple $(B_l,\tilde{\alpha}_l)$ for every leaf $l$ such that $\{(B_l,\tilde{\alpha}_l)\ |\ l\ \mathrm{leaf}\}$ is an $I$-partition of $[\alpha]$. We denote the set of all $\alpha$-index trees by $\Tc(\alpha)$, and we denote the set of $\alpha$-index trees with exactly $k$ non-leaf nodes by $\Tc(\alpha,k)$.

    For every node $v$, we set $B_v$ to be the union of all $B_l$ with $l$ a leaf descendant of $v$. Similarly, we set $\tilde{\alpha}_v$ to be the sum of all $\tilde{\alpha}_l$ with $l$ a leaf descendant of $v$.

    Given a partition $\Xi=\{(B_1,\tilde{\alpha}_1), \dots, (B_k,\tilde{\alpha}_k)\}$ of $[\alpha]$, we denote by $\Tc_\Xi(\alpha)$ the set of $\alpha$-index trees with partition $\Xi$ at the level of root children.
\end{definition}

For several constructions we need a specified ordering on the index tree or on a given partition of $[\alpha]$.
\begin{definition}
    An \textbf{ordering} of an $\alpha$-index tree is a total order on its set of nodes, compatible with the tree structure in the following way. If $v$ is a descendant of $w$, then we require $v \leq w$, and if $v_1$ and $v_2$ are siblings with $v_1 < v_2$, then we require all descendants of $v_1$ to be smaller than all descendants of $v_2$.
\end{definition}

\begin{definition}
    Let $\leq_1, \leq_2$ be two orderings on an index tree $T$. Order the relevant nodes of $T$ such that $v_1 <_1 \cdots <_1 v_k$. There is a unique $\sigma\in S_k$ such that $v_{\sigma(1)} <_2 \cdots <_2 v_{\sigma(k)}$ (the one defined by exactly this condition). We define $s(\leq_1,\leq_2)=\mathrm{sgn}(\sigma)$.
\end{definition}

\begin{definition}
    The standard ordering of an $\alpha$-index tree $T$, which we denote by $\leq_T$, is defined recursively as follows. Take the root and make it the unique maximal node. Order the root children $\{v_1,\dots,v_k\}$ by
    \begin{eqnarray*}
        v_i \leq_T v_j :\Leftrightarrow \min(B_{v_i}) \leq \min(B_{v_j}).
    \end{eqnarray*}
    For each root child $v_i$, order all its descendants recursively as above (pretending that $v_i$ is the root node). Finally for any pair of root children $v_i \leq_T v_j$ order all descendants of $v_i$ to be smaller than all descendants of $v_j$.
\end{definition}

\begin{example}\label{fact:ex:part_std_order}
    Note that given a partition $\Xi=\{(B_1,\tilde{\alpha}_1), \dots, (B_k,\tilde{\alpha}_k)\}$ of $[\alpha]$, the standard ordering on the $\alpha$-index tree
    \begin{equation*}
        \begin{tikzcd}
            & & \left[\alpha\right] \arrow[dash, dll] \arrow[dash, dl] \arrow[dash, d] \arrow[dash, dr] \arrow[dash, drr] & & \\
            (B_1,\tilde{\alpha}_1) & (B_2,\tilde{\alpha}_2) & \dots & (B_{k-1},\tilde{\alpha}_{k-1}) & (B_k,\tilde{\alpha}_k)
        \end{tikzcd}
    \end{equation*}
    defines an ordering of the $k$ subsets of the partition. From now on, when we specify a partition $\Xi=\{(B_1,\tilde{\alpha}_1), \dots, (B_k,\tilde{\alpha}_k)\}$ of $[\alpha]$, we will assume, the $(B_i,\tilde{\alpha}_i)$ are ordered by this standard ordering.
\end{example}

\begin{definition}\label{fact:def:partition_sheaf_via_tree}
    Given an ordered index tree $(T,\leq)$ with leaf partition $\Xi=\{(B_1,\tilde{\alpha}_1), \dots, (B_k,\tilde{\alpha}_k)\}$ of $[\alpha]$ ordered by $\leq$, and a system of $S_{b(\beta)}$-equivariant (2-periodic) complexes of sheaves $\Gc_{\beta}$ on $X^{b(\beta)}$, we define
    \begin{equation*}
        \Gc_{T,\leq} := \Gc_{(\abs{B_1},\tilde{\alpha}_1)} \boxtimes \dots \boxtimes \Gc_{(\abs{B_k},\tilde{\alpha}_k)}
    \end{equation*}
    as a complex of sheaves on $X^{b(\alpha)}$, where each pullback is along the map $\pr_{b_i^1,\dots,b_i^{l_i}}:X^{b(\alpha)} \to X^{\abs{B_i}}$ specified by the subset $B_i=\{b_i^1,\dots,b_i^{l_i}\}$, where we order $B_i$ such that $b_i^1<\dots <b_i^{l_i}$.

    Given a partition $\Xi=\{(B_1,\tilde{\alpha}_1), \dots, (B_k,\tilde{\alpha}_k)\}$ of $[\alpha]$, and a system of $S_{b(\beta)}$-equivariant (2-periodic) complexes of sheaves $\Gc_{\beta}$ on $X^{b(\beta)}$, we define
    \begin{equation*}
        \Gc_{\Xi} := \Gc_{T,\leq}
    \end{equation*}
    as a complex of sheaves on $X^{b(\alpha)}$, where $(T,\leq)$ is the standard ordered index tree associated to the partition $\Xi$ from Example \ref{fact:ex:part_std_order}.

    For a partition $\lambda$ of $\alpha$, we define
    \begin{equation*}
        \Gc_\lambda := \bigoplus_{\Xi\in P^\lambda[\alpha]} \Gc_{\Xi},
    \end{equation*}
    which inherits a natural $S_{b(\alpha)}$-equivariant structure as follows. An element $\sigma$ of $S_{b(\alpha)}$ permutes the various projections $X^{b(\alpha)} \to X^{\abs{B_i}}$ by acting on a partition $\Xi=\{(B_1,\tilde{\alpha}_1), \dots, (B_k,\tilde{\alpha}_k)\}$ of $[\alpha]$ by $\sigma\left(\Xi\right)= \{(\sigma(B_1),\tilde{\alpha}_1), \dots, (\sigma(B_k),\tilde{\alpha}_k)\}$. This induces a permutations $\sigma(T,\leq)=(\sigma(T),\sigma(\leq))$ of the standard ordered index trees $(T,\leq)$ associated to the partitions $\Xi$, which permutes the labels while preserving the ordering. The $S_{b(\beta)}$-equivariant structure of the $\Gc_{\beta}$ on $X^{b(\beta)}$ then induces canonical isomorphisms
    \begin{equation}\label{fact:eq:eq_str_def_morphism_1}
        \sigma^*\left(\Gc_{T,\leq}\right) \xrightarrow{\sim} \Gc_{\sigma\left(T,\leq\right)}.
    \end{equation}
    The $S_{b(\alpha)}$-equivariant structure on $\Gc_\lambda$ is defined as the sum over $\Xi\in P^\lambda[\alpha]$ of the compositions
    \begin{equation}\label{fact:eq:eq_str_def_morphisms}
        \sigma^*\left(\Gc_{T,\leq}\right) \xrightarrow{\sim} \Gc_{\sigma(T),\sigma(\leq)} \xrightarrow{s(\sigma(\leq),\leq)\cdot S}\Gc_{\sigma(T),\leq},
    \end{equation}
    where $S$ is the standard reordering isomorphism to obtain the standard ordering $\leq$ on $\sigma(T)$.
\end{definition}

\begin{remark}\label{fact:rmk:eq_str_id}
    We explain the identification of the above equivariant structure with the induced equivariant structure in the proof of Lemma \ref{fact:lemma:lambda_pieces_push_pull_identities}:
    \begin{itemize}
        \item For a given partition $\Xi=\{(B_1,\tilde{\alpha}_1), \dots, (B_k,\tilde{\alpha}_k)\}$ with associated index tree $T$, the order on the $B_i$ in the definition of $\Gc_{T,\leq}$ together with the standard order on the index tree $T$ fix a specific choice of $\Gc_\Xi$, which is related to other possible choices by pulling it back along the action morphisms of elements in $S_\Xi$.
        \item The isomorphisms in \eqref{fact:eq:eq_str_def_morphisms} correspond to the isomorphisms \eqref{qst:eq:eq_str_H_class_isoms} to pullbacks of the components along the action morphisms of elements in $S_\Xi$. The first isomorphism \eqref{fact:eq:eq_str_def_morphism_1} corresponds to the action of $\prod_i S_{b(\alpha_i)}$, and the second isomorphism $s(\sigma(\leq),\leq)\cdot S$ corresponds to the action of $S^B_\lambda$.
        \item To match the sign in \eqref{fact:eq:eq_str_def_morphisms} under the isomorphism \eqref{fact:eq:quotient_pushforward_equals_G_lambda}, we must also introduce the corresponding sign into the $S_\Xi$-equivariant structure of $\Gc_\Xi$.
    \end{itemize}
\end{remark}

Note that the above definition of $\Gc_{\Xi}$ and $\Gc_{\lambda}$ agrees with the one of Definition \ref{fact:def:G_partition_sheaf}. Now we can define the following complex, which is the central object in the proof of Theorem \ref{fact:thm:fact}. While it is defined with trivial differential for the current study of its K-theory class, we will later equip it with a non-trivial differential.

\begin{definition}\label{fact:def:tree_complex}
    Let $I$ be a factorization index semigroup, and let $\Fc_{\beta}$ be a factorizable system of $S_{b(\beta)}$-equivariant (2-periodic) complexes of sheaves on $X^{b(\beta)}$. For $\alpha\in I$, $k\geq 0$, we define the complex of sheaves on $X^{b(\alpha)}$
    \begin{equation*}
        \Gc_{\alpha}^{-k} := \bigoplus_{T\in\Tc(\alpha,k)} \Fc_{T,\leq_T},
    \end{equation*}
    where we recall that $\Tc(\alpha,k)$ is the set of $\alpha$-index trees with exactly $k$ non-leaf nodes, and $\leq_T$ is the standard ordering of a given index tree. We define the (double) complex
    \begin{equation*}
        \Gc_{\alpha}:=(\Gc_{\alpha}^\bullet,0)
    \end{equation*}
    with trivial differential.
\end{definition}

We equip this complex with an $S_{b(\alpha)}$-equivariant structure as follows.

\begin{definition}\label{fact:def:tree_complex_eq_str}
    The $S_{b(\alpha)}$-equivariant structure on is induced as in Definition \ref{fact:def:partition_sheaf_via_tree}. An element $\sigma$ of $S_{b(\alpha)}$ permutes the partitions $\Xi=\{(B_1,\tilde{\alpha}_1), \dots, (B_k,\tilde{\alpha}_k)\}$ of $[\alpha]$ by $\sigma\left(\Xi\right)= \{(\sigma(B_1),\tilde{\alpha}_1), \dots, (\sigma(B_k),\tilde{\alpha}_k)\}$. This induces permutations $\sigma(T,\leq)=(\sigma(T),\sigma(\leq))$ of the index trees $(T,\leq)$ by permuting the labels while preserving the ordering. The $S_{b(\alpha)}$-equivariant structure on $\Gc_\alpha$ is defined as the sum over $T\in\Tc(\alpha,k)$ of the compositions
    \begin{equation*}
        \sigma^*\left(\Gc_{T,\leq}\right) \xrightarrow{\sim} \Gc_{\sigma(T),\sigma(\leq)} \xrightarrow{s(\sigma(\leq),\leq)\cdot S}\Gc_{\sigma(T),\leq},
    \end{equation*}
    where the first isomorphism is the canonical isomorphism induced by the $S_{b(\beta)}$-equivariant structure of the $\Fc_{\beta}$ on $X^{b(\beta)}$, and $S$ is the standard reordering isomorphism to obtain the standard ordering $\leq$ on $\sigma(T)$. The sign $s(\sigma(\leq),\leq)$ makes the complex a complex of $S_{b(\alpha)}$-equivariant sheaves.
\end{definition}

For the proof of Theorem \ref{fact:thm:fact}, we need to express the K-theory class of the (2-periodic) complexes of sheaves $\Fc_{\beta}$ of the factorizable system in simpler terms. The next lemma allows us to do exactly that by establishing a relation between the K-theory classes of $\Fc_{\beta}$ and the complexes $\Gc_{\alpha}$ defined above.

Note here, that $\Gc_{\alpha}$ is a double complex. Its K-theory class is the K-theory class of its total complex, which is the same as a signed sum of the classes of the complexes $\Gc_{\alpha}^k$.

\begin{lemma}\label{fact:lemma:complex_ktheory_class}
    In $K_{S_{b(\alpha)}}(X^{b(\alpha)})$ we have the identity
    \begin{equation*}
        [\Gc_{\alpha}] = [\Fc_{\alpha}] - \sum_{\lambda\in P(\alpha)\ non-trivial} [\Gc_{\lambda}],
    \end{equation*}
    where the last term is equipped with its standard equivariant structure, defined as in Definition \ref{fact:def:G_partition_sheaf}.
\end{lemma}

\begin{proof}
    By definition $\Gc_\lambda = \bigoplus_{\Xi\in P^\lambda[\alpha]} \Gc_{\Xi}$, so our strategy will be to find a filtration of $\Gc_{\alpha}$ that allows us to split its K-theory class into pieces of the form $-[\Gc_{\Xi}]$.

    Consider the filtration $F_l\Gc_{\alpha}$ of subcomplexes of $\Gc_{\alpha}$ consisting of sums over all trees with $\leq l$ root children. Let $Q_l\Gc_{\alpha}$ be the quotient complexes of the filtration, so the subcomplexes exactly consisting of sums over all trees with exactly $l$ root children. We get
    \begin{equation*}
        [\Gc_{\alpha}] = [\Fc_{\alpha}] + \sum_{l> 0}[Q_l\Gc_{\alpha}].
    \end{equation*}
    Note that the $S_{b(\alpha)}$-equivariant structure is preserved by the above filtration, so this equality holds in $K_{S_{b(\alpha)}}(X^{b(\alpha)})$.

    We can split this further into parts with fixed partition at the level of root children. Given $\Xi\in P[\alpha]$ of length $l$, we let $Q_l^{\Xi}\Gc_{\alpha}$ be the subcomplex consisting of the sums over all trees with partition $\Xi$ at the level of root children. For a partition $\lambda \in P(\alpha)$, define
    \begin{equation*}
        Q_l^{\lambda}\Gc_{\alpha} := \bigoplus_{\Xi\in P^\lambda[\alpha]} Q_l^{\Xi}\Gc_{\alpha},
    \end{equation*}
    with its natural $S_{b(\alpha)}$-equivariant structure defined analogously to Definition \ref{fact:def:partition_sheaf_via_tree}. By definition of the $S_{b(\alpha)}$-equivariant structure of $\Gc_{\alpha}$, we get an equality
    \begin{eqnarray}\label{fact:eq:k_class_filtration_eq}
        [\Gc_{\alpha}] =& [\Fc_{\alpha}] + \sum_{l> 0}\sum_{\lambda\in P_l(\alpha)} [Q_l^\lambda\Gc_{\alpha}]\nonumber\\
        =& [\Fc_{\alpha}] + \sum_{\lambda\in P(\alpha)\ non-trivial} [Q_{l(\lambda)}^\lambda\Gc_{\alpha}]
    \end{eqnarray}
    in $K_{S_{b(\alpha)}}\left(X^{b(\alpha)}\right)$. 
    
    Now we examine the pieces $Q_l^{\Xi}\Gc_{\alpha}$ further. As the differential of $\Gc_\alpha$ is defined as $0$ for now, we can ignore it in this proof. Note that these have no $S_{b(\alpha)}$-equivariant structure, but the $S_{b(\alpha)}$-equivariant structure of $Q_l^{\lambda}\Gc_{\alpha}$ is induced in the same way as the one of $\Gc_{\lambda}$. We want to show
    \begin{equation*}
        Q_{l(\Xi)}^{\Xi}\Gc_{\alpha}[1] = \Gc_{\Xi}.
    \end{equation*}
    Let $l=l(\Xi)$. The complex $Q_{l}^{\Xi}\Gc_{\alpha}$ is given by
    \begin{equation*}
        Q_{l}^{\Xi}\Gc_{\alpha}=\bigoplus_{T\in \Tc_{\Xi}(\alpha,\bullet)} \Fc_{T,\leq_T}.
    \end{equation*}
    Let $\Xi=\{(B_1,\tilde{\alpha}_1),\dots,(B_l,\tilde{\alpha}_l)\}$ be ordered by the standard ordering. Then we have 
    \begin{equation*}
        \bigoplus_{T\in \Tc_{\Xi}(\alpha,k)} \Fc_{T,\leq_T} = \bigoplus_{k_1+\cdots+k_l\leq k-1}\bigoplus_{\substack{T_i\in \Tc((\abs{B_i},\tilde{\alpha}_i),k_i)\\1\leq i\leq l}} \Fc_{T_1,\leq_{T_1}}\boxtimes\cdots\boxtimes \Fc_{T_l,\leq_{T_l}}
    \end{equation*}
    Here we take $k_1+\cdots+k_l\leq k-1$ as we remove the root node in the process of passing from left to right. As a complex this is the same as shifting, so we obtain
    \begin{equation*}
        Q_{l}^{\Xi}\Gc_{\alpha}[1] = \displaystyle\bigoplus_{\substack{k_1+\cdots+k_l\leq\bullet\\T_i\in \Tc\left(\left(\abs{B_i},\tilde{\alpha}_i\right),k_i\right)\\1\leq i\leq l}} \Fc_{T_1,\leq_{T_1}}\boxtimes\cdots\boxtimes \Fc_{T_l,\leq_{T_l}}
    \end{equation*}
    We can distribute to get
    \begin{equation*}
        Q_{l}^{\Xi}\Gc_{\alpha}[1] = \displaystyle\left(\bigoplus_{\substack{k_1,\ T_1\in \Tc\left(\left(\abs{B_1},\tilde{\alpha}_1\right),k_1\right)}} \Fc_{T_1,\leq_{T_1}}\right)\boxtimes\cdots\boxtimes \left(\bigoplus_{\substack{k_l,\ T_l\in \Tc\left(\left(\abs{B_l},\tilde{\alpha}_l\right),k_l\right)}} \Fc_{T_l,\leq_{T_l}}\right),
    \end{equation*}
    where the degree $k$ part on the right-hand side is exactly given by the terms where $k_1+\cdots+k_l\leq k$. But this is exactly
    \begin{equation*}
        Q_{l}^{\Xi}\Gc_{\alpha}[1] = \Gc_{\left(\abs{B_1},\tilde{\alpha}_1\right)}\boxtimes\cdots\boxtimes \Gc_{\left(\abs{B_l},\tilde{\alpha}_l\right)}= \Gc_{\Xi}
    \end{equation*}
    identified as complexes. By summing over partitions $\Xi$ with $\abs{\Xi}=\lambda$ we get an $S_{b(\alpha)}$-equivariant isomorphism of complexes of sheaves
    \begin{equation*}
        Q_{l}^\lambda\Gc_{\alpha}[1] = \Gc_{\lambda},
    \end{equation*}
    since the $S_{b(\alpha)}$-equivariant structure is induced in the same way on both sides. In particular, we get $[Q_{l}^\lambda\Gc_{\alpha}] = -[\Gc_{\lambda}]$. Inserting this into \eqref{fact:eq:k_class_filtration_eq} gives the desired equality in $K_{S_{b(\alpha)}}(X^{b(\alpha)})$.
\end{proof}

The previous Lemma \ref{fact:lemma:complex_ktheory_class} tells us that in order to understand the generating series of K-theory classes of a factorizable system, we can express this in terms of the K-theory classes of the $G_{\alpha}$. To prove Theorem \ref{fact:thm:fact}, we equip $\Gc_\alpha$ with a differential and study the support of the resulting complex. Note that modifying the differential leaves the K-theory class unchanged, and hence the above Lemma \ref{fact:lemma:complex_ktheory_class} still holds for the resulting complex.

\subsubsection{Level Systems}\label{fact:sec:good_replacement}

The terms of $\Gc_\alpha$ are $\Fc_\Xi$ for $[\alpha]$-partitions $\Xi$ associated to leaves of various $\alpha$-index trees. We want to construct the differential of $\Gc_\alpha$ from the factorization morphisms $\phi$ by splitting $[\alpha]$-partitions into smaller parts. To do this, we need to extend our morphisms to be defined over the whole domain. We now make this precise.

\begin{definition}
    Let $\Xi_1$ be the $[\alpha]$-partition $\{(B^1_1,\tilde{\alpha}^1_1), \dots, (B^1_k,\tilde{\alpha}^1_k)\}$, and let $\Xi_2$ be the partition $\{(B^2_1,\tilde{\alpha}^2_1), \dots, (B^2_l,\tilde{\alpha}^2_l)\}$ of $[\alpha]$. We say $\Xi_1$ is \textbf{subordinate} to $\Xi_2$, denoted $\Xi_1\leq\Xi_2$, if any $B^1_i$ is contained in some $B^2_j$, and additionally for $B^2_j=\bigcup_{i\in A} B^1_i$ (such a union must exist by the previous condition) we have $\sum_{i\in A}\tilde{\alpha}^1_i = \tilde{\alpha}^2_j$.
\end{definition}

Instead of using the entire factorizable systems, we use a part of it to define a system in the following sense.

\begin{definition}\label{fact:def:level_system}
    We consider the scheme $X^{b(\alpha)}$ together with its $S_{b(\alpha)}$-action. For each element $B$ of $P[b(\alpha)]$, we a get locally closed subset $X_B$ of points $(x_1,\cdots,x_{b(\alpha)})\in X$ such that $x_i=x_j$ if and only if $i$ and $j$ are in the same set of $B$. We have a partial order on $P[b(\alpha)]$, for which $A \leq B$ if $B$ is a refinement of $A$. $P[b(\alpha)]$ carries a natural $S_{b(\alpha)}$-action with $\sigma(X_B)=X_{\sigma(B)}$. Define
    \begin{equation*}
        U_A \coloneqq \bigcup_{B\geq A} X_B.
    \end{equation*}
    For use below, for $A$ and $B$ in $P[b(\alpha)]$, we write $A\cap B\in P[b(\alpha)]$ for the partition obtained by taking all intersections of sets in $A$ and $B$. For every $A\in P[b(\alpha)]$, we can define an equivalence relation on $P[b(\alpha)]$ by $B_1\sim_A B_2$ if and only if $B_1\cap A$ is the same partition as $B_2\cap A$.

    Now consider $P[\alpha]$, which carries a natural $S_{b(\alpha)}$-action. This also comes with a partial order, for which $\Xi_1\leq \Xi_2$ if $\Xi_1$ is subordinate to $\Xi_2$ (note the flip of the order compared to the order on $P[b(\alpha)]$). For every $B\in P[b(\alpha)]$ we can define an equivalence relation on $P[\alpha]$ by $\Xi_1 \sim_B \Xi_2$ if and only if the partitions $B(\Xi_1)\cap B$ and $B(\Xi_2)\cap B$ agree as partitions of $[b(\alpha)]$. The $S_{b(\alpha)}$-action respects the partial ordering and the equivalence relations. Moreover, we see immediately that for $A\leq B$ in $P[b(\alpha)]$, $\Xi_1 \sim_A \Xi_2$ implies $\Xi_1 \sim_B \Xi_2$, and that for $\Xi_1\leq \Xi_2\leq \Xi_3$, $\Xi_1 \sim_B \Xi_3$ holds if and only if both $\Xi_1 \sim_B \Xi_2$ and $\Xi_2 \sim_B \Xi_3$ hold.

    Given the above data, we define an $\alpha$\textbf{-level system} to be a collection of $S_{b(\alpha)}$-equivariant complexes of coherent sheaves $\left\{\Fc_\Xi\right\}_{\Xi\in P[\alpha]}$ on $X^{b(\alpha)}$, together with, for each $\Xi_1\leq \Xi_2$ in $P[\alpha]$ and $B\in P[b(\alpha)]$ such that $\Xi_1 \sim_B \Xi_2$, a morphism 
    \begin{equation*}
        \phi_{\Xi_1 \Xi_2 B}: \restr{\Fc_{\Xi_1}}{U_B} \to \restr{\Fc_{\Xi_2}}{U_B}
    \end{equation*}
    compatible with the $S_{b(\alpha)}$-equivariant structure, such that
    \begin{enumerate}[(a)]
        \item Any two morphisms $\phi_{\Xi_1 \Xi_2 A}$ and $\phi_{\Xi_1 \Xi_2 A}$ agree on $U_A\cap U_B$ whenever they exist. As a consequence, the morphisms glue to
        \begin{equation*}
            \phi_{\Xi_1 \Xi_2}: \restr{\Fc_{\Xi_1}}{U_{\Xi_1 \Xi_2}} \to \restr{\Fc_{\Xi_2}}{U_{\Xi_1 \Xi_2}},
        \end{equation*}
        where $U_{\Xi_1 \Xi_2}\coloneqq \textstyle\bigcup_{B:\Xi_1 \sim_B \Xi_2}U_B$.
        \item For any $\Xi_1\leq \Xi_2\leq \Xi_3$ we have over $U_{\Xi_1 \Xi_2}\cap U_{\Xi_2 \Xi_3}=U_{\Xi_1 \Xi_3}$
        \begin{equation*}
            \phi_{\Xi_2 \Xi_3}\circ \phi_{\Xi_1 \Xi_2} = \phi_{\Xi_1 \Xi_3}.
        \end{equation*}
        \item $\phi_{\Xi \Xi}$ is the identity for any $\Xi$ in $P[\alpha]$.
    \end{enumerate}

    We additionally assume that given $\Xi_2\in P[\alpha]$ with a refinement $B$ of $B(\Xi_2)$, for every $A$ with $B\sim_A B(\Xi_2)$
    \begin{equation}\label{fact:eq:level_system_isomorphism_property}
        \bigoplus_{\Xi_1\leq \Xi_2, B(\Xi_1)=B}\phi_{\Xi_1 \Xi_2 A}: \restr{\left(\bigoplus_{\Xi_1\leq \Xi_2, B(\Xi_1)=B}\Fc_{\Xi_1}\right)}{U_A} \to \restr{\Fc_{\Xi_2}}{U_A}
    \end{equation}
    is an isomorphism.

    Morphisms between $\alpha$-level systems $\left(\Fc,\phi\right)$ and $\left(\Fc',\phi'\right)$ are given by morphisms $\Fc_\Xi \to \Fc'_\Xi$ that commute with the $\phi$ and $\phi'$ in the obvious way on the open subschemes, where they are defined. Denote the category of $\alpha$-level systems by $\syst^\alpha$. 
\end{definition}

\begin{remark}
    Note that an $\alpha$-level system is an $S_{b(\alpha)}$-equivariant system in the sense of \cite{kr} with the exception of one axiom and a modified isomorphism requirement \eqref{fact:eq:level_system_isomorphism_property}. The missing axiom is used to prove their isomorphism property is preserved by the functors $D$ in Definition \ref{fact:def:D_I_strictness_proof} below. We will show that our modified isomorphism property \eqref{fact:eq:level_system_isomorphism_property} is preserved by these functors in Lemma \ref{fact:lemma:D_I_preserve_isom_property} using a modified version of the missing axiom, which is satisified in our case.
\end{remark}

A factorizable system gives us an $\alpha$-level system for each $\alpha$ in $I$ as follows. The definition of the differentials of $\Gc_\alpha$ will only depend on this level system, rather than the entire factorizable system.

\begin{example}\label{fact:ex:fact_level_system}
    Now take the given $S_{b(\beta)}$-equivariant factorizable system $\left\{\Fc_\beta\right\}_{\beta\in I}$ on $X^{b(\beta)}$ with morphisms $\phi_{\beta_1\beta_2}$ for any $\beta_1+\beta_2=\beta$ in $I$. Now consider again our fixed $\alpha$ for which we study $\Gc_\alpha$. We want to define an $\alpha$-level system from the factorizable system $\left\{\Fc_\beta\right\}_{\beta\in I}$. Given a partition $\Xi$ in $P[\alpha]$, we have the complex of coherent sheaves
    \begin{equation*}
        \Fc_\Xi
    \end{equation*}
    on $X^{b(\alpha)}$. This defines the complexes of sheaves in our level system. Any pair of subordinate partitions $\Xi_1\leq \Xi_2$ in $P[\alpha]$ together with $B\in P[b(\alpha)]$ such that $\Xi_1 \sim_B \Xi_2$, gives us a morphism
    \begin{equation*}
        \phi_{\Xi_1 \Xi_2 B}: \restr{\Fc_{\Xi_1}}{U_B} \to \restr{\Fc_{\Xi_2}}{U_B}.
    \end{equation*}
    The morphisms $\phi_{\Xi_1 \Xi_2 B}$ are constructed by tensor product,so we may assume $\Xi_2$ is just $[\alpha]$. In this case, $\Xi_1 \sim_B \Xi_2=[\alpha]$ just means $B$ is a refinement of $B(\Xi_1)$. So, we may assume $B=B(\Xi_1)$. But then, using the associativity property of the factorizable system $\left\{\Fc_\beta\right\}_{\beta\in I}$, the morphisms $\phi_{\beta_1\beta_2}$ of the factorizable system compose to give the desired morphism $\phi_{\Xi_1 \Xi_2 B}$.

    The axioms of a level system follow from the construction of the $\phi_{\Xi_1 \Xi_2 B}$, together with the equivariance requirements and the associativity and commutativity axioms of a factorizable system in Definition \ref{def:factorizable_system_general}. The isomorphism property \eqref{fact:eq:level_system_isomorphism_property} follows from the isomorphism property \eqref{fact:eq:sum_isom_fact} of the factorizable system.
\end{example}

\begin{definition}\label{fact:def:strict_level_system}
    We say the $\alpha$-level system is \textbf{strict} if for every $\Xi_1\leq \Xi_2$ in $P[\alpha]$ there exist morphisms $\overline{\phi}_{\Xi_1 \Xi_2}:\Fc_{\Xi_1}\to \Fc_{\Xi_2}$, which restrict to $\phi_{\Xi_1 \Xi_2}$ on $U_{\Xi_1 \Xi_2}$ and such that for any $\Xi_1\leq \Xi_2\leq \Xi_3$ we have
    \begin{equation*}
        \overline{\phi}_{\Xi_2 \Xi_3}\circ \overline{\phi}_{\Xi_1 \Xi_2} = \overline{\phi}_{\Xi_1 \Xi_3}.
    \end{equation*}
\end{definition}

We want to find a way to replace the level system of Example \ref{fact:ex:fact_level_system} by a strict one without changing its K-theory class. For this, the following constructions are central.

\begin{definition}\label{fact:def:D_I_strictness_proof}
    Take a fixed $A\in P[b(\alpha)]$ and let $j_A$ be the locally closed embedding of $X_A$ into $X^{b(\alpha)}$. Let $(\Fc,\phi)$ be an $\alpha$-level system. We define a strict $\alpha$-level system $\left(D_A(\Fc),D_A\phi\right)$ as follows. We show that this is well-defined in Lemma \ref{fact:lemma:D_I_preserve_isom_property} below.
    
    For $\Xi\in P[\alpha]$, set
    \begin{equation*}
        D_A(\Fc)_\Xi \coloneqq \im\left(\bigoplus_{\Xi'\sim_A\Xi,\Xi'\leq \Xi}\Fc_{\Xi'}\to\bigoplus_{\Xi'\sim_A\Xi,\Xi'\leq \Xi}j_{A *}j_A^*\Fc_{\Xi'}\to j_{A *}j_A^*\Fc_{\Xi}\right),
    \end{equation*} 
    where the first morphism consists of adjunction morphisms and the second one is the sum of all $j_{A *}\phi_{\Xi' \Xi A}$.

    For $\Xi_1\leq \Xi_2$ in $P[\alpha]$ and $B\in P[b(\alpha)]$ we define
    \begin{equation*}
        D_A(\phi)_{\Xi_1\Xi_2 B}: \restr{D_A(\Fc)_{\Xi_1}}{U_B} \to \restr{D_A(\Fc)_{\Xi_2}}{U_B}
    \end{equation*}
    as follows.
    \begin{itemize}
        \item If $B\not\leq A$, then $U_B\cap \overline{X}_A=\emptyset$, so that $\restr{D_A(\Fc)_{\Xi_1}}{U_B}$ and $\restr{D_A(\Fc)_{\Xi_2}}{U_B}$ vanish. Then we set $D_A(\phi)_{\Xi_1\Xi_2 B}=0$.
        \item If $B\leq A$, but $\Xi_1\not\sim_A \Xi_2$, then we set $D_A(\phi)_{\Xi_1\Xi_2 B}=0$.
        \item If $B\leq A$ and $\Xi_1\sim_A \Xi_2$, then the commutative diagram
        \begin{equation}\label{fact:eq:D_morphisms_comm_diagram}
            \begin{tikzcd}
                \bigoplus_{\Xi'\sim_A\Xi_1,\Xi'\leq \Xi_1}\Fc_{\Xi'} \ar[r]\ar[d] & j_{A *}j_A^*\Fc_{\Xi_1} \ar[d,"j_{A *}\phi_{\Xi_1 \Xi_2 A}"] \\
                \bigoplus_{\Xi'\sim_A\Xi_2,\Xi'\leq \Xi_2}\Fc_{\Xi'} \ar[r] & j_{A *}j_A^*\Fc_{\Xi_2}
            \end{tikzcd}
        \end{equation}
        induces a morphism between the images of the morphisms of the top and bottom rows, which we set as the morphism $D_A(\phi)_{\Xi_1\Xi_2 B}$.
    \end{itemize}
    Note that we put no restriction on $B\in P[b(\alpha)]$, so the morphisms $D_A(\phi)_{\Xi_1\Xi_2}$ are defined globally making the level system $D_A(\Fc,\phi)$ strict. That the $D_A(\phi)_{\Xi_1\Xi_2 B}$ glue and give well-defined morphisms is shown in \cite{kr}.

    For each $\Xi\in P[\alpha]$ we get a morphism $I_{A,\Xi}:\Fc_\Xi \to D_A(\Fc)_\Xi$ by inclusion of $\Fc_\Xi$ in $\bigoplus_{\Xi'\sim_A\Xi,\Xi'\leq \Xi}\Fc_{\Xi'}$ in the definition of $D_A(\Fc)_\Xi$.
\end{definition}

\begin{lemma}\label{fact:lemma:D_I_preserve_isom_property}
    $D_A$ gives a well-defined functor $\syst^\alpha\to \syst^\alpha$ and $I_A$ defines a morphism of $\alpha$-level systems. In particular, $D_A$ preserves the isomorphism property \eqref{fact:eq:level_system_isomorphism_property}.
\end{lemma}
\begin{proof}
    As our level systems are special cases of equivariant systems in the sense of \cite{kr}, their proofs apply to show $D_A$ is a functor and $I_A$ a morphism. The only thing left to check is that our isomorphism property \eqref{fact:eq:level_system_isomorphism_property} is preserved by $D_A$. Take $\Xi_2\in P[\alpha]$ with a refinement $B$ of $B(\Xi_2)$ and fix $C\in P[b(\alpha)]$ with $B\sim_C B(\Xi_2)$. We want to show that
    \begin{equation}
        \bigoplus_{\Xi_1\leq \Xi_2, B(\Xi_1)=B}D_A(\phi)_{\Xi_1 \Xi_2 C}: \restr{\left(\bigoplus_{\Xi_1\leq \Xi_2, B(\Xi_1)=B}D_A(\Fc)_{\Xi_1}\right)}{U_C} \to \restr{D_A(\Fc)_{\Xi_2}}{U_C}
    \end{equation}
    is an isomorphism. Note that by $B\sim_C B(\Xi_2)$, for all $\Xi_1$ in the sum above, we have $\Xi_1\sim_C \Xi_2$. Hence, in the definition of $D_A(\phi)_{\Xi_1 \Xi_2 C}$ we are either in the first case, where we trivially get an isomorphism as both sides are $0$, or in the third case, where $C\leq A$ and $\Xi_1\sim_A \Xi_2$. Let us now consider this third case.

    We take the direct sum and insert it into the commutative diagram \eqref{fact:eq:D_morphisms_comm_diagram} to get on $U_C$
    \begin{equation}\label{fact:eq:D_sum_cd}
        \begin{tikzcd}
            \bigoplus_{\substack{\Xi_1\leq \Xi_2, B(\Xi_1)=B\\\Xi''\sim_A\Xi_1,\Xi''\leq \Xi_1}}\Fc_{\Xi''} \ar[r]\ar[d] & \bigoplus_{\substack{\Xi_1\leq \Xi_2\\ B(\Xi_1)=B}}j_{A *}j_A^*\Fc_{\Xi_1} \ar[d,"\bigoplus_{\substack{\Xi_1\leq \Xi_2\\ B(\Xi_1)=B}}j_{A *}\phi_{\Xi_1 \Xi_2 A}"] \\
            \bigoplus_{\Xi'\sim_A\Xi_2,\Xi'\leq \Xi_2}\Fc_{\Xi'} \ar[r] & j_{A *}j_A^*\Fc_{\Xi_2}.
        \end{tikzcd}
    \end{equation}
    Clearly, the image of the morphism in the lower row contains the image of the morphism in the upper row under the isomorphism on the right side. Take $\Xi'\sim_A\Xi_2$ with $\Xi'\leq \Xi_2$ but $\Xi'\not\leq \Xi_1$. We want to show that the image of $\Fc_{\Xi'}$ under the morphism in the lower row is also contained in the image of the morphism in the upper row. Consider the following part of the direct sum in the upper left corner
    \begin{equation*}
        \bigoplus_{\substack{\Xi_1\leq \Xi_2, B(\Xi_1)=B\\\Xi''\sim_A\Xi_1,\Xi''\leq \Xi_1\\ B(\Xi'')=B\cap B(\Xi')\\\Xi''\leq \Xi'}}\Fc_{\Xi''}=\bigoplus_{\substack{\Xi''\leq \Xi_2,\Xi''\sim_A\Xi_2\\ B(\Xi'')=B\cap B(\Xi')\\\Xi''\leq \Xi'}}\Fc_{\Xi''}.
    \end{equation*}
    Here the equality holds, because: we sum over all $\Xi_1$, so that all $\Xi''$ on the right-hand side are also present on the left-hand side, using that $I$ is closed under addition; the conditions $B(\Xi_1)=B$ and $\Xi''\leq \Xi_1$ determine $\Xi_1$ uniquely given $\Xi''$; and we have $\Xi_1\sim_A \Xi_2$ by assumption. Now note that for $\Xi''$ in this direct sum
    \begin{equation*}
        B(\Xi')\cap C=B(\Xi_2)\cap B(\Xi')\cap C=B\cap B(\Xi')\cap C=B(\Xi'')\cap C,
    \end{equation*}
    where the first equation holds because $\Xi'$ is subordinate to $\Xi_2$, the second one holds by the assumption $B\sim_C B(\Xi_2)$, and the last one holds by the condition on $\Xi''$ in the above sum. Hence, we have $\Xi''\sim_C \Xi'$. By $C\leq A$ and $\Xi'\sim_A \Xi_2$ by assumption, we get $\Xi''\sim_A\Xi_2$ automatically. So, the above direct sum is just
    \begin{equation*}
        \bigoplus_{\Xi''\leq \Xi', B(\Xi'')=B\cap B(\Xi')}\Fc_{\Xi''}.
    \end{equation*}
    Since we have $\Xi''\sim_C \Xi'$ for every summand, the isomorphism property \eqref{fact:eq:level_system_isomorphism_property} gives us that
    \begin{equation}\label{fact:eq:D_sum_partial_isom}
        \bigoplus_{\Xi''\leq \Xi', B(\Xi'')=B\cap B(\Xi')}\phi_{\Xi'' \Xi' C}: \restr{\left(\bigoplus_{\Xi''\leq \Xi', B(\Xi'')=B\cap B(\Xi')}\Fc_{\Xi''}\right)}{U_C} \to \restr{\Fc_{\Xi'}}{U_C}
    \end{equation}
    is an isomorphism. This commutes with adjunction and by assumptions (a) and (b) of the level system $(\Fc,\phi)$, it commutes with the $\phi_A$ in the commutative diagram \eqref{fact:eq:D_sum_cd}. But that means, using the isomorphism \eqref{fact:eq:D_sum_partial_isom}, the image of $\Fc_{\Xi'}$ under the morphism in the lower row of \eqref{fact:eq:D_sum_cd} is also contained in the image of the morphism in the upper row of \eqref{fact:eq:D_sum_cd}. Therefore, the desired isomorphism property \eqref{fact:eq:level_system_isomorphism_property} holds for $\left(D(\Fc),D(\phi)\right)$.
\end{proof}

\begin{lemma}\label{fact:lemma:strict_generation_ktheory}
    Taking K-theory classes of level systems to be defined as the collection of K-theory classes $\left(\left[\Fc_\Xi\right]\right)_\Xi$, we have that the K-theory of level systems is generated by strict level systems.
\end{lemma}
\begin{proof}
    As our level systems are special cases of equivariant systems in the sense of \cite{kr}, their proof of this generation applies. It remains to check that our modified isomorphism property \eqref{fact:eq:level_system_isomorphism_property} is preserved by the steps in their proof. Their proof uses the exact sequence
    \begin{equation*}
        0 \to \ker\left(I_A\right) \to \left(\Fc,\phi\right) \xrightarrow{I_A} D_A\left(\Fc,\phi\right) \to \cok\left(I_A\right)\to 0,
    \end{equation*}
    which allows us to express the K-theory class $\left[\Fc,\phi\right]$ as
    \begin{equation*}
        \left[\Fc,\phi\right]= \left[D_A\left(\Fc,\phi\right)\right] - \left[\cok\left(I_A\right)\right] + \left[\ker\left(I_A\right)\right].
    \end{equation*}
    Now $D_A\left(\Fc,\phi\right)$ is strict by definition, and both $\cok\left(I_A\right)$ and $\ker\left(I_A\right)$ are "smaller" in a way \cite{kr} make precise. Then a Noetherian induction argument finishes the proof.

    For us, it simply remains to check that $D_A\left(\Fc,\phi\right)$, $\cok\left(I_A\right)$, and $\ker\left(I_A\right)$ all satisfy the modified isomorphism property \eqref{fact:eq:level_system_isomorphism_property}. For $D_A\left(\Fc,\phi\right)$ we already checked this in the previous Lemma \ref{fact:lemma:D_I_preserve_isom_property}. Since kernels and cokernels are taken for each $\Xi$ individually, $\cok\left(I_A\right)$ and $\ker\left(I_A\right)$ also satisfy \eqref{fact:eq:level_system_isomorphism_property}.
\end{proof}

\subsubsection{Acyclicity}\label{fact:sec:combinatorics}

Note that the complex $\Gc_\alpha$ was defined only using the elements of the $\alpha$-level system defined in Example \ref{fact:ex:fact_level_system}, not the entire data of the factorizable system. By Lemma \ref{fact:lemma:strict_generation_ktheory} we may assume that this $\alpha$-level system is strict, while still satisfying the computation of the K-theory class of $\Gc_\alpha$ in Lemma \ref{fact:lemma:complex_ktheory_class}. Note that using Lemma \ref{fact:lemma:strict_generation_ktheory} means that we may no longer use that the $\Fc_\Xi$ were constructed as $\boxtimes$ of pieces $\Fc_\beta$.

In the constructions that follow we have to consider the root node, leaves, and the following different types of nodes of an index tree.
\begin{definition}
    Nodes in an index tree can be of the following different (but not disjoint) types:
    \begin{itemize}
        \item A node is \textbf{exceptional} if it is not a leaf, but all of its children are leaves. Note that the root can be exceptional.
        \item A node is \textbf{ordinary} if it is neither the root, nor a leaf, nor exceptional.
        \item A node is \textbf{relevant} if it is either ordinary, or exceptional, but not a root. 
    \end{itemize}
\end{definition}

 We need to define some signs to keep track of the ordering under tree modifications.

\begin{definition}
    Let $(T,\leq)$ be an ordered index tree, and let $v$ be a node of $T$. We set $s_\leq(v)=(-1)^k$, where $k$ is the number of relevant nodes in $T$ strictly smaller than $v$.
\end{definition}

We now consider possible contractions of an index tree. When we consider an $\alpha$-index tree as a way to track sequences of subordinate partitions of $[\alpha]$, then the contractions are exactly the operations, which modify one part of this sequence.

\begin{definition}
    Let $T$ be an $\alpha$-index tree and let $v$ be a relevant node of $T$. We define the \textbf{ordinary contraction} $O(T,v)$ to be the $\alpha$-index tree, which is obtained from $T$ by deleting $v$ and connecting all of its children directly to its parent.

    Let $T$ be an $\alpha$-index tree and let $v$ be an exceptional node of $T$. Let $v_1,\dots,v_k$ be the children of $v$ with associated labels $(B_1,\tilde{\alpha}_1),\dots,(B_k,\tilde{\alpha}_k)$. We define the \textbf{exceptional contraction} $E(T,v)$ to be the $\alpha$-index tree, which is obtained from $T$ by deleting all children of $v$ and labelling $v$, which is now a leaf, by $(\bigcup_{i=1}^k B_i,\sum_{i=1}^k \tilde{\alpha}_i)$ of all its children's labels.

    If $T$ has an order $\leq$, denote by $\leq_v$ the order obtained by restriction to the ordinary or exceptional contraction.
\end{definition}

\begin{definition}
    Let $(\Fc,\phi)$ be a strict $\alpha$-level system and $(T,\leq)$ an ordered $\alpha$-index tree.

    Given a relevant node $v$ in $T$, we define the \textbf{ordinary contraction homomorphism} to be
    \begin{equation*}
        o(T,\leq,v):\Fc_{T,\leq} \xrightarrow{s_\leq(v)\mathrm{id}} \Fc_{O(T,v),\leq_v}.
    \end{equation*}

    Given an exceptional node $v$ in $T$, let $(E(T,v),\leq_v)$ be the exceptional contraction. Let $\Xi_1$ be the partition of $[\alpha]$ associated to the leaves of the index tree $T$, and let $\Xi_2$ be the partition associated to the leaves of $E(T,v)$. Note that $\Xi_1$ is subordinate to $\Xi_2$ by construction, and that $\leq$ and $\leq_v$ order the partitions in a compatible manner. We define the \textbf{exceptional contraction homomorphism} to be
    \begin{equation*}
        e(T,v): \Fc_{T,\leq}=\Fc_{\Xi_1} \xrightarrow{-s_\leq(v)\phi_{\Xi_1\Xi_2}} \Fc_{\Xi_2}=\Fc_{E(T,v),\leq_v}.
    \end{equation*}
\end{definition}

With these contraction homomorphisms we can define the following complex, which we also denote $\Gc_\alpha$ by abuse of notation. Note here that the terms of the complex below are the same as for the complex from Definition \ref{fact:def:tree_complex}, except that the $\Fc_{T,\leq_T}$ have been replaced using the results of Section \ref{fact:sec:good_replacement} to make $(\Fc,\phi)$ a strict $\alpha$-level system. So we may no longer assume $\Fc_{T,\leq_T}$ to be a tensor product of various $\Fc_\beta$, but have to treat them like abstract $S_{b(\alpha)}$-equivariant (2-periodic) complexes of sheaves on $X^{b(\alpha)}$ with certain homomorphisms $\phi_{\Xi_1\Xi_2}$, which we can use to define the differential.

\begin{definition}\label{fact:def:tree_complex_differential}
    Let $I$ be a factorization index semigroup. Take $\alpha\in I$, and let $(\Fc,\phi)$ be a strict $\alpha$-level system. For $k\geq 0$, we define the complex of sheaves on $X^{b(\alpha)}$
    \begin{equation*}
        \Gc_{\alpha}^{-k} := \bigoplus_{T\in\Tc(\alpha,k)} \Fc_{T,\leq_T},
    \end{equation*}
    where we recall that $\Tc(\alpha,k)$ is the set of $\alpha$-index trees with $k$ non-leaf nodes, and $\leq_T$ is the standard ordering of a given index tree. We define the (double) complex
    \begin{equation*}
        \Gc_{\alpha}:=(\Gc_{\alpha}^\bullet,\mathrm{d}),
    \end{equation*}
    where the differential is given by the sum of all possible ordinary and exceptional contraction homomorphisms.
\end{definition}

\begin{lemma}
    $(\Gc_{\alpha}^\bullet,\mathrm{d})$ is a complex, i.e. $\mathrm{d}^2=0$.
\end{lemma}
\begin{proof}
    The proof relies purely on tree combinatorics and sign computations, so the proof in \cite{kr} works verbatim.
\end{proof}

When modifying the $\alpha$-level system to be strict using Lemma \ref{fact:lemma:strict_generation_ktheory}, we preserved the K-theory class, so Lemma \ref{fact:lemma:complex_ktheory_class} still holds for this redefinition of $\Gc_\alpha$.

\subsubsection{Acyclicity}

We can now study the support of the complex $\Gc_\alpha$ with its differentials.

\begin{lemma}\label{fact:lemma:S_support_lemma_G_complex}
    Let $U$ be the non-stacky locus as an open subscheme of $X$, and let $S$ be its complement. Assume that the factorizable system $\Fc_{\alpha}$ is given such that for any $\alpha\in I\setminus \Delta_I$, $\restr{\Fc_{\alpha}}{U^{b(\alpha)}}$ vanishes. Then, for $\alpha\in I\setminus \Delta_I$, the complex $\Gc_{\alpha}$ constructed above also vanishes on $U^{b(\alpha)}$.
\end{lemma}

\begin{proof}
    Take $\alpha\in I\setminus \Delta_I$. First note that any partition of $\alpha$ must have at least one element in $I\setminus \Delta_I$, so all elements $\Fc_{\Xi}$ of the $\alpha$-level system of Example \ref{fact:ex:fact_level_system} vanish on $U^{b(\alpha)}$. The functors $D_A$ and morphisms $I_A$ of Definition \ref{fact:def:D_I_strictness_proof} preserve such vanishing on an open subscheme $U^{b(\alpha)}$ of $X^{b(\alpha)}$. Hence, the procedure in Lemma \ref{fact:lemma:strict_generation_ktheory} preserves this vanishing, so that the the K-theory class of $(\Fc,\phi)$ is a sum of K-theory classes of strict level systems vanishing on $U^{b(\alpha)}$. Therefore, the complex $\Gc_{\alpha}$ also vanishes on $U^{b(\alpha)}$.
\end{proof}

\begin{lemma}\label{fact:lemma:complex_support_diagonal}
    Away from the small diagonal $X\hookrightarrow X^{b(\alpha)}$, the complex $\Gc_{\alpha}$ is acyclic.
\end{lemma}
\begin{proof}
    First note that acyclicity is independent of the $S_{b(\alpha)}$-equivariant structure, so we may only work with the underlying complexes of sheaves here. Secondly, there is nothing to show for $\alpha$ with $b(\alpha)=1$, so we may assume $b(\alpha)>1$.  
    
    Now let $p\in X^{b(\alpha)}$ be a point away from the small diagonal. That means, there exists a partition $\overline{B}=\{\overline{B}_1,\overline{B}_2\}$ of $[b(\alpha)]$ (we assume it is standard ordered) such that $p$ is in the the open subscheme $U_{\overline{B}_1\overline{B}_2} = U_{\overline{B}}$ in $X^{b(\alpha)}$.

    In order to prove acyclicity at $p$, we want to construct a map
    \begin{equation*}
        \Upsilon:\Tc(\alpha)\to J
    \end{equation*}
    from $\Tc(\alpha)$ to a totally ordered set $(J,\leq)$, such that for every contraction $T\to T'$ we have $\Upsilon(T)\geq \Upsilon(T')$. Such a map gives us a filtration of $\Gc_\alpha$ with subquotients of the form, each for some $j\in J$,
    \begin{equation*}
        Q_j\Gc_{\alpha} = \bigoplus_{T\in\Upsilon^{-1}(j)} \Fc_{T,\leq_T},
    \end{equation*}
    with differentials given by the sum of all contractions preserving $\Upsilon$. We want to find such a function $\Upsilon$ so that all such subquotients are acyclic. As extensions of acyclic complexes are acyclic, this would imply acyclicity of $\Gc_\alpha$ at $p$. Note that the function $\Upsilon$ may depend on $p$.

    Set
    \begin{equation*}
        2^{[\alpha]}_I \coloneqq \left\{(B,\tilde{\beta})\ |\ \Xi\in P_I[\alpha],\ (B,\tilde{\beta})\in\Xi\right\}.
    \end{equation*}
    Write $(B_2,\tilde{\beta}_2)\subseteq(B_1,\tilde{\beta}_1)$ if $B_2\subseteq B_1$ and $(B_2,\tilde{\beta}_2)\in 2^{\left[\left(\abs{B_1},\tilde{\beta}_1\right)\right]}_I$. Choose a total ordering on $2^{[\alpha]}_I$, such that for any $(B_1,\tilde{\beta}_1)\in 2^{[\alpha]}_I$ and $(B_2,\tilde{\beta}_2)\subseteq(B_1,\tilde{\beta}_1)$ we have $(B_2,\tilde{\beta}_2)\geq (B_1,\tilde{\beta}_1)$. We take our totally ordered set to be
    \begin{equation*}
        J\coloneqq 2^{[\alpha]}_I\times \Z^4
    \end{equation*}
    with the lexicographical total order. Now we define the desired function
    \begin{equation*}
        \Upsilon=(\Upsilon_0,\Upsilon_1,\Upsilon_2,\Upsilon_3,\Upsilon_4):\Tc(\alpha)\to J.
    \end{equation*}
    Let $T\in\Tc(\alpha)$ be an $\alpha$-index tree. Set
    \begin{equation*}
        B(T)\coloneqq \left\{\left(B_v,\tilde{\alpha}_v\right)\ |\ v\in T\right\}\subset 2^{[\alpha]}_I,
    \end{equation*}
    and consider the subset $R(T)\subseteq B(T)$ where $B_v$ is neither contained in $\overline{B}_1$ nor $\overline{B}_2$. Take $d(T)$ to be the node of $T$ for which $\left(B_{d(T)},\tilde{\alpha}_{d(T)}\right)\in 2^{[\alpha]}_I$ is in $R(T)$ and maximal in $R(T)$ with respect to the total order on $2^{[\alpha]}_I$. Since the root node is always contained in $R(T)$, this is well-defined. We set
    \begin{equation*}
        \Upsilon_0(T)\coloneqq \left(B_{d(T)},\tilde{\alpha}_{d(T)}\right).
    \end{equation*}

    For two nodes $v,v'$ in $T$, write $v\prec v'$ if $v$ is a proper descendant of $v'$. To define the other components, note first that for every descendant $v\prec d(T)$, we have $\left(B_v,\tilde{\alpha}_v\right)\subsetneq \left(B_{d(T)},\tilde{\alpha}_{d(T)}\right)$ and in particular $\left(B_{v},\tilde{\alpha}_{v}\right) > \left(B_{d(T)},\tilde{\alpha}_{d(T)}\right)$ by our choice of the total order on $2^{[\alpha]}_I$. Hence, $\left(B_{v},\tilde{\alpha}_{v}\right)$ cannot be in $R(T)$ by the defining maximality assumption of $d(T)$. This means, every descendant $v\prec d(T)$ satisfies either 
    \begin{equation}\label{fact:eq:dT_desc_inclusions}
        B_v\subseteq \overline{B}_1\cap B_{d(T)}\text{ or }B_v\subseteq \overline{B}_2\cap B_{d(T)}.
    \end{equation}
    This lets us define
    \begin{align*}
        \Upsilon_1(T) &\coloneqq \abs{\left\{v\in T\ |\ v\not\preceq d(T)\right\}},\\
        \Upsilon_2(T) &\coloneqq \abs{\left\{v\prec d(T)\ |\ B_v\subsetneq \overline{B}_1\cap B_{d(T)}\right\}},\\
        \Upsilon_3(T) &\coloneqq \Upsilon_2(T)\cdot\abs{\left\{v\prec d(T)\ |\ B_v\subseteq \overline{B}_2\cap B_{d(T)}\right\}},\\
        \Upsilon_4(T) &\coloneqq \abs{\left\{v\prec d(T)\ |\ B_v\subsetneq \overline{B}_2\cap B_{d(T)}\right\}}.
    \end{align*}

    Now we examine this map and its induced filtration on $\Gc_\alpha$. First, we show that for every contraction $T\to T'$ we have $\Upsilon(T)\geq \Upsilon(T')$. Take a contraction $T\to T'$. Then we have $B(T')\subseteq B(T)$, giving us $\Upsilon_0(T)\geq\Upsilon_0(T')$ immediately. Assume $\Upsilon_0(T)=\Upsilon_0(T')$, which just means $\left(B_{d(T)},\tilde{\alpha}_{d(T)}\right)=\left(B_{d(T')},\tilde{\alpha}_{d(T')}\right)$. Now we consider where the contraction could be. If it contracts a node $v\not\preceq d(T)$, then $\Upsilon_1(T)> \Upsilon_1(T')$ directly by definition. If $d(T)$ is contracted, then by $\left(B_{d(T)},\tilde{\alpha}_{d(T)}\right)=\left(B_{d(T')},\tilde{\alpha}_{d(T')}\right)$, the contraction must be exceptional, so $\Upsilon_1(T)=\Upsilon_1(T')$. But in this case, we also directly get $\Upsilon_i(T')=0$ for $i=2,3,4$, which gives $\Upsilon(T)\geq \Upsilon(T')$. Contracting a node $v\prec d(T)$ also immediately gives $\Upsilon_1(T)=\Upsilon_1(T')$. In this case, we get
    \begin{equation*}
        \left\{\left(B_{v'},\tilde{\alpha}_{v'}\right)\ |\ v'\prec d(T')\right\}\subsetneq 
        \left\{\left(B_{v},\tilde{\alpha}_{v}\right)\ |\ v\prec d(T)\right\},
    \end{equation*} 
    which implies $\Upsilon_i(T)\geq\Upsilon_i(T')$ for $i=2,3,4$. In particular, we see that for a contraction $T\to T'$ with $\Upsilon(T)=\Upsilon(T')$, the contraction must be either
    \begin{equation}\label{fact:eq:upsilon_pres_contr_cases}
        \text{exceptional at }d(T)\text{ or at }v\prec d(T).
    \end{equation}

    Hence, $\Upsilon$ defines a filtration on $\Gc_\alpha$ with subquotients of the form, each for some $j\in J$,
    \begin{equation*}
        Q_j\Gc_{\alpha} = \bigoplus_{T\in\Upsilon^{-1}(j)} \Fc_{T,\leq_T},
    \end{equation*}
    with differentials given by the sum of all contractions preserving $\Upsilon$. The following Lemma \ref{fact:lemma:bijection_grading_tree} shows that the differential of $Q_j\Gc_{\alpha}$, which is built from the sums of $\Upsilon$-preserving contraction homomorphisms, is an isomorphism at $p$, making the subquotients $Q_j\Gc_{\alpha}$ acyclic at $p$. This means $\Gc_{\alpha}$ is acyclic at $p$.
\end{proof}

\begin{lemma}\label{fact:lemma:bijection_grading_tree}
    Take notation as in the above proof of Lemma \ref{fact:lemma:complex_support_diagonal}. The contractions preserving $\Upsilon$ come in collections of contractions $T_{1,k}\to T_2$, such that
    \begin{itemize}
        \item no $T_{1,k}$ has an $\Upsilon$-preserving contraction $T''\to T_{1,k}$, 
        \item $T_{1,k}$ has no other $\Upsilon$-preserving contractions $T_{1,k}\to T'$, 
        \item $T_2$ has no $\Upsilon$-preserving contraction $T_2\to T'$, and
        \item $T_{1,k}\to T_2$ are all $\Upsilon$-preserving contractions into $T_2$.
    \end{itemize}
    The sum of the associated contraction homomorphisms
    \begin{equation*}
        \bigoplus_k \Fc_{T_{1,k}} \to \Fc_{T_2}
    \end{equation*}
    is an isomorphism at $p$.
\end{lemma}
\begin{proof}
    Take a tree $T$ and study its possible $\Upsilon$-preserving contractions. In the previous proof in \eqref{fact:eq:upsilon_pres_contr_cases}, we noted that a contraction $T\to T'$ preserving $\Upsilon$ must be either
    \begin{enumerate}[(A)]
        \item exceptional at $d(T)$, or
        \item a contraction at some $v\prec d(T)$.
    \end{enumerate}
    In case (A), $d(T')$ is the leaf that $d(T)$ is contracted to, so that $\Upsilon_i(T')=0$ for $i=2,3,4$. As $\Upsilon$ is preserved, that means $\Upsilon_i(T)=0$ for $i=2,3,4$. By definition, this means every $v\prec d(T)$ must have $B_v=\overline{B}_1\cap B_{d(T)}$ or $B_v=\overline{B}_1\cap B_{d(T)}$. But that means $d(T)$ must have exactly two children $v_1$ and $v_2$, both leaves, with
    \begin{equation*}
        B_{v_i}=\overline{B}_i\cap B_{d(T)}.
    \end{equation*}

    In case (B), $\Upsilon_2(T)=\Upsilon_2(T')$ implies that the contraction is either
    \begin{enumerate}[(B1)]
        \item ordinary at a node $v$ with $B_v=\overline{B}_1\cap B_{d(T)}$, or
        \item at a node $v$ with $B_v\subseteq\overline{B}_2\cap B_{d(T)}$ by \eqref{fact:eq:dT_desc_inclusions}.
    \end{enumerate}
    In case (B2), we have 
    \begin{equation*}
        \abs{\left\{v\prec d(T)\ |\ B_v\subseteq \overline{B}_2\cap B_{d(T)}\right\}}>\abs{\left\{v\prec d(T')\ |\ B_v\subseteq \overline{B}_2\cap B_{d(T')}\right\}},
    \end{equation*}
    but this, together with $\Upsilon_2(T)=\Upsilon_2(T')$ and $\Upsilon_3(T)=\Upsilon_3(T')$, by definition of $\Upsilon_3$, implies $\Upsilon_i(T)=\Upsilon_i(T')=0$ for $i=2,3$. But then $\Upsilon_4(T)=\Upsilon_4(T')$ implies that the contraction must be ordinary at a node $v$ with $B_v=\overline{B}_2\cap B_{d(T)}$. We now show that these cases cannot happen at the same time. First, since both (B1) and (B2) are ordinary contractions at proper descendants of $d(T)$ neither of them can occur at the same time as (A). Secondly, if (B1) occurs, then, as the contraction is ordinary at a node $v\prec d(T)$ with $B_v=\overline{B}_1\cap B_{d(T)}$, it must have descendants $v'$ with $B_{v'}\subsetneq\overline{B}_1\cap B_{d(T)}$, such that $\Upsilon_2(T)\neq 0$, contradicting case (B2). Analogously, if (B2) occurs, then by $\Upsilon_2(T)= 0$, there must be a leaf $v_1\prec d(T)$ with $B_{v_1}=\overline{B}_1\cap B_{d(T)}$. In summary, we have three distinct possible cases for contractions preserving $\Upsilon$:
    \begin{enumerate}
        \item[(A)] An exceptional contraction at $d(T)$ if $d(T)$ is exceptional with two children.
        \item[(B1)] An ordinary contraction at a relevant $v\prec d(T)$ with $B_v=\overline{B}_1\cap B_{d(T)}$ if one exists.
        \item[(B2)] An ordinary contraction at a relevant $v_2\prec d(T)$ with $B_{v_2}=\overline{B}_2\cap B_{d(T)}$ if one exists and there is a leaf $v_1\prec d(T)$ with $B_{v_1}=\overline{B}_1\cap B_{d(T)}$.
    \end{enumerate}
    If a tree $T$ is in any of the above cases, and it admits an $\Upsilon$-preserving contractions from another tree $T'$, the tree $T'$ again has to satisfy one of the above cases. We can check case-wise that this is not possible, so no tree $T$ in the above cases admits an $\Upsilon$-preserving contraction $T'\to T$. If $T$ falls into neither of these categories, then we have three distinct possibilities
    \begin{enumerate}
        \item[(CA)] $d(T)$ is a leaf.
        \item[(CB1)] $d(T)$ is not a leaf, but no $v\prec d(T)$ has $B_v=\overline{B}_1\cap B_{d(T)}$.
        \item[(CB2)] There is a leaf $v_1\prec d(T)$ with $B_{v_1}=\overline{B}_1\cap B_{d(T)}$, but no $v_2\prec d(T)$ with $B_{v_2}=\overline{B}_2\cap B_{d(T)}$.
    \end{enumerate}
    In each of the cases (CB1) and (CB2) there exists a unique tree $T'$ that fits into cases (B1) and (B2) respectively, together with a contraction $T'\to T$. Here we use that $I$ is closed under addition. In case (CA), consider $B_{d(T)}$ and $\tilde{\alpha}_{d(T)}$. For every splitting $\tilde{\alpha}_1+\tilde{\alpha}_2=\tilde{\alpha}_{d(T)}$, there exist a unique tree $T'$ of type (A) with a contraction to $T$. The tree $T'$ is built from $T$ by attaching two children $v_1,v_2$ to the leaf $d(T)$ with $B_{v_i}=\overline{B}_i\cap B_{d(T)}$ and $\tilde{\alpha}_{v_i}=\tilde{\alpha}_i$. Here we use that, by definition $B_{d(T)}$ contains at least two elements, one in $\overline{B}_1$ and one in $\overline{B}_2$, so that such splittings $\tilde{\alpha}_1+\tilde{\alpha}_2=\tilde{\alpha}_{d(T)}$ exist by $I$ being a factorization index set. These are all possible trees $T'$ contracting to $T$ while preserving $\Upsilon$.
    
    It remains to check that the associated (sums of) contraction homomorphisms are isomorphisms. Take $T\to T'$ an $\Upsilon$-preserving contraction of case (B1) or (B2). Because the contractions in cases (B1) and (B2) are ordinary, their associated contraction homomorphism $\Fc_{T'}\to \Fc_{T}$ are global isomorphisms.

    Consider a tree $T'$ in case (CA), together with all contractions $T_{\tilde{\alpha}_1,\tilde{\alpha}_2}\to T'$ in case (A), indexed by splittings $\tilde{\alpha}_1+\tilde{\alpha}_2=\tilde{\alpha}_{d(T')}$. Take $\Xi'$ to be the partition of $[\alpha]$ defined by $T'$ and $\Xi_{\tilde{\alpha}_1,\tilde{\alpha}_2}$ the partition of $[\alpha]$ defined by $T_{\tilde{\alpha}_1,\tilde{\alpha}_2}$. Then by construction $B\left(\Xi_{\tilde{\alpha}_1,\tilde{\alpha}_2}\right)=B(\Xi')\cap \overline{B}$, so by definition $B\left(\Xi_{\tilde{\alpha}_1,\tilde{\alpha}_2}\right)\sim_{\overline{B}}B(\Xi')$. Then the sum of contraction homomorphisms
    \begin{equation*}
        \bigoplus_{\tilde{\alpha}_1+\tilde{\alpha}_2=\tilde{\alpha}_{d(T')}} \Fc_{T_{\tilde{\alpha}_1,\tilde{\alpha}_2}} \to \Fc_{T'}
    \end{equation*}
    is an isomorphism at $p$ by the isomorphism property \eqref{fact:eq:level_system_isomorphism_property} of the level system, because $p$ is in $U_{\overline{B}}$ by assumption.
\end{proof}

Now we can prove Theorem \ref{fact:thm:fact}.

\begin{proof}[Proof of Theorem \ref{fact:thm:fact}]
    Given a factorizable system $\Fc_{\alpha}$ on $\sym^{b(\alpha)}(X)$, by Lemma \ref{fact:lemma:comp_setup_cd} and Example \ref{fact:ex:sym_to_Xb}, we can pull back to $[X^{b(\alpha)}/S_{b(\alpha)}]$ to get a factorizable system of $S_{b(\alpha)}$-equivariant (2-periodic) complexes of sheaves $\overline{\Fc}_{\alpha}$ on $X^{b(\alpha)}$. Let $\overline{\Gc}_{\alpha}$ be the $S_{b(\alpha)}$-equivariant complexes defined $\overline{\Fc}_{\alpha}$ in Definition \ref{fact:def:tree_complex}. By definition of the plexthystic expontential and Lemma \ref{fact:lemma:complex_ktheory_class}, we have
    \begin{eqnarray*}
        \pexpI{\sum_{\alpha\in I}q^{\alpha} [\overline{\Gc}_{\alpha}]} =& \sum_{\alpha\in I} q^{\alpha} \sum_{\lambda\in P(\alpha)} [\overline{\Gc}_\lambda]\\
        =& \sum_{\alpha\in I} q^{\alpha} \sum_{\Xi\in P[\alpha]} [\overline{\Gc}_{\Xi}]\\
        =& 1 + \sum_{\alpha\in I} q^{\alpha} [\overline{\Fc}_{\alpha}].
    \end{eqnarray*}
    Now we push forward along $p:[X^{b(\alpha)}/S_{b(\alpha)}] \to \sym^{b(\alpha)}(X)$ again to get
    \begin{equation}\label{fact:eq:proof_ktheory_id}
        1 + \sum_{\alpha\in I} q^{\alpha} [p_*p^*\Fc_{\alpha}] = 1 + \sum_{\alpha\in I} q^{\alpha} p_*[\overline{\Fc}_{\alpha}] = p_*\pexpI{\sum_{\alpha\in I}q^{\alpha} [\overline{\Gc}_{\alpha}]}
    \end{equation}
    by Definition \ref{fact:def:pexp}. Now $p_*p^*\Fc_{\alpha}$ is just $\Fc_{\alpha}$ by general properties of the coarse moduli space $p$. 
    
    Now we want to study the support of the complex $\overline{\Gc}_\alpha$. We may focus on each $\alpha$ individually. Each $\overline{\Gc}_\alpha$ may be constructed from just the $\alpha$-level system of Example \ref{fact:ex:fact_level_system} instead of the entire factorizable system. By Lemma \ref{fact:lemma:strict_generation_ktheory}, we may assume this level system is strict and equip $\overline{\Gc}_\alpha$ with a differential as in Definition \ref{fact:def:tree_complex_differential} while preserving the K-theory class identity \eqref{fact:eq:proof_ktheory_id}.
    
    By Lemma \ref{fact:lemma:complex_support_diagonal}, $\overline{\Gc}_{\alpha}$ is supported on the small diagonal. Hence, the pushforwards $p_*\overline{\Gc}_{\alpha}$ are also supported on the small diagonal. So, by d\'evissage \cite[Appl. 6.4.2]{weibel_kbook}, there exist classes $[\Gc_{\alpha}]$ on $X$ such that
    \begin{equation*}
        i_*[\Gc_{\alpha}] = p_*[\overline{\Gc}_{\alpha}],
    \end{equation*}
    where $i:X \to \sym^{b(\alpha)}(X)$ is the diagonal embedding. This gives us
    \begin{align*}
        1 + \sum_{\alpha\in I} q^{\alpha} [\Fc_{\alpha}] &= p_*\pexpI{\sum_{\alpha\in I}q^{\alpha} [\overline{\Gc}_{\alpha}]}\\
        &= \pexpI{\sum_{\alpha\in I}q^{\alpha} p_*[\overline{\Gc}_{\alpha}]}\\
        &= \pexpI{\sum_{\alpha\in I}q^{\alpha} i_*[\Gc_{\alpha}]}\\
        &= \pexpI{\sum_{\alpha\in I}q^{\alpha} [\Gc_{\alpha}]},
    \end{align*}
    where we used the identity \eqref{fact:eq:proof_ktheory_id} for the first equality, Lemma \ref{fact:lemma:lambda_pieces_push_pull_identities}(a) for the second equality, d\'evissage for the third equality, and the definition of the plethystic exponential in Definition \ref{fact:def:pexp} for the last equality.

    Additionally, if the factorizable system $\Fc_{\alpha}$ satisfies that for $\alpha\in I\setminus \Delta_I$, $\restr{\Fc_{\alpha}}{\sym^{b(\alpha)}(U)}$ vanishes, then $\restr{\overline{\Fc}_{\alpha}}{U^{b(\alpha)}}$ also vanishes and Lemma \ref{fact:lemma:S_support_lemma_G_complex} shows that $\overline{\Gc}_\alpha$ vanishes on $U^{b(\alpha)}$ for $\alpha\in I\setminus \Delta_I$. Note that the intersection of the small diagonal with the complement of $U^{b(\alpha)}$ is just $S$. So, by d\'evissage \cite[Appl. 6.4.2]{weibel_kbook} as above, the classes $[\Gc_{\alpha}]$ are pushed forward from $S$ in this case.
\end{proof}

\subsection{Compatible Factorization}

We now want to compare factorizations of sheaves on different orbifolds in some way. We have the following example in mind.

\begin{example}
    We consider $\Xc=[\C^3/\mu_r]$ with $\mu_r$ acting with weights $(1,r-1,0)$. Let $Y$ be a crepant resolution of $\Xc$. Let $U$ be the non-stacky locus of $\Xc$. It also embeds an open subscheme in $Y$. This gives us the following commutative diagram
    \begin{equation*}
        \begin{tikzcd}
            \hilb^n(Y) \arrow[d] & \hilb^n(U) \arrow[l] \arrow[r, "\sim"] \arrow[d] & \hilb^{(n,\dots,n)}(U) \arrow[r] \arrow[d] & \hilb^{(n,\dots,n)}(\Xc) \arrow[d] \\
            \sym^n(Y) & \sym^n(U) \arrow[l] \arrow[r, "\sim"] & M_{(n,\dots,n)}(U) \arrow[r] & M_{(n,\dots,n)}(\Xc).
        \end{tikzcd}
    \end{equation*}
    Now the twisted virtual structure sheaves $\hilb^n(Y)$ and $\hilb^{(n,\dots,n)}(\Xc)$ are compatible under this identification, meaning they pull back to isomorphic sheaves on $\hilb^n(U)$. In this section, we develop machinery to compare the classes $\left[\restr{\Gc^Y_n}{U}\right]$, $\left[\Gc^U_n\right]$, $\left[\Gc^U_{(n,\dots,n)}\right]$, and $\left[\restr{\Gc^\Xc_{(n,\dots,n)}}{U}\right]$ of Theorem \ref{fact:thm:fact}. 
\end{example}

Consider an open embedding $j:\Yc \to \Xc$ of stacks. Take $Y$ and $X$ to be the coarse moduli spaces of $\Yc$ and $\Xc$ respectively. They come with an open embedding $Y\to X$, which we also denote by $j$ by abuse of notation. Assume that $j_*$ restricts to an embedding of factorization index semigroups $I_j: I_{\Yc}\hookrightarrow I=I_{\Xc}$, so that
\begin{equation*}
    \begin{tikzcd}
        I_\Yc \ar[r, hookrightarrow, "I_j"]\ar[d, hookrightarrow] & I_\Xc \ar[d, hookrightarrow]\\
        \Nr_0(\Yc) \ar[r, "j_*"]\ar[dr,"b"] & \Nr_0(\Xc)\ar[d,"b"]\\
        & \Nr_0(Y)=\Nr_0(X)=\Z 
    \end{tikzcd}
\end{equation*}
commutes. We assume that if $\alpha_1+\alpha_2$ is in $I_\Yc$, then $\alpha_1\in I_\Yc$ or $\alpha_2\in I_\Yc$ implies the other is as well. We assume further that for $\alpha\in I_\Xc$, the pullback of K-theory classes $j^*\alpha$ lands in $I_\Yc$. From now on we identify $I_\Yc$ with a subset of $I=I_\Xc$ via the above embedding.

We consider factorizable systems in the sense of Example \ref{fact:ex:sym_coarse}. We have an open embedding morphism
\begin{equation*}
    M(j):\sym(Y) \to \sym(X),
\end{equation*}
such that, for any $b\in \Z$, the component $\sym^b(Y)$ gets sent to the component $\sym^b(X)$.

For $b_1,b_2\in \Z$, we naturally get a morphism $U(j)$ making the diagram
\begin{equation}
    \begin{tikzcd}
        \sym^{b_1}(Y)\times \sym^{b_2}(Y) \ar[d, "M(j)\times M(j)", open] & U^Y_{b_1 b_2}\ar[d,"U(j)", open]\ar[l, open']\ar[r, open] & \sym^{b_1+b_2}(Y)\ar[d,"M(j)", open]\\
        \sym^{b_1}(X)\times \sym^{b_2}(X) & U^X_{b_1 b_2}\ar[l, open']\ar[r, open] & \sym^{b_1+b_2}(X)
    \end{tikzcd}
\end{equation}
commute, with both squares cartesian.

\begin{definition}
    In the above setup, given two factorizable systems $\{\Ec_{\alpha}\}_{\alpha\in I_\Yc}$ on $\sym^{b(\alpha)}(Y)$ and $\{\Fc_{\alpha}\}_{\alpha\in I}$ on $\sym^{b(\alpha)}(X)$, we call them \textbf{compatibly factorizable} with respect to $j$ if for any $\alpha\in I\setminus I_\Yc$, we have $M(j)^*\Fc_{\alpha}=0$, and for any $\alpha\in I_\Yc\subseteq I$, we have an isomporphism $\delta_\alpha:\Ec_{\alpha}\xrightarrow{\sim}M(j)^*\Fc_{\alpha}$, so that the factorization morphisms $\phi$ of Definition \ref{def:factorizable_system_general} are compatible with pullback, meaning the following diagram commutes
    \begin{equation}\label{eq:compa_fact_seq}
        \begin{tikzcd}
            \restr{\left(\Ec_{\alpha_1}\boxtimes \Ec_{\alpha_2}\right)}{U^Y_{b(\alpha_1)b(\alpha_2)}} \arrow[d, "\phi^\Ec_{\alpha_1\alpha_2}"]\arrow[r, "\delta_{\alpha_1}\boxtimes\delta_{\alpha_2}"]& \restr{M(j)^*\left(\Fc_{\alpha_1}\boxtimes \Fc_{\alpha_2}\right)}{U^Y_{b(\alpha_1)b(\alpha_2)}}\arrow[d, "U(j)^*\phi^\Fc_{\alpha_1\alpha_2}"]\\
            \restr{\Ec_{\alpha_1+\alpha_2}}{U^Y_{b(\alpha_1)b(\alpha_2)}} \arrow[r, "\delta_{\alpha_1+\alpha_2}"]& \restr{M(j)^*\left(\Fc_{\alpha_1+\alpha_2}\right)}{U^Y_{b(\alpha_1)b(\alpha_2)}}.
        \end{tikzcd}
    \end{equation}
\end{definition}

\begin{remark}\label{fact:rmk:comp_cond_automatic_hilb}
    Note, that the condition that for any $\alpha\in I\setminus I_\Yc$, we have $M(j)^*\Fc_{\alpha}=0$ is automatically satisfied if the factorization index set is the one from Lemma \ref{fact:lemma:hilb_fact_index_set} and the factorizable systems are pushed forward from $\hilb(\Yc)$ and $\hilb(\Xc)$ respectively, because for $\alpha\in I\setminus I_\Yc$, $\hilb^\alpha(\Yc)$ is the empty set.
\end{remark}

For compatible factorizable systems, we can compare the resulting systems from Theorem \ref{fact:thm:fact}.

\begin{theorem}\label{thm:fact:comp_fact}
    In the above setup, let $\{\Ec_{\alpha}\}_{\alpha\in I_\Yc}$ on $\sym^{b(\alpha)}(Y)$ and $\{\Fc_{\alpha}\}_{\alpha\in I}$ on $\sym^{b(\alpha)}(X)$ be two factorizable systems, which are compatibly factorizable with respect to $j$. Then the classes $\left[\Gc^\Ec_{\alpha}\right]$ and $\left[\Gc^\Fc_{\alpha}\right]$ from Theorem \ref{fact:thm:fact} satisfy $j^*\left[\Gc^\Fc_{\alpha}\right]=\left[\Gc^\Ec_{\alpha}\right]$ for $\alpha$ in $I_\Yc$.
\end{theorem}

\begin{proof}
    The proof of this theorem consists of checking functoriality of the construction of the $\left[\Gc_{\alpha}\right]$ under pullback. The first step in proving Theorem \ref{fact:thm:fact} was to pull back to $[X^{b(\alpha)}/S_{b(\alpha)}]$ to get factorizable systems $\overline{\Ec}_{\alpha}$ and $\overline{\Fc}_{\alpha}$ of $S_{b(\alpha)}$-equivariant (2-periodic) complexes of sheaves on $Y^{b(\alpha)}$ and $X^{b(\alpha)}$ respectively. This step preserves the isomorphisms $\delta$.

    The complexes $\overline{\Gc}^{\overline{\Ec}}_{\alpha}$ and $\overline{\Gc}^{\overline{\Fc}}_{\alpha}$ are constructed from the $\alpha$-level systems associated to the factorizable systems in Example \ref{fact:ex:fact_level_system}. By \eqref{eq:compa_fact_seq} the morphisms of the level systems are compatible with the isomorphisms $\delta$, which hence induce isomorphisms between the level systems on $Y^{b(\alpha)}$ for $\alpha$ in $I_\Yc$. Note here, that for $\alpha\notin I_\Yc$, we have $M(j)^*\Fc_{\alpha}=0$, which ensures that the terms of both systems are non-zero only for the same partitions. By Lemma \ref{fact:lemma:strict_generation_ktheory}, we may assume the level systems are strict. The proof of Lemma \ref{fact:lemma:strict_generation_ktheory} is a noetherian induction using the functor $D_A$ and morphism $I_A$ defined in Definition \ref{fact:def:D_I_strictness_proof}. Both of these preserve the isomorphisms between the level systems on $Y^{b(\alpha)}$ for $\alpha$ in $I_\Yc$. Hence, for $\alpha$ in $I_\Yc$, the resulting complexes $\overline{\Gc}^{\overline{\Ec}}_{\alpha}$ and $\overline{\Gc}^{\overline{\Fc}}_{\alpha}$ are isomorphic over $Y^{b(\alpha)}$.

    The final step of the proof of Theorem \ref{fact:thm:fact} is using d\'evissage \cite[Appl. 6.4.2]{weibel_kbook} to get the classes $\left[\Gc^\Ec_{\alpha}\right]$ and $\left[\Gc^\Fc_{\alpha}\right]$ on $Y$ and $X$ respectively. This respects the isomorphisms over $Y^{b(\alpha)}$, giving us $j^*\left[\Gc^\Fc_{\alpha}\right]=\left[\Gc^\Ec_{\alpha}\right]$ for $\alpha$ in $I_\Yc$.
\end{proof}
\section{Factorization for Orbifolds}\label{sect:factorization_application}

\subsection{Factorizability of Virtual Structure Sheaves}\label{sec:fact_vir_str_sheaf}

\subsubsection{Untwisted Virtual Structure Sheaves}
Now that we have developed the necessary factorization machinery for moduli spaces of zero-dimensional sheaves on orbifolds, we show that the (twisted) virtual structure sheaves we consider yield factorizable systems, and hence that Corollary \ref{fact:cor:comp_pexp_fact} lets us simplify our generating series of interest.

As tensor products of factorizable systems are again factorizable, we can deal separately with the virtual structure sheaves $\Oc^\vir$ and its twists.

\begin{proposition}\label{appl:prop:fact_vir_str_sheaf}
    Let $\Xc$ be a smooth three-dimensional orbifold. Consider the factorization index set $I\subset \Nr_0(\Xc)$ of Lemma \ref{fact:lemma:hilb_fact_index_set}. Then
    \begin{equation*}
        \left\{\Oc^\vir_{\hilb^\alpha(\Xc)}\right\}_{\alpha\in I}
    \end{equation*}
    is a factorizable system on $\hilb^\alpha(\Xc)$ in the sense of Definition \ref{def:factorizable_system_general} and Example \ref{fact:ex:hilb}.
\end{proposition}
\begin{proof}
    By adapting \cite[Prop. 5.7]{bf97_intrinsic_normal_cone} to virtual sheaves, we see that additivity of obstruction theories results in multiplicativity of virtual structure sheaves, in the sense that the virtual structure sheaf defined by $\Eb_{\alpha_1} \boxplus \Eb_{\alpha_2}$ on $\hilb^{\alpha_1}(\Xc)\times\hilb^{\alpha_2}(\Xc)$ comes with a natural isomorphism
    \begin{equation*}
        \Oc^\vir_{\Eb_{\alpha_1} \boxplus \Eb_{\alpha_2}} \cong \Oc^\vir_{\Eb_{\alpha_1}}\boxtimes \Oc^\vir_{\Eb_{\alpha_2}}.
    \end{equation*}
    The resulting isomorphisms naturally preserve associativity and commutativity in the sense of Definition \ref{def:factorizable_system_general}, so that it suffices to show the obstruction theories $\Eb_\alpha$ make a $\boxplus$-factorizable system. More specifically, we want isomorphisms
    \begin{equation}\label{appl:eq:obstruction_theories_oplus_fact}
        \restr{\left(\Eb_{\alpha_1} \boxplus \Eb_{\alpha_1}\right)}{U_{\alpha_1 \alpha_2}} \xrightarrow{\sim} \restr{\Eb_{\alpha_1+\alpha_2}}{U_{\alpha_1 \alpha_2}}
    \end{equation}
    which are associative and commutative as in Definition \ref{def:factorizable_system_general}.

    Take $\alpha_1+\alpha_2=\alpha$ in $I$ and consider the universal family $\Zc$ of $\hilb^{\alpha}(\Xc)$, pulled back to $U_{\alpha_1 \alpha_2}\times\Xc$. We denote this by $\widetilde{\Zc}$. Over $U_{\alpha_1 \alpha_2}\times\Xc$ this splits as $\widetilde{\Zc} = \widetilde{\Zc}_1 \sqcup \widetilde{\Zc}_2$, where $\widetilde{\Zc}_i$ are disjoint families of substacks of $\Xc$ of class $\alpha_i$. Consider the following commutative diagram.
    \begin{equation}\label{appl:eq:hilb_family_splitting_diagram}
        \begin{tikzcd}[column sep=small]
            & \widetilde{\Zc}=\widetilde{\Zc}_1 \sqcup \widetilde{\Zc}_2 \ar[d, closed]\ar[r, open] & \Zc\ar[d, closed] \\
            \hilb^{\alpha_1}(\Xc)\times\hilb^{\alpha_2}(\Xc)\times\Xc \ar[d] & U_{\alpha_1 \alpha_2}\times\Xc \ar[l, open']\ar[r, open] \ar[d] & \hilb^{\alpha}(\Xc)\times\Xc \ar[d] \\
            \hilb^{\alpha_1}(\Xc)\times\hilb^{\alpha_2}(\Xc) & U_{\alpha_1 \alpha_2} \ar[l, open']\ar[r, open] & \hilb^{\alpha}(\Xc).
        \end{tikzcd}
    \end{equation}
    The composition
    \begin{equation*}
        \begin{tikzcd}[column sep=small]
            \widetilde{\Zc}_i \ar[r, hookrightarrow] & U_{\alpha_1 \alpha_2}\times\Xc \ar[r,open] & \hilb^{\alpha_1}(\Xc)\times \hilb^{\alpha_1}(\Xc)\times\Xc \ar[r, "p_{i 3}"] & \hilb^{\alpha_i}(\Xc)\times \Xc
        \end{tikzcd}
    \end{equation*}
    is just the universal family $\Zc_i$ restricted to the image of $U_{\alpha_1 \alpha_2}$ in $\hilb^{\alpha_i}(\Xc)$. Using this, we obtain a natural isomorphism
    \begin{equation}\label{appl:eq:pullback_E_i}
        j^*p_i^* \Eb_{\alpha_i} \cong Rp_{U *} R\homc\left(\Ic_{\tilde{\Zc}_i},\Ic_{\tilde{\Zc}_i}\right)_0[2],
    \end{equation}
    where $p_U$ is the projection of $U_{\alpha_1 \alpha_2}\times\Xc$ onto the first factor. We also have a natural isomorphism
    \begin{equation}\label{appl:eq:pullback_E}
        j'^* \Eb_{\alpha_1+\alpha_2} \cong Rp_{U *} R\homc\left(\Ic_{\tilde{\Zc}},\Ic_{\tilde{\Zc}}\right)_0[2].
    \end{equation}
    Using disjointness of $\widetilde{\Zc}_1$ and $\widetilde{\Zc}_2$, we get a canonical natural isomorphism
    \begin{equation*}
        R\homc\left(\Ic_{\tilde{\Zc}},\Ic_{\tilde{\Zc}}\right)_0 \cong R\homc\left(\Ic_{\tilde{\Zc}_1},\Ic_{\tilde{\Zc}_1}\right)_0\oplus R\homc\left(\Ic_{\tilde{\Zc}_2},\Ic_{\tilde{\Zc}_2}\right)_0,
    \end{equation*}
    which yields the desired isomorphisms \eqref{appl:eq:obstruction_theories_oplus_fact} by combining with \eqref{appl:eq:pullback_E_i} and \eqref{appl:eq:pullback_E}. Note here, that this isomorphism couldn't exist without taking the traceless part. Using similar decompositions of the universal family, the required commutativity and associativity properties for these isomorphisms can be shown analogously. 
\end{proof}

\subsubsection{Twisted Virtual Structure Sheaves}

Now we see that the twisted virtual structure sheaf of Definition \ref{setup:def:twisted_vir_str_sh} make a factorizable system. In contrast to the untwisted virtual structure sheaf, we have to pass to the coarse space $X$ to prove factorizability of the twist.

\begin{proposition}\label{appl:lemma:fact_twist_vir}
    Let $\Xc$ be a smooth quasi-projective three-dimensional Calabi--Yau orbifold with coarse space $X$. Consider the factorization index set $I\subset \Nr_0(\Xc)$ of Lemma \ref{fact:lemma:hilb_fact_index_set}. Then 
    \begin{equation*}
        \left\{\widehat{\Oc}^\vir_{\hilb^\alpha(\Xc)}\right\}_{\alpha\in I}
    \end{equation*}
    is a factorizable system in the sense of Definition \ref{def:factorizable_system_general} and Example \ref{fact:ex:hilb}.
\end{proposition}

\begin{proof}
    Factorization properties are stable under tensor products, and the desired property for $\Oc^\vir_{\hilb^\alpha(\Xc)}$ has already been shown in the proof of Proposition \ref{appl:prop:fact_vir_str_sheaf}. 

    Take $\alpha_1+\alpha_2=\alpha$ in $I$ and consider the diagram \eqref{appl:eq:hilb_family_splitting_diagram} with cartesian squares. We want to compare the sheaves $j^*p_i^* p_{\Zc_{i} *}(\Oc_{\Zc_{i}})$ with $j^* p_{\Zc *}(\Oc_{\Zc})$, where $\Zc_{i}$ are the universal families of $\hilb^{\alpha_i}(\Xc)$ and $\Zc$ is the universal family of $\hilb^\alpha(\Xc)$. Repeatedly applying the push-pull formula yields natural isomorphisms
    \begin{align*}
        j^*p_i^* p_{\Zc_{i} *}\left(\Oc_{\Zc_{i}}\right) &\cong  p_{U_{\alpha_1 \alpha_2} *}\left(\Oc_{\tilde{\Zc}_{i}}\right),\\
        j^* p_{\Zc *}\left(\Oc_{\Zc}\right) &\cong p_{U_{b_1b_2} *}\left(\Oc_{\tilde{\Zc}}\right).
    \end{align*}
    As illustrated in the diagram \eqref{appl:eq:hilb_family_splitting_diagram}, $\widetilde{\Zc}$ splits into $\widetilde{\Zc}_{1} \sqcup \widetilde{\Zc}_{2}$. This gives us natural isomorphisms
    \begin{equation*}
        p_{U_{b_1b_2} *}\left(\Oc_{\tilde{\Zc}}\right) \cong p_{U_{\alpha_1 \alpha_2} *}\left(\Oc_{\tilde{\Zc}_{1}}\right)\oplus p_{U_{\alpha_1 \alpha_2} *}\left(\Oc_{\tilde{\Zc}_{2}}\right).
    \end{equation*}
    Because this natural isomorphism is essentially just the splitting of a sheaf into its restriction to connected components of its support, these isomorphisms satisfy the associativity and commutativity assumptions of Definition \ref{def:factorizable_system_general} with $\boxplus$. Multiplying by the constant factor $\kappa^{\frac{1}{2}}$ in the equivariant case is compatible with these isomorphisms. Taking determinants preserves these isomorphisms, but the commutativity diagram \eqref{fact:eq:gen_comm_diagram} only commutes up to the sign $(-1)^{\rk \left(p_{\Zc_{1} *}(\Oc_{\Zc_{1}})\right)\rk \left(p_{\Zc_{2} *}(\Oc_{\Zc_{2}})\right)}$, which by Lemma \ref{setup:lemma:twist_loc_free_rank} is exactly $(-1)^{b(\alpha_1)b(\alpha_2)}$. The degree shift of $b(\alpha_i)$ of our 2-periodic complexes in Definition \ref{setup:def:twisted_vir_str_sh} introduces an additional sign $(-1)^{b(\alpha_1)b(\alpha_2)}$ in the commutativity of the diagram \eqref{fact:eq:gen_comm_diagram}, which exactly cancels the previous sign, and hence makes the twist factorizable.
\end{proof}

\begin{remark}
    The strategy in the above proof to first show $\boxplus$-factorizability and then taking determinants works more generally and is implicitly already used in \cite[5.3.10]{ok15}. The shifts must be introduced, because taking determinants makes the commutativity diagram only commute up to sign, which can be fixed using these shifts.
\end{remark}

\begin{remark}\label{appl:rmk:sign}
    The shift, which is required to make the diagrams commute, coincides with signs in the literature. For example, in \cite[Cor. 6.2]{Cao_2023}, the same sign $(-1)^{n_0}$ is introduced as a sign of the formal variable $q_0$ instead of being part of the twisted virtual structure sheaf as in our Definition \ref{setup:def:twisted_vir_str_sh}. The above lemma shows that this sign is exactly the sign making the system of twisted virtual structure sheaves factorizable.
\end{remark}

\begin{lemma}
    Let $\Xc$ be a smooth three-dimensional orbifold, and let $U$ be its non-stacky locus. Consider the factorization index set $I\subset \Nr_0(\Xc)$ of Lemma \ref{fact:lemma:hilb_fact_index_set}. By Proposition \ref{appl:lemma:fact_twist_vir}, Lemma \ref{fact:lemma:comp_setup_cd} and Example \ref{fact:ex:hilb_to_sym}, we have the factorizable system
    \begin{equation*}
        \left\{\Fc_\alpha \coloneqq \xi_{\alpha *}\HC_{\alpha *}\widehat{\Oc}^\vir_{\hilb^\alpha(\Xc)}\right\}_{\alpha\in I}.
    \end{equation*}
    It satisfies that for any $\alpha\in I\setminus \Delta_I$, $\Fc_{\alpha}$ vanishes on $\sym^{b(\alpha)}(U)$.
\end{lemma}
\begin{proof}
    Set $\eta_\alpha = \xi_\alpha\circ \HC_\alpha$. We have the cartesian diagram
    \begin{equation*}
        \begin{tikzcd}
            \hilb^\alpha(U) \ar[r,open]\ar[d,"\eta_\alpha^U"] & \hilb^\alpha(\Xc) \ar[d,"\eta_\alpha"] \\
            \sym^{b(\alpha)}(U) \ar[r,open] & \sym^{b(\alpha)}(X).
        \end{tikzcd}
    \end{equation*}
    This gives us
    \begin{equation*}
        \restr{\Fc_\alpha}{\sym^{b(\alpha)}(U)} \cong \left(\eta_\alpha^U\right)_*\left(\restr{\widehat{\Oc}^\vir_{\hilb^\alpha(\Xc)}}{\hilb^\alpha(U)}\right),
    \end{equation*}
    but any substack of $U$ has class in $\Delta_I$ by definition. So, for a class $\alpha\in I\setminus \Delta_I$, $\hilb^\alpha(U)$ is emtpy, proving $\restr{\Fc_\alpha}{\sym^{b(\alpha)}(U)}$ vanishes. 
\end{proof}

\subsection{Orbifold Generating Series}

\subsubsection{Orbifold Generating Series}
For a given smooth quasi-projective Calabi--Yau orbifold $\Xc$ of dimension three, we apply the factorization theorem to simplify its generating series of (equivariant) Donaldson--Thomas invariants. Let $I$ be the factorization index set of Lemma \ref{fact:lemma:hilb_fact_index_set}. As $\hilb^{\alpha}(\Xc)$ is empty for $\alpha\notin I$ by definition, it suffices to sum over $\alpha\in I$ in the generating series $\Zs(\Xc)$. As seen above, the pushforwards of the twisted virtual structure sheaves $\widehat{\Oc}^{vir}$ on $\hilb^\alpha(\Xc)$ to $\sym^{b(\alpha)}(X)$ yield a factorizable system of sheaves. 

We assume from now on that $I$ is closed under addition. Note that the factorization index set of Lemma \ref{fact:lemma:hilb_fact_index_set} is closed under addition for orbifolds of the form $\Xc=\Xc'\times \C$, which we are mainly interested in. By Theorem \ref{fact:thm:fact}, there are classes
\begin{equation*}
    [\Gc_\alpha]\in K(X)\text{ for $\alpha\in \Delta_I$,}\quad [\Gc_\alpha]\in K(S)\text{ for $\alpha\notin \Delta_I$,}
\end{equation*}
such that the generating series equals
\begin{align*}
    \Zs(\Xc) &= 1+\sum_{\alpha\in I} q^{\alpha} \chi\left(\hilb^{\alpha}(\Xc),\widehat{\Oc}^{\vir}_{\hilb^{\alpha}(\Xc)}\right)\\
    &= \pexp{\sum_{\alpha\in \Delta_I}q^{\alpha} \chi\left(X,\Gc_{\alpha}\right) + \sum_{\alpha\in I\setminus\Delta_I}q^{\alpha} \chi\left(S,\Gc_{\alpha}\right)}.
\end{align*}

\subsubsection{Crepant Resolutions and Compatible Factorization}
Recall that we are working with a fixed coarse moduli space $\pi:\Xc\to X$. Let $\nu:Y\to X$ be a crepant resolution. Let $U$ be the non-stacky locus and let $S$ be its complement in the coarse moduli space. Note that $U$ also openly embeds into $Y$. Let $E$ be its complement in $Y$. Take a fixed point $p$ in $U$ to define $\Delta_I=\N [\Oc_p]$. We can compute the equivariant K-theoretic generating function for $\hilb^n(Y)$ using Nekrasov's formula \cite[Theorem 3.3.6]{ok15}. We have the following identification
\begin{equation*}
    \hilb^n(Y) \hookleftarrow \hilb^n(U) \xrightarrow{\sim} \hilb^{n[\Oc_p]}(U) \hookrightarrow \hilb^{n[\Oc_p]}(\Xc).
\end{equation*} 
The obstruction theories agree under this isomorphism, since the deformation theory of sheaves restricts naturally to open substacks. This gives us the required isomorphisms for compatible factorization, so the twisted virtual structure sheaves are compatibly factorizable along the embeddings $U\hookrightarrow Y$ and $U\hookrightarrow\Xc$ respectively. Pushforward along the open embeddings embeds the factorization index semigroup $I_U=\N$ into $I_Y=\N$ by the identity and into $I_\Xc$ by $n\mapsto n[\Oc_p]$.

Again, we assume that $I_\Xc$ is closed under addition. Applying Theorem \ref{thm:fact:comp_fact}, we get
\begin{equation*}
    \left[\restr{\Gc_{n[\Oc_p]}^\Xc}{U}\right] = \left[\restr{\Gc_{n}^Y}{U}\right].
\end{equation*}
Hence, the class $\left[\Gc_{n[\Oc_p]}^\Xc\right] - \left[\nu_*\Gc_{n}^Y\right]$ is supported on $S$. By d\'evissage \cite[Appl. 6.4.2]{weibel_kbook}, it must be the pushforward of some class $\left[\Gc_{n[\Oc_p]}^S\right]\in K(S)$. By invariance of $\chi$ under pushforwards, we obtain
\begin{equation*}
    \chi\left(X,\Gc_{n[\Oc_p]}^\Xc\right) = \chi\left(Y,\Gc_{n}^Y\right) + \chi\left(S, \Gc_{n[\Oc_p]}^S\right).
\end{equation*}

Putting all these parts together, we obtain the following form.

\begin{proposition}\label{pp:prop:general_open_locus_form}
    Let $\pi:\Xc\to X$ be the coarse moduli space, and let $\nu:Y\to X$ be a crepant resolution. Let $U$ be the non-stacky locus and let $S$ be its complement in the coarse moduli space. Assume that the factorization index set $I_\Xc$ of Lemma \ref{fact:lemma:hilb_fact_index_set} is closed under addition. We have the equation
    \begin{eqnarray*}
        Z(\Xc) =& \pexp{\sum_{n>0}q^{n[\Oc_p]} \left(\chi\left(Y,\Gc_{n}^Y\right)+\chi\left(S,\Gc^S_{n[\Oc_p]}\right)\right) + \sum_{\alpha\in I\setminus \Delta_I} q^{\alpha} \chi\left(S,\Gc_{\alpha}^\Xc\right)}\\
        =& Z(Y,q^{[\Oc_p]}) \cdot \pexp{\sum_{n>0}q^{n[\Oc_p]} \chi\left(S,\Gc^S_{n[\Oc_p]}\right) + \sum_{\alpha\in I\setminus \Delta_I} q^{\alpha} \chi\left(S,\Gc_{\alpha}^\Xc\right)}
    \end{eqnarray*}
    where $Z(Y,q)$ is the equivariant K-theoretic generating function for the Hilbert schemes of points on $Y$. 
\end{proposition}
\section{Rigidity}\label{sect:rigidity}

\subsection{Orbifold Computation}
Now we specialize to our specific situation with $\mu_r$ acting on $\C^3$ with weight $(1,r-1,0)$. Let $\Ts=(\C^*)^3$ act by $(t_1,t_2,t_3)$ on $\C^3$. In this section, we will prove the following theorem.
\begin{theorem}\label{rig:thm:main_thm}
    The $\Ts$-equivariant K-theoretic degree $0$ DT generating series for $[\C^3/\mu_r]$, with $\mu_r$ acting on $\C^3$ with weight $(1,r-1,0)$, is
    \begin{equation*}
        \Zs\left([\C^3/\mu_r],q_0,\dots,q_{r-1}\right) = \pexp{\Fs_r(q)+\Fs_r^\col(q_0,\dots,q_{r-1})},
    \end{equation*}
    where
    \begin{align*}
        \Fs(q) &\coloneqq \frac{[t_2t_3][t_1t_3][t_1t_2]}{[t_1][t_2][t_3]}\frac{1}{[\kappa^{1/2}q][\kappa^{1/2}q^{-1}]},\\
        \Fs_r(q) &\coloneqq \sum_{k=0}^{r-1} \Fs\left(t_1^{r-k}t_2^{-k},t_1^{-r+k+1}t_2^{k+1},t_3,q\right),\\
        \Fs_r^\col(q) &\coloneqq \frac{[t_1t_2]}{[t_3]}\frac{1}{[\kappa^{1/2}q][\kappa^{1/2}q^{-1}]}\left(\sum_{0< i \leq j < r} \left(q_{[i,j]}+q_{[i,j]}^{-1}\right)\right),
    \end{align*}
    where $q=q_0\cdots q_{r-1}$, $q_{[i,j]}=q_i\cdots q_j$, $[w]=w^{1/2}-w^{-1/2}$, and $\pexp{\cdot}$ denotes the plethystic exponential. Note that the factorization index set $I$ of Lemma \ref{fact:lemma:hilb_fact_index_set} is closed under addition in this case.
\end{theorem}

The proof occupies the remainder of the section. As $\mu_r$ acts trivially on the crepant resolution of $[\C^3/\mu_r]$ is given by $Y=\Ac_{r-1}\times\C$, where $\Ac_{r-1}$ is the minimal resolution of $[\C^2/\mu_r]$. A toric description of $\Ac_{r-1}$ can be found in \cite[\S 10.1]{cls_toric}. From this description we obtain $r$ charts of $Y$ isomorphic to $\C^3$ with weights $\left(t_1^{r-k}t_2^{-k},t_1^{-r+k+1}t_2^{k+1},t_3\right)$, where $0\leq k\leq r-1$. Localization together with Nekrasov's formula \cite[Theorem 3.3.6]{ok15} then lets us compute
\begin{equation*}
    \Zs(Y,q) = \pexp{\Fs_r(q)}.
\end{equation*}

By computing the Euler characteristics in Proposition \ref{pp:prop:general_open_locus_form} via localization on $S$, and by inserting the above formula for $\Zs(Y,q)$, we obtain the following.

\begin{proposition}\label{pp:prop:loc_open_locus_form}
    For $\mu_r$ acting on $\C^3$ with weight $(1,1-1,0)$, we have
    \begin{equation}\label{pp:eq:loc_open_locus_form}
        \Zs\left([\C^3/\mu_r],\dv{q}\right) = \pexp{\Fs_r(q)}\cdot \pexp{\frac{[t_1t_2]}{[t_3]}\sum_{\dv{n}\in I} \dv{q}^\dv{n} h_{\dv{n}}},
    \end{equation}
    with $h_{\dv{n}}$ in $\Z[t_1^{\pm 1/2},t_2^{\pm 1/2},t_3^{\pm 1/2}]$.
\end{proposition}
\begin{proof}
    We insert the $\Zs(Y,q)$ of the crepant resolution computed above into Proposition \ref{pp:prop:general_open_locus_form}. Then we note that $S$ is the coarse moduli space of $[\pt/\mu_r]\times\C$, which is just $\C$, with $\Ts$ acting just by $t_3$. So, localization on $S$ gives us
    \begin{equation*}
        \Zs\left([\C^3/\mu_r],\dv{q}\right) = \pexp{\Fs_r(q)}\cdot \pexp{\frac{1}{1-t_3^{-1}}\sum_{\dv{n}\in I} \dv{q}^\dv{n} \tilde{h}_{\dv{n}}},
    \end{equation*}
    for some $\tilde{h}_{\dv{n}}$ in $\Z[t_1^{\pm 1/2},t_2^{\pm 1/2},t_3^{\pm 1/2}]$.

    To show the proposition, we claim that all $\tilde{h}_{(n_0,n_1)}$ are divisible by $t_1t_2-1$. By invertibility of the plethystic exponential, it suffices to prove that any localization weight in $\Zs\left([\C^3/\mu_r],q_0,q_1\right)$ is divisible by $t_1t_2-1$. By the localization weights $\hat{a}(\pi)$, which we computed in Section \ref{sec:colored_vertex}, we need to check that $t_1t_2$ has negative coefficient in the class of the virtual tangent space at $\pi$. By the natural identification \eqref{setup:eq:obstruction_theory_fixed_part}, the virtual tangent space at $\pi$ is exactly the $\mu_r$-fixed part of the virtual tangent space of $\hilb(\C^3)$ at the fixed point corresponding to the underlying plane partition of $\pi$. Now \cite[Lemma 5]{maulik2005gromovwitten} shows that the virtual tangent space of $\hilb(\C^3)$ at each fixed point has negative coefficient of $t_1t_2$. Taking the $\mu_r$-invariant part retains this negative coefficient.
\end{proof}

Applying Okounkov's rigidity principle, we can refine this proposition as follows.

\begin{lemma}\label{pp:lemma:rigidity}
    The functions $h_{\dv{n}}$ in Proposition \ref{pp:prop:loc_open_locus_form} are in $\Z[\kappa^{\pm 1}]$, that is, they are only dependent on $\kappa$, not $t_1,t_2,t_3$.
\end{lemma}
\begin{proof}
    We follow the proof of \cite[Prop. 3.5.11]{ok15}. For any $\Ts$-weight $w$, the fraction
    \begin{equation*}
        \frac{[\kappa/w]}{[w]}
    \end{equation*}
    remains bounded for any limit $t_i^{\pm 1}\to \infty$ that keeps $\kappa$ fixed. That means all terms of the first factor in \eqref{pp:eq:loc_open_locus_form} as well as the local contributions $\hat{a}(\pi)$ to $\Zs([\C^3/\mu_r])$ from Section \ref{sec:colored_vertex} and $\frac{[t_1t_2]}{[t_3]}$ in the second factor of \eqref{pp:eq:loc_open_locus_form} remain bounded under any such limit. 

    If we can prove that all $h_{\dv{n}}$ from Proposition \ref{pp:prop:loc_open_locus_form} remain bounded under any limit $t_i^{\pm 1}\to \infty$ that keeps $\kappa$ fixed, then as Laurent polynomials, they only depend on $\kappa$ and not on the $t_i$. To show this, we work by induction in $I$. For any $I$-indecomposable element $\dv{n}$ in $I$, the coefficient of $\dv{q}^\dv{n}$ in the right-hand side of \eqref{pp:eq:loc_open_locus_form} is just
    \begin{equation*}
        \frac{[t_1t_2]}{[t_3]}h_\dv{n},
    \end{equation*}
    plus a term coming from the first factor of the right-hand side of \eqref{pp:eq:loc_open_locus_form} if $\dv{n}\in\Delta_I$. As the other terms of the equation resulting from taking $\dv{q}^\dv{n}$-coefficients in \eqref{pp:eq:loc_open_locus_form}, as well as $\frac{[t_1t_2]}{[t_3]}$, are bounded under limits $t_i^{\pm 1}\to \infty$ keeping $\kappa$ fixed, so is $h_\dv{n}$. In particular, we have shown this property for $\dv{n}$ with $\abs{\dv{n}}=1$.

    Assume the desired limit property is shown for all $I$-indecomposable $\dv{n}$ and all $\dv{n}$ with $\abs{\dv{n}}<N$. If $\dv{n}$ in $I$ is not $I$-indecomposable with $\abs{\dv{n}}=N$, the $\dv{q}^\dv{n}$-coefficient in the right-hand side of \eqref{pp:eq:loc_open_locus_form} consists of one term
    \begin{equation*}
        \frac{[t_1t_2]}{[t_3]}h_\dv{n},
    \end{equation*}
    plus other terms, which are products of $\frac{[t_1t_2]}{[t_3]}$, $h_\dv{n}$ with $\abs{\dv{n}}<N$, and terms from the first factor of \eqref{pp:eq:loc_open_locus_form}. All of these other terms, $\dv{q}^\dv{n}$-coefficient in the left-hand side of \eqref{pp:eq:loc_open_locus_form}, as well as $\frac{[t_1t_2]}{[t_3]}$ remain bounded under limits $t_i^{\pm 1}\to \infty$ keeping $\kappa$ fixed. So, the same holds for $h_\dv{n}$.
\end{proof}

\subsection{Equivariant Limit Computation}\label{rig:sec:eq_lim_comp}
We want to compute the equivariant K-theoretic Donaldson--Thomas partition function from the result of Proposition \ref{pp:prop:loc_open_locus_form}. To determine the functions $h_{\dv{n}}$ in Proposition \ref{pp:prop:loc_open_locus_form}, using Lemma \ref{pp:lemma:rigidity}, it suffices to compute a limit in the $\Ts$-parameters which fixes $\kappa$. We modify the computation of colored plane partition counts in \cite{young10} to compute this limit.

Following the proof of Nekrasov's formula in \cite[3.5.12]{ok15}, we compute the limit in the equivariant parameters $t_1,t_2,t_3$, which sends
\begin{equation}\label{rig:eq:limit}
    t_1,t_3\to 0,\ \lvert t_1 \rvert\ll\lvert t_3\rvert,\ \kappa\ \mathrm{fixed}.
\end{equation}
We denote the limit of the generating series under this limit of equivariant parameters by
\begin{equation*}
    \overrightarrow{\Zs}\left([\C^3/\mu_r],\dv{q}\right).
\end{equation*}
In order to compute this limit using the argument in \cite{young10}, we need the following result from \cite[3.5.17]{ok15}.

\begin{proposition}\label{ppc:partition_limit_contribution}
    The contribution of the colored partition $\pi$ to the $\Ts$-equivariant K-theoretic count in the above specified limit of its parameters is $\left(-\kappa^{1/2}\right)^{\ind(\pi)}$, where 
    \begin{equation*}
        \ind(\pi) = \sum_{\Box=(i_1,i_2,i_3)\in\pi^G}\mathrm{sgn}(i_2-i_1),
    \end{equation*}
    where we recall that $\pi^G$ denotes the $0$-colored boxes of $\pi$.
\end{proposition}

\begin{proof}
    As computed in Section \ref{sec:colored_vertex}, the contribution of the colored partition $\pi$ to the $\Ts$-equivariant K-theoretic count is
    \begin{equation*}
        \hat{a}(\pi) = \prod_{w\in W} \frac{[\kappa/w]}{[w]},
    \end{equation*}
    where $W$ is a subset of the $\Ts$-weights of the virtual tangent space at $\pi$, such that $\Tb_{\Xc,\pi}^\vir = \sum_{w\in W} \left(w-\frac{\kappa}{w}\right)$. The limits of $\frac{[\kappa/w]}{[w]}$ can be computed
    \begin{equation*}
        \frac{[\kappa/w]}{[w]} \to \begin{cases}
            -\kappa^{-1/2}, & w\to\infty,\\
            -\kappa^{1/2}, & w\to 0.
        \end{cases}
    \end{equation*}
    Hence, the limit of $\hat{a}(\pi)$ becomes
    \begin{eqnarray*}
        \hat{a}(\pi) \to (-\kappa^{1/2})^{\ind(\pi)},
    \end{eqnarray*}
    where we set
    \begin{equation*}
        \ind(\pi) \coloneqq \abs{\left\{w\in W\ |\ w\to 0\right\}}-\abs{\left\{w\in W\ |\ w\to \infty\right\}}.
    \end{equation*}
    Now we want to find the above expression of this index. Note first, that, since $\kappa$ is fixed by the limit, the weights $w$ and $\frac{\kappa}{w}$ of the virtual tangent space have opposite limits. To compute this limit, we want to use the index computation in \cite[Appendix A]{nekrasov2014membranes}. For a $\Ts$-representation $V$, given as a linear combination of weights $V=\sum_w c_w w$ with $c_w\in \C$, their index is defined as
    \begin{eqnarray*}
        \ind(V) &\coloneqq \sum_w c_w \ind(w),\\
        \ind(w) &\coloneqq \begin{cases}
            1, & w\to 0,\\
            0, & w\to \infty.
        \end{cases}
    \end{eqnarray*}
    Inserting $\Tb_{\Xc,\pi}^\vir = \sum_{w\in W} \left(w-\frac{\kappa}{w}\right)$, we have
    \begin{equation*}
        \ind(\pi) = \ind\left(\Tb_{\Xc,\pi}^\vir\right).
    \end{equation*}
    By \eqref{setup:eq:obstruction_theory_fixed_part}, $\Tb_{\Xc,\pi}^\vir $ is the $\mu_r$-equivariant part of $\Tb_{\C^3,\pi}^\vir$ and \cite[Appendix A]{nekrasov2014membranes} show $\ind\left(\Tb_{\C^3,\pi}^\vir\right)=\ind\left(t_3^kV-t_3^{-k}V^\vee\right)$ for $k$ sufficiently large, but such that still $\abs{t_1}\ll \abs{t_3}^k$. Here, $V$ is the character of $\Oc_{Z_\pi}$. 

    We want to show a similar identity for the $\mu_r$-invariant parts $\ind\left(\Tb_{\Xc,\pi}^\vir\right)=\ind\left(t_3^kV^{\mu_r}-t_3^{-k}(V^\vee)^{\mu_r}\right)$. Note that the claim of the proposition follows immediately from such an identity, because
    \begin{equation*}
        V^{\mu_r} = \sum_{\Box=(i_1,i_2,i_3)\in\pi^G} t_1^{-i_1}t_2^{-i_2}t_3^{-i_3},
    \end{equation*}
    and
    \begin{equation*}
        t_3^k t_1^{-i_1}t_2^{-i_2}t_3^{-i_3} \to \begin{cases}
            0, & , i_2\geq i_1\\
            \infty, &, i_2 < i_1,
        \end{cases}
    \end{equation*}
    which gives us 
    \begin{equation*}
        \ind\left(\Tb_{\Xc,\pi}^\vir\right)=\ind(t_3^kV^{\mu_r}-t_3^{-k}(V^\vee)^{\mu_r}) = \sum_{\Box=(i_1,i_2,i_3)\in\pi^G}\mathrm{sgn}(i_2-i_1).
    \end{equation*}
    
    To prove $\ind(\Tb_{\Xc,\pi}^\vir)=\ind(t_3^kV^{\mu_r}-t_3^{-k}(V^\vee)^{\mu_r})$, we cannot directly use the result of \cite[Appendix A]{nekrasov2014membranes}, but we will follow their proof. Note first, that since, the $\kappa$ is fixed in the limit, the index is independent of the value of $\kappa$, so may work with the subtorus $\As=\{\kappa=1\}\subset \Ts$. In \cite[A.1.3]{nekrasov2014membranes}, they identify the $\As$-characters
    \begin{equation*}
        \Tb_{\C^3,\pi}^\vir = (1-t_3)\Nc_2(Z_\pi) + t_3^kV - t_3^{-k}V^\vee,
    \end{equation*}
    where $\Nc_2(Z_\pi)$ is the character of the tangent space of a moduli space of framed torsion-free sheaves on $\C^2$. There is a symplectic form on this tangent space, which matches weights $t_1^a t_2^b t_3^c$ with $t_1^{-a} t_2^{-b} t_3^{-c-1}$, and hence attracting and repelling weights of its character, except for possibly the matched weights $1$ and $t_3^{-1}$. Then \cite[Lemma A.1]{nekrasov2014membranes} shows that $\Nc_2(Z_\pi)^A=0$, and hence neither weight $1$ or $t_3$ can come up in the character. The symplectic form is $\mu_r$-invariant, so it matches $\mu_r$-invariant weights with $\mu_r$-invariant weights. Here, we used that $\mu_r$ only acts on the first two coordinates. Using $\Nc_2(Z_\pi)^A=0$ and the matchings of weights by the symplectic form $\omega$, we can match weights of $\Nc_2(Z_\pi)$ and $t_3\Nc_2(Z_\pi)$ as follows
    \begin{equation*}
        \begin{tikzcd}
            \Nc_2(Z_\pi): & t_1^a t_2^b t_3^c \ar[r, leftrightarrow, "\omega"]\ar[dr, leftrightarrow, "\cdot^{-1}"]\ar[d,"\cdot t_3"] & t_1^{-a}t_2^{-b}t_3^{-c-1} \ar[d, "\cdot t_3"]\ar[dl, leftrightarrow, "\cdot^{-1}"]\\
            t_3\Nc_2(Z_\pi): & t_1^a t_2^b t_3^{c+1} & t_1^{-a}t_2^{-b}t_3^{-c}.
        \end{tikzcd}
    \end{equation*}
    Since $\omega$ matches attracting and repelling weights in $\Nc_2(Z_\pi)$, attracting and repelling weights in $t_3\Nc_2(Z_\pi)$ are also matched. This matches $\mu_r$-invariant weights with $\mu_r$-invariant ones, so $\ind\left(((1-t_3)\Nc_2(Z_\pi))^{\mu_r}\right)=0$.
\end{proof}

The above result allows us to adjust the weight operator in the argument of \cite{young10} and compute the limit of the generating series to finish the proof of Theorem \ref{rig:thm:main_thm}. We recall now the setup of \cite{young10}, which is based on the setup in \cite[Appendix A]{ok01}. We consider a plane partition as a collection of slices, which are partitions $\{(a,b)\in \Z_{\geq 0}^2\ |\ (a+k,a,b)\}$ consisting of the boxes with a fixed $i_1-i_2=k$. With our specific $\mu_r$-action a colored plane partition is then sliced into monochrome slices. This is pictured in Figure \ref{fig:slicing} below.

\begin{figure}[!ht]
    \centering
    \includegraphics[width=\linewidth]{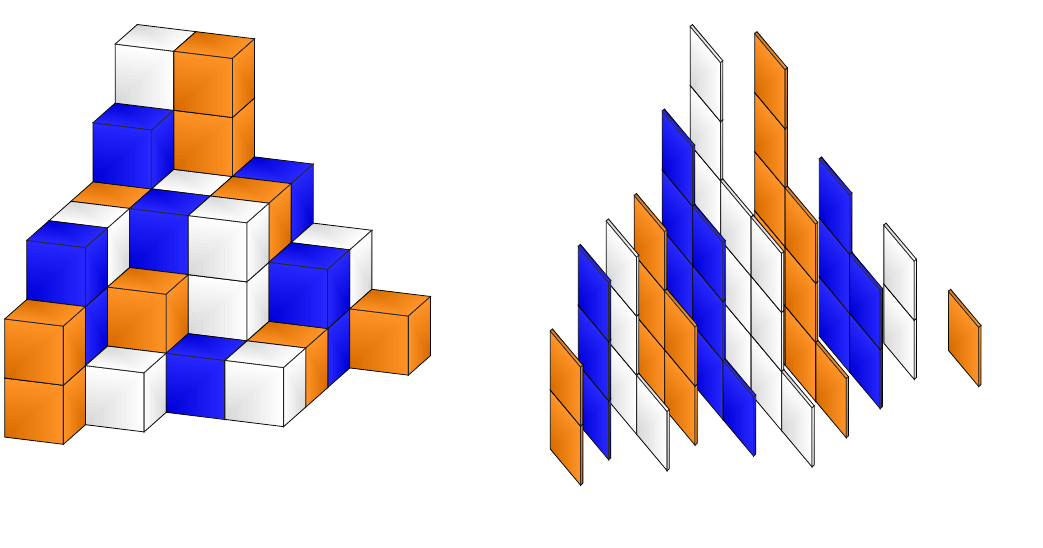}
    \caption{Slicing of a $\mu_3$-colored plane partition.}
    \label{fig:slicing}
\end{figure}

This slicing can be used for counting by starting with empty partitions on the right and left side and iteratively applying operators, which give weighted sums of possible next slices moving inwards. We explain now how to formally set this up. We work in to charge-zero subspace of the infinite wedge space $\big(\extp^{\infty/2}\big)_0V$, where $V$ is a vector space with a basis labelled by the elements of $\Z+\frac{1}{2}$. The vector space $\big(\extp^{\infty/2}\big)_0V$ has an orthonormal basis given by partitions $\ket{\lambda}$.

There are operators
\begin{equation*}
    \Gamma_\pm(x) = \mathrm{exp}\left(\sum_l \frac{x^l}{l}\alpha_{\pm l}\right)
\end{equation*}
acting on $\big(\extp^{\infty/2}\big)_0V$, which act on a partition $\ket{\lambda}$ by giving a weighted sum of possible partitions in the next smaller/bigger slice of a plane partition if the current one is $\lambda$.

To compute the partition function, we consider weight operators
\begin{equation*}
    Q_i\ket{\lambda} = q_i^{\lvert\lambda\rvert}\ket{\lambda},
\end{equation*}
whose role it will be to count the contribution to the partition function of each slice. We now modify the argument as follows. In the limit of the equivariant parameters specified in \eqref{rig:eq:limit}, by Proposition \ref{ppc:partition_limit_contribution}, each slice $\lambda$ has a fixed contribution $1$ if the slice doesn't have color $0$, as only the boxes in $\pi^G$ contribute, and the slices of color $0$ have a fixed contribution
\begin{eqnarray*}
    (-1)^{\lvert\lambda\rvert}(-\kappa^{1/2})^{\lvert\lambda\rvert}, &\mathrm{if }i_2\geq i_1,\\
    (-1)^{\lvert\lambda\rvert}(-\kappa^{1/2})^{-\lvert\lambda\rvert}, &\mathrm{if }i_2< i_1.
\end{eqnarray*}
Here multiplied by the additional sign $(-1)^{\lvert\lambda\rvert}$. This comes from the sign $(-1)^{n_0}$ in the contribution of each partition, which comes from our definition of the twisted virtual structure sheaf and is usually manually introduced, as discussed in Remark \ref{appl:rmk:sign}. Since we consider slices $\lambda$ with fixed $i_1-i_2$ this contribution is well-defined for each slice. We use this to define an additional equivariant limit weight operator
\begin{equation*}
    K_\pm\ket{\lambda} = (\kappa^{1/2})^{\pm\lvert\lambda\rvert}\ket{\lambda}.
\end{equation*}

We can now use the operators to define
\begin{equation*}
    \bar{A}_\pm(x) = \Gamma_\pm(x) Q_{r-1} \cdots \Gamma_\pm(x) Q_1 \Gamma_\pm(x)K_\pm Q_0,
\end{equation*}
so that
\begin{equation}\label{rig:eq:limit_operator_comp_setup}
    \overrightarrow{\Zs}([\C^3/\mu_r],\dv{q}) = \bra{\phi}\cdots \bar{A}_+(1)\bar{A}_+(1)\bar{A}_+(1)\bar{A}_-(1)\bar{A}_-(1)\bar{A}_-(1)\cdots\ket{\phi}
\end{equation}
computes the desired limit of the partition function, where $\phi$ denotes the empty partition. To compute this partition function, we want to compute the commutators of the operators involved and then reorder them. Note that we wrote an infinite product of operators, which should be interpreted as follows. For any given finite order, we can make this a finite product of operators, such that \eqref{rig:eq:limit_operator_comp_setup} computes the limit generating series up to that given finite order.

Now we compute commutators of the various operators, so that we can reorder them in \eqref{rig:eq:limit_operator_comp_setup} to compute the generating series. The commutator
\begin{equation}\label{rig:eq:gamma_commutator}
    \left[\Gamma_+(a),\Gamma_-(b)\right] = \frac{1}{1-ab}
\end{equation}
was already computed in \cite[Lemma 3.3]{young10}. Moreover, Young computes
\begin{equation*}
    \Gamma_+(x)Q_i=Q_i\Gamma_+(xq_i),\quad Q_i\Gamma_-(x)=\Gamma_-(xq_i)Q_i.
\end{equation*}
Using the same formula \cite[(4)]{young10}, we can similarly compute
\begin{equation*}
    \Gamma_+(x)K_+=K_+\Gamma_+\left(x\kappa^{1/2}\right),\quad K_-\Gamma_-(x)=\Gamma_-\left(x\kappa^{-1/2}\right)K_-.
\end{equation*}
Finally, by definition, the operators $Q_i$ and $K_\pm$ commute with each other. Writing $q=q_0\cdots q_{r-1}$ and $Q=Q_0\cdots Q_{r-1}$, we can then rewrite $\bar{A}_\pm(x)$ as follows.
\begin{eqnarray*}
    \bar{A}_+(x) =& QK_+\Gamma_+\left(xq\kappa^{1/2}\right)\Gamma_+\left(xq_{[0,r-2]}\kappa^{1/2}\right)\cdots\Gamma_+\left(xq_0\kappa^{1/2}\right),\\
    \bar{A}_-(x) =& \Gamma_-(x)\Gamma_-\left(xq_{r-1}\right)\cdots \Gamma_-\left(xq_{[2,r-1]}\right)\Gamma_-\left(xq_{[1,r-1]}\right)K_-Q.
\end{eqnarray*}
We set
\begin{eqnarray*}
    A_+(x) :=& \Gamma_+\left(xq\kappa^{1/2}\right)\Gamma_+\left(xq_{[0,r-2]}\kappa^{1/2}\right)\cdots\Gamma_+\left(xq_0\kappa^{1/2}\right),\\ 
    A_-(x) :=& \Gamma_-(x)\Gamma_-\left(xq_{r-1}\right)\cdots \Gamma_-\left(xq_{[2,r-1]}\right)\Gamma_-\left(xq_{[1,r-1]}\right).
\end{eqnarray*}
Now iteratively using the commutators of $\Gamma_\pm$  from \eqref{rig:eq:gamma_commutator} we can compute
\begin{eqnarray}
    A_+(x)A_-(y)=&\Gamma_+\left(xq_{[0,r-1]}\kappa^{1/2}\right)\Gamma_+\left(xq_{[0,r-2]}\kappa^{1/2}\right)\cdots\Gamma_+\left(xq_{[0,0]}\kappa^{1/2}\right)\cdot\nonumber\\ 
    &\Gamma_-\left(y\frac{q}{q_{[0,r-1]}}\right)\Gamma_-\left(y\frac{q}{q_{[0,r-2]}}\right)\cdots \Gamma_-\left(y\frac{q}{q_{[0,1]}}\right)\Gamma_-\left(y\frac{q}{q_{[0,0]}}\right)\nonumber\\
    =& A_-(y)A_+(x) \prod_{i,j=0}^{r-1}\left(1-xyq\frac{q_{[0,i]}}{q_{[0,j]}}\kappa^{1/2}\right)^{-1}.\label{rig:eq:comm_A}
\end{eqnarray}
To simplify notation, we write
\begin{equation*}
    C_r(x,y) := \prod_{i,j=0}^{r-1}\left(1-xyq\frac{q_{[0,i]}}{q_{[0,j]}}\kappa^{1/2}\right)^{-1}.
\end{equation*}
Now we can compute the partition function in the same way as \cite{young10}. Note that as a product of $\Gamma_\pm$, moving $Q$ and $K_\pm$ past the $A_\pm$ works in the same way as simply for $\Gamma_\pm$. After moving them to the outside, they act trivially on the empty partition.
\begin{eqnarray*}
    \overrightarrow{\Zs}\left([\C^3/\mu_r],\dv{q}\right) =& \bra{\phi}\cdots \bar{A}_+(1)\bar{A}_+(1)\bar{A}_+(1)\bar{A}_-(1)\bar{A}_-(1)\bar{A}_-(1)\cdots\ket{\phi}\\
    =& \bra{\phi}\cdots QK_+A_+(1)QK_+A_+(1)A_-(1)K_-QA_-(1)K_-Q\cdots\ket{\phi}\\
    =& \bra{\phi}\cdots A_+\left(q^2(\kappa^{1/2})^2\right)A_+\left(q(\kappa^{1/2})\right)A_+(1)\cdot\\
    & A_-(1)A_-(q(\kappa^{-1/2}))A_-\left(q^2(\kappa^{-1/2})^2\right)\cdots\ket{\phi}.
\end{eqnarray*}
Now we can move the operators $A_\pm$ past each other using \eqref{rig:eq:comm_A} to get
\begin{eqnarray*}
    \overrightarrow{\Zs}\left([\C^3/\mu_r],\dv{q}\right) =& \prod_{a,b=0}^{\infty}C_r\left(q^a(\kappa^{1/2})^a,q^b(\kappa^{-1/2})^b\right)\bra{\phi}A_-(1)A_-\left(q(\kappa^{-1/2})\right)\cdot\\
    &A_-\left(q^2(\kappa^{-1/2})^2\right)\cdots A_+\left(q^2(\kappa^{1/2})^2\right)A_+\left(q(\kappa^{1/2})\right)A_+(1)\ket{\phi}\\
    =& \prod_{a,b=0}^{\infty}C_r\left(q^a(\kappa^{1/2})^a,q^b(\kappa^{-1/2})^b\right),
\end{eqnarray*} 
where we used that the operators $A_\pm$ act trivially on the empty partition $A_+(x)\ket{\phi}=\ket{\phi}$ and $\bra{\phi}A_-(x)=\bra{\phi}$, as in \cite[p. 14]{young10}. Now we can simplify this expression using standard formulas for the plethystic exponential. These formulas can be derived similar to the proof of Lemma \ref{fact:lemma:comp_formula}. They are also discussed in \cite[p. 4]{dos20}. We compute
\begin{eqnarray}
    \overrightarrow{\Zs}\left([\C^3/\mu_r],\dv{q}\right) =& \prod_{a,b=0}^{\infty} \prod_{i,j=0}^{r-1}\left(1-q^{a+b+1}\frac{q_{[0,i]}}{q_{[0,j]}}\kappa^{\frac{i-j+1}{2}}\right)^{-1}\nonumber\\
    =&\prod_{i,j=0}^{r-1}\pexp{\frac{\kappa^{1/2}q}{(1-\kappa^{1/2}q)(1-\kappa^{-1/2}q)}\frac{q_{[0,i]}}{q_{[0,j]}}}\nonumber\\
    =&\pexp{\frac{-\kappa^{1/2}}{[\kappa^{1/2}q][\kappa^{1/2}q^{-1}]}\sum_{i,j=0}^{r-1}\frac{q_{[0,i]}}{q_{[0,j]}}}.\label{rig:eq:limit_final_comp}
\end{eqnarray}
By splitting into parts with $i=j$, $i<j$, and $i>j$, we see that 
\begin{equation*}
    \sum_{i,j=0}^{r-1}\frac{q_{[0,i]}}{q_{[0,j]}}=r+\sum_{0< i \leq j < r} q_{[i,j]}+q_{[i,j]}^{-1},
\end{equation*}
so that \eqref{rig:eq:limit_final_comp} is exactly the desired limit 
\begin{equation*}
    \pexp{\frac{-\kappa^{1/2}}{[\kappa^{1/2}q][\kappa^{1/2}q^{-1}]}\left(r+\sum_{0< i \leq j < r} q_{[i,j]}+q_{[i,j]}^{-1}\right)}
\end{equation*}
of $\pexp{\Fs_r(q)+\Fs_r^\col(q_0,\dots,q_{r-1})}$. This determines the functions $h_{\dv{n}}$ in Proposition \ref{pp:prop:loc_open_locus_form} by rigidity, see Lemma \ref{pp:lemma:rigidity}, completing the proof of Theorem \ref{rig:thm:main_thm}.

\appendix
\section{Quotient DM Stacks}\label{app:quotient_stack}
We state and prove a few simple facts about quotient DM stacks for use in Lemma \ref{fact:lemma:comp_formula}. We consider the following setup. Let $U$ be a scheme, and let $i:H\hookrightarrow G$ be finite groups acting on $U$, such that
\begin{equation*}
    \begin{tikzcd}
        G\times U \ar[r,"\mu_G"] & U\\
        H\times U \ar[ur,"\mu_H"]\ar[u, hookrightarrow, "i\times\id"] & 
    \end{tikzcd}
\end{equation*}
commutes. We get the following induced morphism of DM stacks.

\begin{proposition}\label{qst:prop:morphism}
    There is a morphism of DM stacks
    \begin{equation*}
        \eta:[U/H]\to [U/G],
    \end{equation*}
    induced by the inclusion $i:H\hookrightarrow G$.
\end{proposition}
\begin{proof}
    An object of $[U/G]$ over a scheme $T$ is a principal $G$-bundle $P\to T$ and a $G$-equivariant morphism $P\to U$. The morphism $\eta$ is then concretely given by
    \begin{equation*}
        \left(T,P\to T,P\xrightarrow[H]{} U\right)\mapsto \left(T, G\times_H P, G\times_H P \xrightarrow[G]{} U\right).
    \end{equation*}
    Here $G\times_H P\coloneqq [(G\times P)/H]$ with the anti-diagonal $H$-action $h\cdot (g,p)=(gh,h^{-1}p)$, and $G\times_H P \xrightarrow[G]{} U$ is the composition $\mu_G\circ \left(\id\times(P\xrightarrow[H]{} U)\right)$, which descends to $G\times_H P$ as it is $H$-invariant.
\end{proof}

The pushforward along this morphism is just the induced representation construction.

\begin{proposition}\label{qst:prop:induced_eq_str}
    Let $n=[G:H]$ and take a full set of representatives $g_1,\dots,g_n$ of left cosets in $G/H$. Given an $H$-equivariant sheaf $\Ec$ on $U$, we have
    \begin{equation*}
        \eta_* \Ec \cong \bigoplus_{i=1}^n g_i \Ec
    \end{equation*}
    as $G$-equivariant sheaves on $U$.
\end{proposition}
\begin{proof}
    We consider the following cartesian diagram\cite[Section 2.3.7]{alper_moduli}
    \begin{equation}\label{app:eq:G_cart_diag}
        \begin{tikzcd}
            G\times U \ar[r,"\mu"]\ar[d,"p_2"] & U\ar[d, "\pi_G"]\\
            U \ar[r, "\pi_G"] & \left[U/G\right],
        \end{tikzcd}
    \end{equation}
    where $\pi_G$ is the quotient morphism $U\to[U/G]$. All morphisms are $H$-equivariant with respect to $H$ acting via $\mu_H$ on $U$ on the left side, acting anti-diagonally on $G\times U$, and trivially on $U$ and $[U/G]$ on the right side. Quotienting by $H$ then gives us the cartesian diagram
    \begin{equation}\label{app:eq:GH_cart_diag}
        \begin{tikzcd}
            G\times_H U \ar[r,"\bar{\eta}"]\ar[d,"p_2"] & U\ar[d,"\pi_G"]\\
            \left[U/H\right] \ar[r, "\eta"] & \left[U/G\right].
        \end{tikzcd}
    \end{equation}
    $\eta_*\Ec$ as a $G$-equivariant sheaf on $U$ is just $\pi_G^*\eta_*\Ec$ with its natural $G$-equivariant structure induced by the cartesian diagram \eqref{app:eq:G_cart_diag}. By the cartesian diagram \eqref{app:eq:GH_cart_diag}, this is isomorphic to $p_2^*\Ec$, with its natural $G$-equivariant structure, pushed forward along the $G$-equivariant morphism $\bar{\eta}$.
    
    As $g_1,\dots,g_n$ is a full set of representatives, any $g\in G$ has a presentation as $g=g_{i(g)}h(g)$. This gives us a natural isomorphism
    \begin{equation*}
        G\times_H U \xrightarrow{\sim} \bigsqcup_{i=1}^n g_i U,\ [g,u]\mapsto g_{i(g)} \left(h(g)\cdot u\right).
    \end{equation*}
    Again using that $g_1,\dots,g_n$ is a full set of representatives, for any $g\in G$ and $g_i$, there exists a $j(i,g)$ and an $h(i,g)\in H$, such that $gg_i = g_{j(i,g)}h(i,g)$. The above isomorphism induces the $G$-action $g\cdot g_i u = g_{j(i,g)} (h(i,g)\cdot u)$ on the right-hand side.

    Under the above isomorphism, $\bar{\eta}$ becomes just the action morphism $g_iU \to U, g_iu\mapsto g_i\cdot u$ on each component. Under this isomorphism $p_2^*\Ec$ is just a copy of $\Ec$ on every component. The natural $G$-equivariant structure of $p_2^*\Ec$ under this isomorphism is the following. For every $g\in G$ and $g_i$ the action $g: \bigsqcup_{i=1}^n g_i U \xrightarrow{\sim} \bigsqcup_{i=1}^n g_i U$ restricts to $h(i,g):g_i U \xrightarrow{\sim} g_{j(i,g)}U$. On every component $g_iU$ the isomorphism $g^*p_2^*\Ec\cong p_2^*\Ec$ is given by the isomorphism
    \begin{equation*}
        h(i,g)^*\Ec \cong \Ec,
    \end{equation*}
    coming from the $H$-equivariant structure of $\Ec$ on $U$. Pushing this forward along $\bar{\eta}$ yields exactly the desired isomorphism
    \begin{equation*}
        \bar{\eta}_*p_2^* \Ec \cong \bigoplus_{i=1}^n g_i\Ec,
    \end{equation*}
    with the $G$-equivariant structure given by the isomorphisms
    \begin{equation}\label{qst:eq:eq_str_H_class_isoms}
        h(i,g)^*\left(g_{j(i,g)}\Ec\right) \cong g_i\Ec.
    \end{equation}
    Note that the $g_i$ in $\bigoplus_{i=1}^n g_i\Ec$ are purely formal for tracking the various copies of $\Ec$.
\end{proof}

\printbibliography{}

\end{document}